\documentclass[11pt,oneside,reqno]{amsart}
\usepackage{mathpazo} % math & rm
\normalfont
\usepackage[T1]{fontenc}

\usepackage{multirow}

\usepackage[utf8]{inputenc}
\usepackage{amsmath}
\usepackage{amsfonts}
\usepackage{amssymb}
\usepackage{amsthm}
\usepackage{mathtools}%gives the arrow typeset for inclusion
\usepackage{stmaryrd}%Including this for the "double-bracket" \llbracket and \rrbracket
\usepackage{tikz}
\usepackage{tikz-cd}
\usepackage[all]{xy}
\usepackage{enumitem}%This package allows me to change the formatting of the numbers for an enumerate environment.
\usepackage[margin=1.2in]{geometry} % to set the margins

\usepackage{skak}
\usepackage{accents}

\usepackage{cjhebrew}%To use hebrew letters, because I ran out of letters
\usepackage{wasysym}%To use \gemini, because I ran out of letters

\usepackage{hyperref}

\newtheorem{thm}{Theorem}[section]
\newtheorem{theorem}[thm]{Theorem}
\newtheorem{lemma}[thm]{Lemma}
\newtheorem{conj}[thm]{Conjecture}
\newtheorem{conjecture}[thm]{Conjecture}
\newtheorem{coro}[thm]{Corollary}

\newtheorem{prop}[thm]{Proposition}
\newtheorem{proposition}[thm]{Proposition}
\newtheorem{question}[thm]{Question}
\theoremstyle{definition}
\newtheorem{defi}[thm]{Definition}
\newtheorem{definition}[thm]{Definition}
\newtheorem{remark}[thm]{Remark}
\newtheorem{example}[thm]{Example}
\newtheorem{notation}[thm]{Notation}

\newtheorem*{conjecture*}{Conjecture}
\newtheorem{construction}[thm]{Construction}

\newcommand\xqed[1]{%
  \leavevmode\unskip\penalty9999 \hbox{}\nobreak\hfill
  \quad\hbox{#1}}
\newcommand\exEnd{\xqed{$\diamondsuit$}}

\newcommand{\B}{\mathbb{B}}

\newcommand{\R}{\mathbb{R}}
\newcommand{\T}{\mathbb{T}}
\newcommand{\TT}{\T}
\newcommand{\Z}{\mathbb{Z}}
\newcommand{\ZZ}{\Z}
\newcommand{\N}{\mathbb{N}}
\newcommand{\Q}{\mathbb{Q}}

\newcommand{\calC}{\mathcal{C}}
\newcommand{\calc}{\calC}

\newcommand{\fraka}{\mathfrak{a}}

\newcommand{\frakf}{\mathfrak{f}}
\newcommand{\frakp}{\mathfrak{p}}

\newcommand{\comd}[1]{$$\xymatrix{#1}$$}
\newcommand{\inj}[0]{\ar@{^{(}->}}
\newcommand{\surj}[0]{\ar@{->>}}
\newcommand{\bij}[0]{\ar@{^{(}->>}}
\newcommand{\lbij}[0]{\ar@{_{(}->>}}
\newcommand{\linj}[0]{\ar@{_{(}->}}
\newcommand{\parr}[0]{\ar@{.>}}
\newcommand{\pinj}[0]{\ar@{^{(}.>}}
\newcommand{\psurj}[0]{\ar@{.>>}}
\newcommand{\pbij}[0]{\ar@{^{(}.>>}}
\newcommand{\lin}[0]{\ar@{-}}
\newcommand{\llin}[0]{\ar@{=}}
\newcommand{\usu}[0]{\ar@{->}}%The usual, unadorned, arrow. You can also just use \ar
\newcommand{\toup}[1]{\stackrel{#1}{\longrightarrow}}%An arrow to the right, taking an argument that is the name of the function, to be placed over the arrow

\newcommand{\dsum}{\displaystyle\sum}
\newcommand{\dprod}{\displaystyle\prod}

\newcommand{\sdrop}{\backslash}
\newcommand{\into}{\hookrightarrow}

\newcommand{\ph}{\varphi}
\newcommand{\TerryPair}[2]{\Lambda_{#2}(#1)}

\newcommand{\supp}{\operatorname{supp}}%Support of, well, anything.
\newcommand{\Frac}{\operatorname{Frac}}%Fraction field of a integral domain
\newcommand{\trop}{\operatorname{trop}}

\def\Mon{\mathcal{M}}
\def\base{S}%The base-semifield that we are working over. For now, it is S.
\newcommand{\dcap}{\displaystyle\bigcap}

\newcommand{\CR}{A} %coordinate semiring
\newcommand{\FCR}{R} %Total semiring of fractions of the coordinate semiring 

\newcommand{\angbra}[1]{\left\langle #1\right\rangle}%Angle brackets, to be used for a pairing or a module being generated by a list of things.
\newcommand{\AangX}[2]{#1\!\angbra{#2}}
\newcommand{\aangx}[2]{\AangX{#1}{#2}}

\def\J-int{{J-integral}} 
\def\Jinty{{J-integrality}} 
\def\D-int{{D-integral}} 
\def\Dinty{{D-integrality}} 
\def\Q-int{{quasi-integral}}
\def\JQ-int{{annihilator-free quasi-integral}}
\def\AQ-int{{annihilator-free quasi-integral}}
\def\SQ-int{{quasi-integral}}
\def\WQ-int{{weakly quasi-integral}}
\def\dc{{downward-closed}}
\def\n-int{{N-integral}}%For "Naively integral". I can't use \N-int because it confuses \N
\newcommand{\namedClosure}[2]{\nbar{#1}^{\mathrm{#2}}}
\newcommand{\NClosure}[1]{\namedClosure{#1}{N}}
\newcommand{\Ncl}[1]{\NClosure{#1}}
\newcommand{\ncl}[1]{\Ncl{#1}}
\newcommand{\JClosure}[1]{\namedClosure{#1}{J}}
\newcommand{\Jcl}[1]{\JClosure{#1}}
\newcommand{\jcl}[1]{\Jcl{#1}}
\newcommand{\DClosure}[1]{\namedClosure{#1}{D}}
\newcommand{\Dcl}[1]{\DClosure{#1}}
\newcommand{\dcl}[1]{\Dcl{#1}}
\newcommand{\VClosure}[1]{\namedClosure{#1}{V}}
\newcommand{\Vcl}[1]{\VClosure{#1}}
\newcommand{\vcl}[1]{\Vcl{#1}}
\newcommand{\QClosure}[1]{\namedClosure{#1}{aQ}}
\newcommand{\Qcl}[1]{\QClosure{#1}}
\newcommand{\qcl}[1]{\Qcl{#1}}
\newcommand{\strongQClosure}[1]{\namedClosure{#1}{Q}}
\newcommand{\sQcl}[1]{\strongQClosure{#1}}
\newcommand{\sqcl}[1]{\sQcl{#1}}
\newcommand{\weakQClosure}[1]{\namedClosure{#1}{wQ}}
\newcommand{\wQcl}[1]{\weakQClosure{#1}}
\newcommand{\wqcl}[1]{\wQcl{#1}}

\newcommand{\wt}[1]{\widetilde{#1}}
\newcommand{\what}[1]{\widehat{#1}}%A wide hat. Also, WHAT?

\newcommand{\Bend}{\operatorname{Bend}}%The bend congruence of an ideal.
\newcommand{\bend}{\operatorname{bend}}%The bend relations of a polynomial.
\newcommand{\nbar}[1]{\overline{#1}}%A new bar. Better than the usual \bar.
\DeclareFontFamily{U}{min}{}
\DeclareFontShape{U}{min}{m}{n}{<-> dmjhira}{}
\newcommand{\MissMon}{missing monomial}%A missing monomial for a polynomial f. May or may not be friends with Mrs. Pacman.
\newcommand{\missmon}{\MissMon}

\newcommand{\newt}{\operatorname{Newt}}

\newcommand{\susbeteq}{\subseteq}% Because Nati makes a lot of typos.

\newcommand{\dom}{\operatorname{dom}}%domain
\DeclareMathOperator{\wsgn}{\wt{sgn}}%The "weird" sign of a partial permutation.

\title{Integral Closure for (additively idempotent) semirings}
\author{Netanel Friedenberg}
\address{Department of Mathematics, Tulane University, New Orleans, LA 70118, USA}
\email{nfriedenberg@tulane.edu}

\author[K.~Mincheva]{Kalina~Mincheva}
\address{Department of Mathematics, Tulane University, New Orleans, LA 70118, USA}
\email{kmincheva@tulane.edu}
\date{}

\begin{document}

\begin{abstract}
    In commutative ring theory there are multiple equivalent definitions of integrality. These notions diverge when working with idempotent semirings. In this paper we present these different definitions of integrality {for semirings} and explore the relations between them. As a tool, we prove a Cayley-Hamilton theorem over additively idempotent semirings, which may be of broader interest. In examples, we compute integral closures of coordinate semirings in their total semiring of fractions and integral closures of sub-semirings of coordinate semirings. Such computation gives avenues to defining and understanding the normalization of tropical varieties as well as computing normalization of varieties tropically.
\end{abstract}

\maketitle

%=======================================
% Notes (done with): 62
% Notes (left): 63-66
%=======================================

%\section{Conventions}
\tableofcontents
\section{Introduction}
% \red{TODO
% \begin{enumerate}
%     \item check for color
%     \item check that the theorems in the intro have the updated statements in case updates have happened
%     \item We might consider going through and, before any lemma that is elementary, saying that the lemma is elementary and maybe ``we include the details for completeness'' or something like that.
% \end{enumerate}
% }

\subsection{Context and motivation}   

Tropical geometry associates a combinatorial object, called a tropical variety, to an algebraic variety, allowing for the use of combinatorial tools to approach classical problems of algebraic geometry. Historically, the underlying algebra of a tropical variety have been ignored and algebraic computations were done on the classical side. However, the semiring algebra retains much more information than just the tropical variety and can provide deeper insights than the combinatorial properties of tropical varieties alone. In recent years, there has been a lot of effort dedicated to developing the necessary tools of additively idempotent commutative algebra using different frameworks, among which are prime congruences \cite{JM17}, \cite{FM23}, congruences on rational function semifields \cite{Son23a}, \cite{Son23b}, tropical ideals \cite{MR18}, and tropical schemes \cite{GG13}. These approaches allow for both the exploration of properties of tropicalized spaces without tying them up to the original varieties and for augmenting existing combinatorial methods as done in \cite{AKS25}.

In this paper we focus on the integral closure of a semiring. Unsurprisingly, the tropical analogues of equivalent notions of integral elements over rings are no longer equivalent. Recall that if $B$ is a ring and $A$ a subring, then an element $b\in B$ is integral over $A$ if (1) $b$ is the root a monic polynomial with coefficients in $A$ or equivalently (2) there exists a faithful $A[b]$-module $M$ such that $M$ is finitely generated as an $A$-module.
In \cite{Tol16}, Tolliver first considers one analogue of each of (1) and (2) for additively idempotent semirings, which we refer to as \J-int and \JQ-int, respectively. In this paper we also add a number of refinements. We ask when the different notions of integrality agree and are equal to the intersection of valuation semirings. We find that, while some definitions are easier to compute with, others more often agree with the suitable intersection of valuation semirings.

Of particular interest is the case of finding the integral closure of the `coordinate semiring' of a tropical variety in its total semiring of fractions. Such computation gives avenues to defining and understanding the normalization of tropical varieties as well as computing normalization of varieties tropically. Due to the piecewise-linear nature of tropical varieties the only way one can approach normalization tropically is from this algebraic point of view. %\red{sure. You're welcome to add such a sentence.}

% \green{Our last} \orange{Another} motivation \orange{to study integral elements} 
Another motivation to study integral elements 
comes from the theory of adic spaces. One defines the adic spectrum of a pair of rings $(R, R^+)$, where the subring $R^+$ of an f-adic ring $R$ satisfies a number of conditions, one of which is that it is integrally closed in $R$. 
% \green{To describe a tropical adic space which we started in \cite{FM23}, we need to be able to compute the sub-semiring of $A$, which plays an analogous role to $R^+$ for an adic space. For us this is the intersection of all valuation semirings containing $A$, or the "valuative" integral closure of $A$.}
So, to continue the project, started in \cite{FM23}, of defining tropical adic spaces, we need to understand integral closure of sub-semirings of additively idempotent semirings.

%\blue{I think it's better here but can be par 2} 
There are other works that have considered integral elements in the context of semirings. Our starting point is unpublished work by Tolliver in \cite{Tol16}, where he introduces the notions of \J-int and \JQ-int elements. In \cite{JMR22} the authors use a similar definition to \J-int for integrality in the context of systems. More recently, Jun, Ray, and Tolliver in \cite{JRT20} present some versions of integral closure operations for ideals. They suggest in \cite[Remark 5.7]{JRT20} that their integral closure is contained in an analogue of the ring theoretic integral closure of an ideal defined in terms of valuations. The unpublished manuscript \cite{Mac18} introduces a notion for integrality for elements and order ideals of modules over semirings. However, this notion of integrality does not align with the notions presented in this paper, \cite{Tol16}, and \cite{JRT20}, and is constructed in order to force elements to be multiplicatively cancellative.

\subsection{Results}   

Since the equivalent notions of integrality for ring elements diverge over semirings our first goal is to see how they relate. 
Four notable notions are \J-int elements, which are defined in terms of a monic polynomial, \SQ-int elements, which are defined in terms of faithful submodules of the ambient semiring, \D-int elements, which are those bounded above by \J-int elements\footnote{Equivalently, the set of \D-int elements is the saturation (or downward-closure) of the set of \J-int elements.}, and the valuative integral closure, defined as an intersection of valuation semirings.
We refer the reader to Definition~\ref{def: integrals} for the rigorous definition of each. For an extension of semirings $A \subseteq R$ we denote the \J-int (resp. \D-int, resp. \SQ-int, resp. valuative) closure of $A$ in $R$ by $\Jcl{A}$ (resp. $\dcl{A}$, resp. $\sQcl{A}$, resp. $\vcl{A}$). 

%\red{(This next sentence needs to be changed. And we may want to re-arrange where in the intro this is based on the fact that Theorem~\ref{thm: intro-int-semifield} is now done later in the paper.)}
%\green{We first focus on the case when in the extension of semirings $A \subseteq R$, $R$ is a semifield. We recover a result stated in \cite{Tol16} about the integral closure of an additively idempotent semiring $A$ in a semifield. %The original result contains a small gap; see Remark~\ref{rmk: gap}. Our proof of this result uses very different techniques. The tools we develop are semiring analogues of tools used to prove that integral closures are intersections of valuation rings in the full generality of arbitrary ring extensions; see \cite{Gra82} and \cite{Man2}.}

One of the main technical tools for working with \J-int elements is a tropical version of the Cayley-Hamilton theorem. It allows us to show that in certain cases all introduced notions of integral closure agree.
\begin{theorem}[Theorem~\ref{thm:CH}]
    Let $A = (a_{ij})$ be an $n\times n$ matrix over an additively idempotent semiring $S$ and let $p(x)$ be the characteristic polynomial of $A$. Then $A$ satisfies the bend relations of $p(x)$.
\end{theorem}

%\green{We use this Cayley-Hamilton theorem to understand the relations between the different notions of integrality. Our first result in this direction provides an analogue of the equivalence of monic polynomial integrality and module integrality that is true in an arbitrary semiring.}\red{we already said the CH is used for relating notions of integrality before stating it}
%
As a consequence of this Cayley-Hamilton we obtain an analogue of the equivalence of monic polynomial integrality and module integrality that is true in an arbitrary semiring.

% \red{(It could be that we should emphasize Corollary~\ref{coro:D-intAndFaithfulA[x]Mods} more. This is because $X$ is the natural semiring analogue of the more usual faithful module condition in the ring setting, (that $x\in R$ is integral over $A\susbeteq R$ if and only if there is a faithful $A[x]$-module that is finitely generated as an $A$-module) and so this statement gives a form of equation/inequality integrality is the same as module integrality that is true in an arbitrary semiring.)}
% \blue{this is a coro of CH}
\begin{theorem}[Corollary~\ref{coro:D-intAndFaithfulA[x]Mods}]
Let $A\subseteq R$ be an extension of additively idempotent semirings. Let $$X:=\{x\in R\,:\,\text{there is a %strongly 
faithful }A[x]\text{-module that is finitely generated as an }A\text{-module}\}.$$ Then $\dcl{A}$ is the downward closure of $X$.
\end{theorem}

In the above theorem, $X$ is the natural semiring analogue of the more usual faithful module condition in the ring setting: $x\in R$ is integral over $A\susbeteq R$ if and only if there is a faithful $A[x]$-module that is finitely generated as an $A$-module.

%\pinky{We now impose certain conditions on our semirings to obtain stronger results.}
Next we focus on answering the question of when the different notions of integrality agree. Neither of the following two theorems implies the other; we give applications of each that cannot be deduced from the other one.

\begin{theorem}%[Theorem~\ref{thm:intCloSemifields}]
[Proposition~\ref{prop:RCancellativeGivesAllIntClosuresEqual}]
\label{thm: intro-int-semifield}\label{thm: intro-RCancellativeGivesAllIntClosuresEqual}
Let $A\subseteq R$ be an extension of semirings with $R$ cancellative. Then the following sets are all equal: the downward closure of $A$ in $R$, $\jcl{A}$, $\dcl{A}$, $\sqcl{A}$, %$\qcl{A}$, $\wqcl{A}$, 
and $\vcl{A}$.
\end{theorem}

\begin{theorem}[Corollary~\ref{coro:VIntIsSQInt}, Theorem~\ref{thm: A+=AJ-can gen}]\label{thm: intro-int-general}
     Let $A \subseteq R$ be an extension of additively idempotent semirings such that $R$ is cancellatively generated, $\sqrt{\Delta_R}$ cancellative and for every $x \in R$, the \dc\ $A$-module generated by $x$ is finite as an $A$-module then $$\Jcl{A}= \Dcl{A}=\sQcl{A} = \vcl{A}.$$
\end{theorem}

% \green{
% The hypotheses of the above theorem are satisfied in many cases of interest, such as the coordinate semirings of tropical toric varieties, as shown in Example~\ref{ex: allIntegral-toric} and the semiring of finitely generated submodules of a B\'{e}zout domain; see Example~\ref{ex: allIntegral-terry}.
% }

The hypotheses of the above theorem are satisfied in many cases of interest, such as when $R$ is the coordinate semiring of a tropical toric variety or when $A$ is the semiring of finitely generated ideals of a B\'{e}zout domain; see Examples~\ref{ex: allIntegral-toric} and \ref{ex: allIntegral-terry}. %\blue{we probably should put here coro \ref{coro:ModBendVIntIsQInt} with the text: In particular, this tells us that the normalization can be viewed as any of these kinds of integral closures.}
%\green{We would like to stress that neither Theorem~\ref{thm: intro-RCancellativeGivesAllIntClosuresEqual} nor Theorem~\ref{thm: intro-int-general} imply the other.}

%We now turn our focus to 
%%%%% After this next Theorem, we use the language "We then turn", which felt too similar to me.
For other examples, we focus on 
the semirings arising from geometry, namely the tropicalizations of coordinate rings of varieties.

\begin{theorem}[Corollary~\ref{coro:ModBendVIntIsQInt}]
Let $K$ be a field with valuation $K\to S$, let $\Mon$ a toric monoid, and let $I$ be an ideal of $K[\Mon]$ whose variety $Y := V(I)$ has no component contained in the toric boundary. 
Let $R$ be any localization of $\base[\Mon]/\Bend(I)$ by a set of cancellative elements, and let $A$ be any sub-semiring of $R$. Then $\vcl{A}=\sqcl{A}$.
\end{theorem}

% Note that neither of Theorems~\ref{thm: intro-int-general} and \ref{thm: intro-int-semifield} implies the other. Specifically, the third hypothesis in Theorem~\ref{thm: intro-int-general} is not automatically satisfied by a semifield.

We then turn to computing integral closure of a coordinate semiring $A$ in its total semiring of fractions $\Frac(A)$; this is the analogue of computing a normalization. We note that such $A$ is almost never cancellative, as opposed to semirings of functions - tropical polynomials carry more information than tropical polynomial functions. As such, in order to compute $\Frac(A)$, we need to know which elements of $A$ are cancellative. We investigate this in Section~\ref{sec:cancellativeElements}, and we find that once we have $\Frac(A)$, in practical situations it is relatively easy to compute the integral closure of $A$.

Using results from \cite{FM25a} we are able to give geometric examples of normalization of tropical curves; see Example~\ref{ex:nodal-cubic-CPL}, Example~\ref{ex: cuspidalCubic}, and Example~\ref{ex: nodalCubic-2}, the last of which is dependent on Conjecture~\ref{conj: cancellatives} about cancellative elements in $A$.

In these examples, the normalization of the tropicalization of a curve with unique singular point at the origin looks like the tropicalization of the normalization. We make the following conjecture: 
\begin{conjecture}
    Let $C$ be an irreducible affine algebraic curve with one singular point. Then there is an embedding of $C$ with ideal $I\subseteq K[x_1,\ldots,x_n]$ satisfying the following property. Let $A=\base[x_1,\ldots,x_n]/\Bend(\trop I)$, let $R=\Frac(A)$, and let $\vcl{A}$ be the valuative integral closure of $A$ in $R$. Then the integral closure of $B=K[x_1,\ldots,x_n]/I$ in its field of fractions is obtained by adjoining to $B$ lifts of the elements in $\vcl{A}\sdrop A$.
\end{conjecture}

Of course, one cannot begin addressing this conjecture until one can compute $\Frac(A)$, i.e., identify the cancellative elements of $A$. Construction~\ref{constr: pairsGH} tells us how to find pairs of polynomials certifying that a polynomial is not cancellative and 
%Lemma~\ref{lem: wit-to-M-wit} \red{(fix this reference)} 
Corollary~\ref{coro:SuppConditionForNotCancellative} 
gives a sufficient condition for verifying that $f$ is not cancellative in the quotient by the bend congruence of a 
%\green{tropical} \red{(The strong Maclagan-Rinc\'on lemma is stated for tropicalized ideals. I'm inclined to believe that the proof can be modified to work for tropical ideals, but I haven't checked the details. So for now the new reference is for the tropicalized case.)} 
tropicalized ideal $I$. 
% \green{
% In section~\ref{sec: witness pairs} and Appendix~\ref{app: countingCircuits} we provide an outline how to approach the normalization problem computationally (at least up to a finite degree). In Appendix~\ref{app: countingCircuits}, for a specific curve we enumerate the polynomials whose bend relations need to be avoided by the above pairs. 
% }
In 
% Appendix~\ref{app: countingCircuits} \red{(Consider saying "Appendices~\ref{app: countingCircuits} and \ref{app:ImageOfTrinomial}")} \orange{sure, either is fine} 
Appendices~\ref{app: countingCircuits} and \ref{app:ImageOfTrinomial} 
we outline a method that can be used to computationally check that polynomials of small support are not cancellative modulo the bend congruence of a tropicalized ideal. We illustrate part of this computation for the case of the tropicalization of a principal ideal generated by a general trinomial.

\addtocontents{toc}{\protect\setcounter{tocdepth}{-1}}
\section*{Acknowledgments}
The authors extend their sincere gratitude to Joachim Gr\"ater for clarifying some of the proofs of his results on valuation theory for commutative rings. K.M. acknowledges the support of the Simons Foundation MPS-TSM-00008148.
\addtocontents{toc}{\protect\setcounter{tocdepth}{2}}

\section{Preliminaries}\label{prelims}

A \textit{semiring} is a set $R$ with two binary operations (addition $+$ and multiplication $\cdot$ ) satisfying the same axioms as rings, except the existence of additive inverses, as well as $a\cdot0=0=0\cdot a$ for all $a\in R$. In this paper, a semiring is always assumed to be commutative. A semiring $(R,+,\cdot)$ is \emph{semifield} if $(R\backslash\{0_R\},\cdot)$ is a group. A semiring $R$ is called \emph{additively idempotent} if for all $a\in R$ we have that $a+a =a$.

\textbf{Key assumption:} All semirings will be assumed to be additively idempotent.

Any additively idempotent semiring comes with a partial order defined by $a\leq b\iff a+b=b$. Both addition and multiplication preserve this partial order. One consequence of this is the following well-known fact.

\begin{lemma}\label{lem:zero-sum-free}
    Let $A$ be an additively idempotent semiring, then it is zero-sum free, i.e., if $a, b \in A$ with $a+b = 0$, then $a = b = 0$.
\end{lemma}

\begin{proof}
    If $a+b = 0$, then $a+a+b = a+0$ and so $a+b = a$ implying that $b \leq a$ and by symmetry $a \leq b$. So $a=b$ and $a+ b = a+a = a = 0$.
\end{proof}

We will denote by $\mathbb{B}$ the semifield with two elements $\{1,0\}$, where $1$ is the multiplicative identity, $0$ is the additive identity and $1+1 = 1$. 
The {\it tropical semifield}, denoted $\mathbb{T}$, is the set $\mathbb{R}  \cup \{-\infty\} $ with the $+$ operation to be the maximum and the $\cdot$ operation to be the usual addition, with $-\infty = 0_\mathbb{T}$. 

In order to distinguish when we are considering a real number as being in $\R$ or being in $\T$ we introduce some notation. For a real number $a$, we let $t^a$ denote the corresponding element of $\T$. In the same vein, given $\mathfrak{a}\in\T$, we write $\log(\mathfrak{a})$ for the corresponding element of $\R\cup\{-\infty\}$. This notation is motivated as follows.
Given a non-archimedean valuation $\nu: K \rightarrow \R\cup\{-\infty\}$ on a field and $\lambda \in \R$ with $\lambda>1$, we get a non-archimedean absolute value $|\cdot|_\nu:K\rightarrow[0, \infty)$ by setting $|x|_{\nu}=\lambda^{\nu(x)}$. Since $\T$ is isomorphic to the semifield $\big([0,\infty),\max,\cdot_{\R}\big)$, we use a notation for the correspondence between elements of $\R\cup\{-\infty\}$ and elements of $\T$ that is analogous to the notation for the correspondence between $\nu(x)$ and $|x|_{\nu}$. This notation is also convenient as we get many familiar identities such as $\log(1_{\T})=0_{\R}$ and $t^a t^b=t^{a+_{\R}b}$.

% %\newcommand{\downclosure}[1]{\overline{#1}^{\mathrm{down}}}
% \newcommand{\downclosure}[1]{\overline{#1}^{(\leq)}}
% \newcommand{\downcl}[1]{\downclosure{#1}}

Ideals and modules of semirings are defined analogously to their ring counterparts. 
If a subset $I$ of a semiring or module satisfies the property that $x\leq y\in I\implies x\in I$, we say that $I$ is \emph{\dc}.  
% \pinky{The \emph{downward closure} of a subset $J$ of a semiring or module is $\downclosure{J}:=\{x\,:\, x\leq y\text{ for some } y\in J\}$.\label{down-closure-def} } 
% \red{Other options for the notation would be $\nbar{J}^{\leq}$ or $\nbar{J}^{(\leq)}$} 
%
%%%%% Now that we removed this next sentence, it feels like this paragraph needs another sentence.
We will be interested in downward closures of submodules, ideals, and sub-semirings.

\begin{defi}
Let $A$ be a semiring and let $a \in A\sdrop\{0\}$. We say that $a$ is a \emph{cancellative element} if for all $b, c \in A$, whenever $ab=ac$ then $b=c$. If all elements of $A\sdrop\{0\}$ are cancellative, then we say that $A$ is a \textit{cancellative semiring}.
\end{defi}

Our next lemma shows that cancellative elements can be cancelled from inequalities.

\begin{lemma}\label{lemma:cancelIneqs}
Suppose that $A$ is a semiring and $s,x,y\in A$. If $s$ is cancellative then $sx\leq sy\implies x\leq y$.
\end{lemma}
\begin{proof}
We have $s(x+y)=sx+sy=sy$. Cancelling $s$ we get $x+y=y$, i.e., $x\leq y$.
\end{proof}

%\red{transition sentence : we are going to congruences}
We now define one of the main components of our theory, congruences, and list some of their key properties.
%\pinky{We now define one of the main ingredients of our theory, congruences, and list some of their key properties.}

\begin{definition}\label{def: prime_domain}
A \textit{congruence} on a semiring $A$ is an equivalence relation on $A$ that respects the operations of $A$. The \emph{trivial congruence} on $A$ is the diagonal $\Delta\subseteq A\times A$, for which $A/\Delta\cong A$. We call a congruence $\mathcal{C}$ on $A$ \emph{cancellative} if $A/\mathcal{C}$ is a cancellative semiring. We call a proper congruence $P$ of an idempotent semiring $A$  {\it prime} if $A/P$ is totally ordered and cancellative (cf.\ Definition 2.3 and Proposition 2.10 in \cite{JM17}).
\end{definition}

The following lemma will be useful later.
\begin{lemma}\label{lemma:MaclaganRincon_weak_congruence_structure_lemma}\label{lemma:MaclaganRincon_weak_lemma}
Let $R$ be an additively idempotent semiring and let $X$ be a subset of $R$ that additively generates $R$. Let $C$ be any set of pairs of elements of $R$. Then the congruence $\calc=\angbra{C}$ generated by $C$ is the union of $\Delta_R$ and the transitive closure of the set of pairs of the form $(m\cdot f + h,\, m\cdot g + h)$ or $(m\cdot g + h,\, m\cdot f + h)$ where $(f,g)\in C$, $m\in X$, and $h\in R$.
\end{lemma}\begin{proof}
This is proven in \cite[Lemma 2.4]{MR14}. The result is stated there for $R=\T[x_1,\ldots,x_n]$ and $X$ the set of 
%monomials 
terms 
of $\T[x_1,\ldots,x_n]$, but the proof works just as well in this generality.
\end{proof}

Let $\mathcal{C}$ be a congruence on a semiring $A$. The \emph{ideal-kernel} of $\mathcal{C}$ is the set of elements $a \in A$ such that $(a,0_A) \in \mathcal{C}$. It is easy to see that the ideal-kernel of $\mathcal{C}$ is an ideal of $A$. If the ideal-kernel of $\mathcal{C}$ is the zero ideal, then we say that $\mathcal{C}$ has \emph{trivial ideal-kernel}.

Let $\ph:A\to A'$ be a homomorphism of semirings. The \emph{congruence-kernel} of $\ph$, denoted $\ker(\ph)$, is the pullback $\ph^*(\Delta)$ of the trivial congruence on $A'$. If $A'$ is totally ordered and cancellative, then $\ker(\ph)$ is prime. The \emph{ideal-kernel} of $\ph$ is the set of those $a\in A$ such that $\ph(a)=0_{A'}$. Clearly, the ideal-kernel of $\ph$ is the same as the ideal-kernel of $\ker(\ph)$.

\begin{lemma}\label{lemma:ZeroDivsModCongWithNoIdealKernel}
Let $R$ be an additively idempotent semiring with no zero divisors and let $\calc$ be a congruence on $R$ with trivial ideal-kernel. Then $R/\calc$ has no zero divisors.
\end{lemma}\begin{proof}
Let $\pi:R\to R/\calc$ be the quotient map. Suppose that $\pi(x)\pi(y)=0$ in $R/\calc$. Then, because $\calc$ has trivial ideal-kernel, $xy=0$ in $R$. Since $R$ has no zero divisors, $x=0$ or $y=0$. 
\end{proof}

\begin{lemma}\label{lemma:ZeroDivsModCongGendByTotallyNonzeroPairs}
Let $R$ be an additively idempotent semiring with no zero divisors and let $C$ be a set of pairs such that, for any $(f,g)\in C$, either $f$ and $g$ are both nonzero or they are both zero. Let $\calc=\angbra{C}$ be the congruence generated by $C$. 
Then $\calc$ has trivial ideal-kernel and so $R/\calc$ has no zero divisors.
\end{lemma}
\begin{proof}
Suppose that $\calc$ has nontrivial ideal-kernel, i.e., there is a nonzero $F\in R$ with $(0,F)\in\calc$. Let $X=R\sdrop\{0\}$; certainly $X$ additively generates $R$. 

Note that, by replacing $C$ with $C\cup\{(g,f)\,:\,(f,g)\in C\}$, we may assume without loss of generality that $C$ is symmetric.

We claim that, if $(f,g)\in C$, $m\in X$, $c\in R$ and $mf+c=0$ then $mg+c=0$. Lemma~\ref{lem:zero-sum-free} applied to $mf+c=0$ gives us that $mf=0$ and $c=0$. Since $R$ has no zero divisors, $m=0$ or $f=0$. If $m=0$ then $mg+c=0\cdot g+0=0$. If $f=0$ then $g=0$ because $(f,g)\in C$, and so $mg+c=m\cdot0+0=0$. This proves the claim.

Since $(0,F)\notin\Delta_R$, Lemma~\ref{lemma:MaclaganRincon_weak_lemma} tells us that there is a sequence of pairs $(m_1\cdot f_1+c_1, m_1\cdot g_1+c_1), (m_2\cdot f_2+c_2, m_2\cdot g_2+c_2),\ldots,(m_n\cdot f_n+c_n, m_n\cdot g_n+c_n)$ where $m_i\in X$ for all $i$, $(f_i,g_i)\in C$ for all $i$, $c_i\in R$ for all $i$,  $m_1\cdot f_1+c_1=0$, $m_i\cdot g_i+c_i=m_{i+1}\cdot f_{i+1}+c_{i+1}$ for $i=1,\ldots,n-1$, and $m_n\cdot g_n + c_n=F$.

We prove by induction on $i$ that $(m_if_i+c_i,m_ig_i+c_i)=(0,0)$. For the base case we have $m_1f_1+c_1=0$ by assumption and then $m_1g_1+c_1=0$ by the claim proven above. For the inductive step, assume $(m_if_i+c_i,m_ig_i+c_i)=(0,0)$; we want to show $(m_{i+1}f_{i+1}+c_{i+1},m_{i+1}g_{i+1}+c_{i+1})=(0,0)$. By assumption we have $m_{i+1}f_{i+1}+c_{i+1}=m_ig_i+c_i=0$ and then the claim proven above tells us that $m_{i+1}g_{i+1}+c_{i+1}=0$, completing the induction.

Thus $F=m_ng_n+c_n=0$, contradicting $F\neq0$.
\end{proof}

\begin{notation}\label{notation:leqP}
Let $R$ be semiring and let $P$ be a prime congruence on $R$. For two elements $r_1, r_2 \in R$ we say that $r_1 \leq_P r_2$ (resp. $r_1 <_P r_2$, resp. $r_1 \equiv_P r_2$) whenever $\overline{r_1} \leq \overline{r_2}$ (resp. $\overline{r_1} <  \overline{r_2}$, resp. $\overline{r_1} = \overline{r_2}$) in $R/P$. Here by $\overline{r}$ we mean the image of $r$ in $R/P$.
\end{notation}

Note that $P$ is determined by the relation $\leq_P$ on $R$. Moreover, if $\mathcal{G}$ is a set of additive generators for $R$, then $P$ is determined by the restriction of $\leq_P$ to $\mathcal{G}$.

\begin{defi}\label{def: semiring-of-fractions}
Let $R$ be a semiring and $A$ the set of all cancellative elements in $R$. The usual operations on fractions make $(A \times R)/\sim$ into a semiring, called the \textit{total semiring of fractions} and denoted by $\Frac(R)$. There is a natural injective homomorphism $\eta:R\to\Frac(R)$ which we consider as an inclusion map.
\end{defi}

This semiring is well-defined and satisfies the usual universal property, namely: if $R$ is a semiring and $\ph:R\to B$ is a morphism of semirings, then $\ph$ factors through $\eta:R\to\Frac(R)$ if and only if $\ph$ maps every cancellative element of $R$ to an invertible element of $B$. Moreover, in this case it factors uniquely. For a detailed study of the semiring of fractions and its properties we refer the reader to \cite{Gol92} or \cite{HW98}. Note that $\Frac(R)$ is a semifield if and only if $R$ is cancellative.

\begin{defi}\label{def:ResidueSemifield}
Let $P$ be a prime congruence on a semiring $A$. The {\it residue semifield of $A$ at $P$}, denoted $\kappa(P)$, is the total semiring of fractions of $A/P$. We denote the canonical homomorphism $A\to A/P\to\kappa(P)$ by $a\mapsto|a|_P$.
\end{defi} 

A large class of semirings that we will be interested in are the monoid $\base$-algebras $\base[\Mon]$, where $\base$ is a totally ordered semifield and $\Mon$ a monoid. 
%Elements of $\base[\Mon]$ are finite sums of expressions of the form $a\chi^u$ with $a\in\base$ and $u\in\Mon$, which are called \emph{terms}, and we call elements of $\base[\Mon]$ \emph{polynomials}. 
Elements of $\base[\Mon]$, called \emph{polynomials}, are finite sums of expressions of the form $a\chi^u$ with $a\in\base$ and $u\in\Mon$. 
We call these $a\chi^u$ \emph{terms} and use \emph{monomials} to refer to the $\chi^u$s. 
For any $f\in\base[\Mon]$, the \emph{support of $f$}, denoted $\supp(f)$ is the set of those monomials that occur in $f$ with nonzero coefficient.
%The semirings $\base[\Mon]$ arise as the tropicalizations of the coordinate rings of toric varieties. In those case, we say that $\Mon$ is a \emph{toric} monoid.
When $\Mon$ is the type of monoid that arises in the study of normal toric varieties, see \cite{Ful93} and \cite{CLS}, we will call $\Mon$ a \emph{toric monoid} and call $\base[\Mon]$ a \emph{toric semiring}.
%We will often assume that $\Mon$ is a toric monoid, that is .....

When $\base\subseteq \TT$ we describe the prime congruences on $\base[\ZZ^n]$ in terms of their \emph{defining matrices}. Every prime congruence on $\base[\Z^n]$ has a defining matrix; see \cite{JM17}, where it is shown that the matrix can be taken to be particularly nice. 

Let $R = \base[\ZZ^n]$ where $\base$ is a sub-semifield of $\TT$. A $k\times (n+1)$ real valued matrix $\mathcal{C}$ such that the first column of $\mathcal{C}$ is lexicographically greater than or equal to the zero vector gives a prime congruence on $R$ as follows:
For any monomial $m=t^a\chi^u\in\base[\Z^n]$ we let $\Phi(m)=\mathcal{C}\begin{pmatrix}a \\ u\end{pmatrix}\in\R^k$, where we view $u\in\Z^n$ as a column vector. We call $\begin{pmatrix}a \\ u\end{pmatrix}$ the \emph{exponent vector} of the monomial $m$. For any nonzero $f\in\base[\Z^n]$, write $f$ as a sum of monomials $m_1,\ldots,m_r$ and set $\Phi(f)=\max_{1\leq i\leq r}\Phi(m_i)$, where the maximum is taken with respect to the lexicographic order. Finally, set $\Phi(0)=-\infty$. We specify the prime congruence $P$ by saying that $f$ and $g$ are equal modulo $P$ if $\Phi(f)=\Phi(g)$. In this case we say that $\mathcal{C}$ is a defining matrix for $P$. The first column of $\mathcal{C}$ is called the column corresponding to the coefficient or the column corresponding to $\base$. For $1\leq i\leq n$, we say that the $(i+1)^{\text{st}}$ column of $\mathcal{C}$ is the column corresponding to $x_i$.

\begin{example} 
    Let $\base [\Z^n] = \T[x_1, x_2]$. The lexicographic order on this semiring with $x_1 \gg x_2$ is given by the matrix $$\begin{pmatrix} 0& 1 & 0\\ 0& 0 & 1 \end{pmatrix}.$$
    Let $\base [\N^n] = \T[x_1, x_2]$. The congruence $\mathcal{C} = \left< (x_1, 1), (x_2, 0)\right>$ is given by the matrix $$\begin{pmatrix} 1& 0 & -\infty\end{pmatrix}.$$ 
    When a congruence has kernel, we note that it has the same quotient as a congruence without kernel on a smaller semiring. In this example, we can consider $\mathcal{C}'=\left< (x_1, 1)\right>$, and so $\T[x_1, x_2]/\mathcal{C} = \T[x_1]/\mathcal{C}'$. Geometrically, the rows of a matrix with infinities in the same columns are in the same stratum of a toric variety.  
    \exEnd
\end{example}

Let $P$ be a prime on $\base [\Mon]$. If $\base [\Mon]/P \cong \base$ we say that $P$
is a \emph{a geometric prime}. The defining matrix of such $P$ has always a single row and can be chosen so that the (1,1) entry is 1. 

\begin{definition}[adapted from Definition 5.1.1 in \cite{GG13}]\label{def:bend_rel}
Let $A$ be an idempotent semiring, let $\Mon$ be a toric monoid, and let $f \in A[\Mon]$. 
For $i$ in the support of $f$ 
we write $f_{\hat \imath}$ for the result of deleting the $i$ term from $f$. Then the \emph{bend relations of $f$} is the set of pairs 
$$\bend(f) = \{(f , f_{\hat \imath})\}_{i\in \supp(f)}.$$
In particular, $\bend(0)=\emptyset$ because $\supp(0)=\emptyset$.
The \emph{bend congruence} is the congruence generated by the bend relations of a set $I$ of polynomials in $A[\Mon]$, denoted $\Bend(I)$. 
\end{definition}

\begin{coro}\label{coro:BendAndNoZeroDivs}
Let $S$ be a 
%totally ordered semifield, 
semiring with no zero divisors,
let $\Mon$ be a toric monoid, and let $I\subseteq S[\Mon]$ be an ideal that does not contain any monomial. Then $\Bend(I)$ has trivial ideal-kernel and $S[\Mon]/\Bend(I)$ has no zero divisors.
\end{coro}\begin{proof}
This follows immediately from Lemma~\ref{lemma:ZeroDivsModCongGendByTotallyNonzeroPairs} because the only way for a bend relation of $f$ to contain a zero is if $f$ is a monomial.
\end{proof}

\begin{definition}\label{def: rad}
    Let $\mathcal{C}$ be a congruence, the \textit{radical of $\mathcal{C}$}, denoted $\sqrt{\mathcal{C}}$, is the intersection of all prime congruences containing $\mathcal{C}$. If $\mathcal{C} = \sqrt{\mathcal{C}}$, we say that that $\mathcal{C}$ is a \emph{radical} congruence. 
\end{definition}

\begin{thm}[See Theorem~3.9 and Proposition~5.2(i) of \cite{JM17}]\label{thm: radGP}
For any congruence $\mathcal{C}$ of an additively idempotent semiring $A$, we have that 
$$\sqrt{\mathcal{C}} = \{(x, y)\,|\, \text{there are }k\in\Z_{\geq0}\text{ and }c\in A\text{ such that }((x+y)^k+c)(x,y)\in\mathcal{C}\}.$$
\end{thm}

\begin{lemma}\label{lem:noZD-diagNoIK}
    Let $A$ be a semiring with no zero divisors. Then the ideal-kernel of $\sqrt{\Delta}$ is trivial.% \orange{and $A/\sqrt{\Delta}$ has no zero divisors}.
\end{lemma}

\begin{proof}
    Assume that the ideal-kernel of $\sqrt{\Delta}$ is non-trivial, that is, there exists an element $a\in A \setminus \{0\}$ such that $(a, 0) \in \sqrt{\Delta}$. By Theorem~\ref{thm: radGP}
    %\cite[Proposition 5.2]{JM17} \red{(Ideally this will become a reference to a single statement in Section 2, rather than relying on both \cite[Proposition 5.2(i)]{JM17} and the characterization of $\sqrt{\calc}$ in terms of generalized powers.)} 
    for the pair $(a,0)$ there exist $i \in \mathbb{Z}_{\geq 0}$ and $h \in A$ such that $(a^i + h)(a,0)$ is in $\Delta$, or equivalently, $(a^i+h)a = 0$. Since $a \neq 0$ by assumption and $A$ has no zero divisors, we must have $a^i + h = 0$ which in turn implies that 
    %$a^i = h = 0$ 
    $a^i=0$
    by Lemma~\ref{lem:zero-sum-free}, making $a$ a zero divisor. Note that $i$ has to be bigger than 1, because $a$ is not 0 by assumption.
    %
    %\orange{The second claim now follows from Lemma~\ref{lemma:ZeroDivsModCongWithNoIdealKernel}.}
\end{proof}

\begin{definition}\label{def:refuced}
    We call a semiring $A$ \emph{reduced} if $\sqrt{\Delta_A}=\Delta_A$.
\end{definition}

\begin{proposition}\label{prop: canIsRed}
    If $A$ is cancellative then $A$ is reduced.
\end{proposition}
\begin{proof}
% %    By Theorem~\ref{thm: radGP} $$\sqrt{\Delta} = \{(x, y)\,|\, \text{there are }k\in\Z_{\geq0}\text{ and }c\in A\text{ such that }((x+y)^k+c)(x,y)\in\Delta\}.$$ 
% %
% \green{
% Let $(x,y) \in \sqrt{\Delta}$ and assume for contradiction that $(x,y) \not\in {\Delta}$. Since $(x,y) \in \sqrt{\Delta}$, by Theorem~\ref{thm: radGP} $r(x,y) \in \Delta$, where $r = (x+y)^k+c$. Since $A$ is cancellative and $(x,y) \not\in {\Delta}$ we must have that $r = 0$. By Lemma~\ref{lem:zero-sum-free} this means that $(x+y)^k=0$, so $x+y=0$ and thus $x=y=0$ implying that $(x,y)\in \Delta$ which is a contradiction.
% }
This is a restatement of \cite[Proposition 3.10]{JM17}.
%\blue{indeed, that's why it is in the prelims, but we can remove the proof.}
\end{proof}

\begin{proposition}[Proposition 4.9 in \cite{Tol16}]\label{prop:Jef4.9}
    Let $A$ be a semiring, every surjective homomorphism from $A$ to a reduced semiring factors uniquely through $A/\sqrt{\Delta}$.
\end{proposition}

\begin{coro}\label{coro:factor-quotient-map}
    Let $A$ be a semiring. Then every homomorphism from $A$ to a cancellative semiring factors uniquely through $A/\sqrt{\Delta}$.
\end{coro}

\begin{proof}
    This follows directly from Proposition~\ref{prop: canIsRed} and Proposition~\ref{prop:Jef4.9}.
 %   This follows directly from Proposition~\ref{prop:Jef4.9} and that cancellative semirings are reduced \orange{by Proposition~\ref{prop: canIsRed}}%\red{(Nati says: we should give a reference for this.)}
\end{proof}

\noindent The following statement is shown in \cite[Remark 5.14]{FM25a} as a consequence of \cite[Theorem 5.13]{FM25a}.
\begin{theorem}\label{thm: trop-coordinate-semiring-is-cancellative}
Let $K$ be a field with valuation $K\to S$ and $\Mon$ a toric monoid. Let $I$ be an ideal of $K[\Mon]$ whose variety $Y := V(I)$ has no component contained in the toric boundary. Then every minimal prime congruence of $\base[\Mon]/\Bend(\trop I)$ has trivial ideal-kernel. In particular, $\sqrt{\Bend(\trop I)}$ is cancellative.
\end{theorem}

\begin{coro}\label{coro:trop-coordinate-semiring-is-cancellative - localized}
Let $K$ be a field with valuation $K\to S$ and $\Mon$ a toric monoid. 
Let $I$ be an ideal of $K[\Mon]$ whose variety $Y := V(I)$ has no component contained in the toric boundary. 
%Then every minimal prime congruence of $\base[\Mon]/\Bend(\trop I)$ has trivial ideal-kernel. 
Let $R$ be any localization\footnote{Localization is a standard construction for semirings; see \cite[Section 11]{Gol92} for details.} of $\base[\Mon]/\Bend(\trop I)$ by a set of cancellative elements. Then $\sqrt{\Delta_R}$ is cancellative.
\end{coro}\begin{proof}
Let $R_0:=\base[\Mon]/\Bend(\trop I)$. Since $R$ is a localization of $R_0$, every prime congruence on $R$ is the extension of a prime congruence on $R_0$. Moreover, because $R$ is a localization of $R_0$ by cancellative elements, every prime congruence on $R_0$ with trivial ideal-kernel extends to a prime congruence on $R$ and the extension also has trivial ideal-kernel. Thus, Theorem~\ref{thm: trop-coordinate-semiring-is-cancellative} implies that every minimal prime congruence of $R$ has trivial ideal-kernel. So, by \cite[Proposition 5.18]{FM25a}, $\sqrt{\Delta_R}$ is cancellative.
\end{proof}

%\subsection{Tropical ideals}
We finish the section with the definition of a tropical ideal. Tropical ideals have many combinatorial properties which can be described in the language of valuated matroids, so we define those first. For the basics of matroids and valuated matroids we refer the reader to \cite{Ox} and \cite{DW}. %\orange{see also the beginning of Appendix~\ref{app: countingCircuits}.} \red{Nati says: I've now basically removed this from Appendix~\ref{app: countingCircuits}, so we may want to remove this line..}

\begin{definition}[Adapted from Definition 1.1 of \cite{MR18}]\label{def: trop_ideals}
Let $\base$ be a sub-semifield of $\T$, let $\Mon$ be any monoid, and let $I\subseteq\base[\Mon]$ be an ideal. For any $D\subseteq \Mon$ we let $I_D$ denote the set of those $f\in I$ whose support is contained in $D$.

We say that $I$ is a \emph{tropical ideal} if, for every finite set $D\subset\Mon$, $I_{D}$ is a tropical linear space, i.e., $I_{D}$ is the set of vectors of a valuated matroid. To be very concrete, $I$ is a tropical ideal if it satisfies the following monomial elimination axiom:

For any $f, g \in I$ and any monomial ${\pmb x}^{\pmb u}$ which appears in both $f$ and $g$ with the same nonzero coefficient, there exists a polynomial $h \in I$ such that $\pmb u$ is not in the support of $h$ and, for any other $\pmb v\in \Mon$, the coefficient $c_{\pmb v}(h)$ of ${\pmb x}^{\pmb v}$ in $h$ satisfies $c_{\pmb v}(h) \leq \max (c_{\pmb v}(f), c_{\pmb v}(g))$ with equality holding whenever $c_{\pmb v}(f) \neq c_{\pmb v}(g)$. Note that, in this case, we have $\supp(h)\subseteq \Big(\supp(f)\cup\supp(g)\Big)\sdrop\{\pmb u\}$, so if $f,g\in I_D$ then $h\in I_D$.
\end{definition}

\begin{remark}
    One can think of a homogeneous tropical ideal as a ``tower'' of valuated matroids that determine its various homogeneous parts as described in \cite[Definition 2.5]{MR18}.
\end{remark}

\begin{example}
    Let $K$ be a field and $\Mon$ a toric monoid. Given any ideal $I$ of $K[\Mon]$, the ideal $\trop(I) \subseteq \T[x_1, \dots x_n]$ is a tropical ideal.
\end{example}

%\subsection{Different notions of integrality}

%\section{Elementary properties of integral elements.}
\section{Definitions of integrality and elementary properties}

%\red{(Nati asks: Should we have an introductory paragraph to this section reminding the reader of several equivalent definitions of integrality in the ring case and saying that we will adapt these?)}

% In commutative (ring) algebra, the integral closure of a integral domain $R$ with field of fractions $K$ is the intersection of the valuation domains (of $K$) which contain $R$. Motivated by that we consider a semiring extension $A\subseteq R$ the \textit{valuation semirings} defined as $$R_P = \{x \in R : |x|_P \leq 1_{\kappa(P)}\},$$ for each prime congruence $P$. We say $\vcl{A}$ is \textit{the valuative integral closure of $A$}, where $\vcl{A} = \bigcap R_P$ and the intersection is over all $P$ for which $R_P \supseteq A$.  
% We will often be interested in the case $R = \Frac(A)$. 
% \red{(Nati: This paragraph needs some rewording. I should come back to this.)}

In commutative (ring) algebra, one way to express the integral closure of a subring $A$ of a ring $R$ is as the intersection over (equivalence classes) of valuations $v$ on $R$ of the valuation sub-rings $R_v\subseteq R$, where we only include in the intersection those $v$ for which $A\subseteq R_v$. Given that the semiring analogue of an equivalence class of valuations is a prime congruence, for each prime congruence $P$ on a semiring $R$ we define the corresponding \emph{valuation sub-semiring} of $R$ as 
$$R_P := \{x \in R \,:\, |x|_P \leq 1_{\kappa(P)}\}.$$
We then define the \emph{valuative integral closure of $A\subseteq R$} to be $\vcl{A}:=\bigcap R_P$ where the intersection is over all prime congruences $P$ on $R$ for which $A\subseteq R_P$. As in the context of rings, one case that is of particular interest is when $R=\Frac(A)$.

To define other versions of integrality we need a few more definitions.

\begin{definition}
    We call an $R$-module $M$ \emph{annihilator-free faithful}
    \footnote{Tolliver calls 
    %these 
    such modules 
    faithful in \cite{Tol16}.} if $rM \neq 0$ for all $r \in R$. We will say that an $R$-module $M$ is \emph{faithful} if the structure homomorphism of not-necessarily-commutative semirings $R \to \text{End}(M)$ is one-to-one. That is, if $a, b\in R$ are such that $ax = bx$ for all $x\in M$, then $a = b$.
\end{definition}
%\red{specify what we mean by a \emph{finitely generated module} over $A$ and what we mean by a \emph{finitely generated DC-module} over $A$. HERE}

\begin{definition}
    Let $\R$ be an extension of additively idempotent semirings and let $M$ be an $R$-module. We say that $M$ is \emph{finitely generated as an $A$-module} if there exists an $n\in \N$ and a surjective module homomorphism $A^n \to M$.     We say that $M$ is \emph{finitely generated as a \dc\ $A$-module} if there exists a finitely generated $A$-module $L \subseteq M$ such that $M$ is the downwards-closure of $L$.
\end{definition}

\begin{notation}
    For a semiring $R$ and $A$ a sub-semiring of $R$, we will denote by $\aangx{A}{x}$ the smallest \dc{} sub-semiring of $R$ that contains both $A$ and some element $x \in R$. 
\end{notation}

We now present a few different definitions of integral elements. When we work with semirings, we no longer have the well-known equivalent definitions of integral elements.

\begin{definition}\label{def: integrals} 
Let $R$ be an additvely idempotent semiring, $A$ be a sub-semiring of $R$, and $x \in R$. We say that 
\begin{itemize}
    \item $x$ is \n-int over $A$ if, for some positive integer $n$ and $a_0,a_1,\ldots,a_{n-1}\in A$, we have that $x^n=\sum_{i=0}^{n-1}a_ix^i$.
    \item $x$ is \J-int over $A$ if, for some positive integer $n$ and $a_0, a_1, ... , a_{n-1}\in A$, we have that $x^n \leq \sum_{i=0}^{n-1}a_ix^i.$
    \item $x$ is \D-int over $A$ if there 
    %exists an element $y \in R$ 
    is a $y\in R$
    such that $y$ is \J-int over $A$ and $x \leq y$.
    \item $x$ is \SQ-int (resp.\ \JQ-int, resp.\ \WQ-int) over $A$ if there exists an $\aangx{A}{x}$-module $M \subseteq R$ which is finite as %an $A$-module 
    a downward-closed $A$-module and faithful (resp.\ annihilator-free faithful, resp.\ nonzero) as an $\aangx{A}{x}$-module.% \red{(Nati: we may want to consider adding a footnote here saying that what we call \JQ-int is what Tolliver calls quasi-integral, if for no other reason than that we don't really use this notion in this paper, and it is here mostly to point out how our stuff relates to his.)} \blue{There is already a footnote for \J-int. Here is one for \Q-int: }
    \footnote{Our notion of \JQ-int is what Tolliver \cite{Tol16} calls quasi-integral.}
\end{itemize}
We denote the set of all %\J-int (resp.\ \D-int, \Q-int, strongly \Q-int, weakly \Q-int) elements over $A$ by $\Jcl{A}$ (resp. $\Dcl{A}$, $\Qcl{A}$, $\sQcl{A}$, $\wQcl{A}$).
\n-int (resp.\ \J-int, \D-int, \JQ-int, \SQ-int, \WQ-int) elements over $A$ by $\Ncl{A}$ (resp.\ $\Jcl{A}$, $\Dcl{A}$, $\Qcl{A}$, $\sQcl{A}$, $\wQcl{A}$).
\end{definition}

Over the course of this paper, we will be interested in finding the relationships between these various ``integral closures'' as well as whether they are semirings and whether they give closure operations. The proofs in this section are elementary but quite useful and sometimes surprising. 

\begin{prop}\label{prop:IntegralAsAlgebrasFinite}
Let $A\subseteq R$ be an extension of semirings. For any $x\in R$, 
\begin{enumerate}
\item\label{propItem:N-integralAsAlgebraFinite} $x$ is \n-int over $A$ if and only if $A[x]$ is finitely generated as an $A$-module and
\item\label{propItem:J-integralAsAlgebraFinite} $x$ is \J-int over $A$ if and only if $\aangx{A}{x}$ is finitely generated as a \dc\ $A$-module.
\end{enumerate}
\end{prop}

\begin{proof}

(1) Let $x$ be \n-int, that is, $x^n=\sum_{i=0}^{n-1}a_ix^i$ for a positive integer $n$ and $a_0, a_1, \ldots , a_{n-1}$ elements of $A$. Then $1,x,x^2,\ldots,x^{n-1}$ generate $A[x]$ as an $A$-module. 

Now assume that $f_1,\ldots,f_m$ generate $A[x]$ as an $A$-module. Write each $f_i$ as a polynomial in $x$. Denote the degree of each $f_i$ by $n_i$. Let $\displaystyle n = \max_{1\leq i\leq m}\{n_i\}+1$. Note that $1, x, x^2, \ldots x^{n-1}$ generate $A[x]$, 
%in particular there exist a positive integer $n$ and $a_0, a_1, \ldots , a_{n-1}$ elements of $A$ such that $x^n=\sum_{i=0}^{n-1}a_ix^i$. (2) is the content of \cite[Proposition 5.2]{Tol16}\footnote{Note that \J-int is what the author calls integral in \cite[Definition 5.1]{Tol16}.}. 
so there are $a_0, a_1, \ldots , a_{n-1}$ elements of $A$ such that $x^n=\sum_{i=0}^{n-1}a_ix^i$. 

(2) This is the content of \cite[Proposition 5.2]{Tol16}\footnote{Note that \J-int is what 
%the author 
Tolliver
calls integral in \cite[Definition 5.1]{Tol16}.}. 
\end{proof}

%\red{Prove this. Here is a sketch:
%(1) ($\Rightarrow$) $1,x,x^2,\ldots,x^{n-1}$ generate $A[x]$.
%(1) ($\Leftarrow$) If $f_1,\ldots,f_m$ generate $A[x]$, writing each of them as a polynomial in $x$, we see that the list of all powers of $x$ appearing in these polynomials is also a generating set of $A[x]$. Then let $n$ be $1$ plus the greatest power of $x$ occurring in this list.}

Certain relations between these notions are known. In \cite{Tol16} Tolliver shows that if $R$ is a semiring then \J-int elements are \JQ-int and when $R$ is a semifield, then \J-int and \JQ-int elements are the same. If $R$ is a semiring and $A$ a sub-semiring of $R$, Tolliver shows that when $R$ is unit-generated $A$-algebra, then the set of all elements of $R$ that are \JQ-int over $A$ is the intersection of all valuation semirings of $R$ that contain $A$. It is elementary to show that $\ncl{A}\subseteq\jcl{A}\subseteq\dcl{A}$ and $\sqcl{A}\subseteq\qcl{A}\subseteq\wqcl{A}$.
%\pinky{all \n-int elements are \J-int,} all \J-int elements are \D-int, all strongly \Q-int elements are \Q-int, and all \Q-int elements are weakly \Q-int. \orange{(Alternative that may be easier to refer back to: $\ncl{A}\subseteq\jcl{A}\subseteq\dcl{A}$ and $\sqcl{A}\subseteq\qcl{A}\subseteq\wqcl{A}$.)}
Another relation is shown in the lemma below.

\begin{lemma}\label{lem:DintImpliesQint}
    Let $R$ be an additively idempotent semiring and $A$ a sub-semiring of $R$. If $x$ is \D-int then $x$ is \SQ-int.
\end{lemma}

\begin{proof}
    By definition of \D-int, there exists an element $y \in R$ which is \J-int over $A$, such that $x \leq y$. Then $\aangx{A}{y}$ is finite as an $A$-module. Since $\aangx{A}{x} \subseteq \aangx{A}{y}$ and $\aangx{A}{y}$ is a semiring, $\aangx{A}{y}$ is an $\aangx{A}{x}$-module; since $1\in \aangx{A}{y}$, $\aangx{A}{y}$ is a %strongly 
    faithful $\aangx{A}{x}$-module.
\end{proof}

We therefore have $\ncl{A} \subseteq \jcl{A} \subseteq \dcl{A} \subseteq \sqcl{A} \subseteq \qcl{A} \subseteq \wqcl{A}$.

\begin{remark}
Note that $\dcl{A}\subseteq\vcl{A}$, but the relationship of $\vcl{A}$ with the various quasi-integralities are unclear. This is partly because it is not clear whether the intersection of two quasi-integrally closed sub-semirings is again quasi-integrally closed.
\end{remark}

%\section{Elementary properties of integral elements.}

% The proofs in this section are elementary but quite useful and sometimes surprising. \red{(Nati: move this past sentence.)}

For technical reasons, it will also be convenient to consider the notion of \Jinty\ in a slightly more general context. Let $A\subseteq R$ be an extension of semirings and let $L\subseteq R$ be an $A$-submodule. We say that $x\in R$ is \emph{\J-int over $L$} if there are $n\in\Z_{\geq1}$ and $\ell_0, \ell_1,\ldots,\ell_{n-1}\in L$ such that $x^n\leq \dsum_{i=0}^{n-1}\ell_i x^i$. 
Note that, by setting $n=1$ in 
%the definition of \Jinty, 
this definition,
we have that the downward closure of $A$ is contained in $\jcl{A}$.

\begin{lemma}\label{lemma:J-int_equivalents-module}
% Let $A\subseteq R$ be an extensions of semirings, let $L\subseteq R$ be an $A$-submodule, and let $x\in R$. The following are equivalent:
Let $A\subseteq R$ be an extension of semirings and let $L\subseteq R$ be an $A$-submodule. For any $x\in R$ and $n\in\Z_{\geq1}$ the following are equivalent:
\begin{enumerate}
\item\label{lemmaItem:J-int-module} There are $\ell_0,\ldots,\ell_{n-1}\in L$ such that $x^n\leq \ell_0+\ell_1x+\cdots+\ell_{n-1}x^{n-1}$.
\item\label{lemmaItem:SingleCoeff-module} There is an $\ell\in L$ such that $x^n\leq \ell(1+x+\cdots+x^{n-1} )$.
\item\label{lemmaItem:LargeCoeff-module} There is a $b\in L$ such that $x^n\leq \ell(1+x+\cdots+x^{n-1} )$ whenever $\ell\in L$ and $\ell\geq b$.
\end{enumerate}
\end{lemma}\begin{proof}

It is clear that (\ref{lemmaItem:SingleCoeff-module}) implies (\ref{lemmaItem:J-int-module}) and that (\ref{lemmaItem:LargeCoeff-module}) implies (\ref{lemmaItem:SingleCoeff-module}). So it is enough to show that
(\ref{lemmaItem:J-int-module}) implies (\ref{lemmaItem:LargeCoeff-module}). 

Suppose we have $n\in\Z_{\geq1}$ and $\ell_0,\ldots,\ell_{n-1}\in L$ such that $x^n\leq \ell_0+\ell_1x+\cdots+\ell_{n-1}x^{n-1}$. Letting $b=\sum_{i=0}^{n-1}\ell_i\in L$, for any $\ell\geq b$ we have that $\ell_i\leq \ell$. So if $\ell\geq b$ then
$$x^n\leq \ell_0+\ell_1x+\cdots+\ell_{n-1}x^{n-1}
\leq \ell+\ell x+\cdots+\ell x^{n-1}
=\ell(1+x+\cdots+x^{n-1}).$$
\par\nopagebreak\vspace{-1.4\baselineskip}\mbox{}
\end{proof}

\begin{coro}\label{coro:SamePowerAndCoeff}
Let $A\subseteq R$ be an extension of semirings, let $L\subseteq R$ be an $A$-submodule, and suppose that $x,y\in R$ are \J-int over $L$. Then there are $n\in\Z_{\geq1}$ and $\ell\in L$ such that $x^n\leq\ell(1+x+\cdots+x^{n-1})$ and $y^n\leq\ell(1+y+\cdots+y^{n-1})$.
\end{coro}\begin{proof}
By multiplying integral ineqalities for $x$ and $y$ by appropriate powers of $x$ and $y$, we may assume that the inequalities for $x$ and $y$ have the same degree $n$. Lemma~\ref{lemma:J-int_equivalents-module} now tells us that there are $b_1,b_2\in L$ such that, for any $\ell\in L$, we have $\ell\geq b_1\implies x^n\leq\ell(1+x+\cdots+x^{n-1})$ and $\ell\geq b_2\implies y^n\leq\ell(1+y+\cdots+y^{n-1})$. So, if we let $\ell=b_1+b_2$, then $n$ and $\ell$ are as desired.
\end{proof}

Before we use this to show that $\Jcl{L}$ has structure, we need one more lemma. While the statement is quite similar to Lemma~5.2 from \cite{JRT20}, we include a very short proof for the reader's convenience.

\begin{lemma}\label{lemma:ABinomialTheorem}
Let $x$ and $y$ be elements of an (additively idempotent) semiring, let $n,m\in\Z_{\geq1}$, and let $N=n+m-1$. Then $(x+y)^N=x^n(x+y)^{m-1}+y^m(x+y)^{n-1}$. 
\end{lemma}
\begin{proof}
We have 
\begin{align*}
(x+y)^{N} &= \sum_{i=0}^{N} x^i y^{N-i} = \sum_{i=n}^{N} x^i y^{N-i} + \sum_{i=0}^{n-1} x^i y^{N-i} 
= \sum_{i=n}^{N} x^i y^{N-i} + \sum_{j=m}^{N} x^{N-j} y^{j} \\
&= \sum_{k=0}^{m-1} x^n x^k y^{m-1-k} + \sum_{l=0}^{n-1} y^{m} x^{n-1-l} y^{l} = x^n(x+y)^{m-1} + y^m(x+y)^{n-1}.
\end{align*}

\par\nopagebreak\vspace{-2\baselineskip}\mbox{}
\end{proof}

\begin{proposition}\label{prop:J-int_closure_submodule_is_submodule}
Let $A\subseteq R$ be an extension of semirings and let $L\subseteq R$ be an $A$-submodule. Then the \J-int closure $\jcl{L}$ of $L$ is also an $A$-submodule of $R$.
%    \green{Let $x$ and $y$ be \J-int elements of $R$, then $x+y$ is also \J-int.}
%    \pinky{
%    Let $x$ and $y$ be elements of $R$ that are \J-int over $A$. Then $x+y$ is also \J-int over $A$.
%    }
%    %[The sum of two \J-int elements is \J-int]
%    %\#notes 77a
\end{proposition}
\begin{proof}
Say $x,y\in\jcl{L}$; we will show $x+y\in\jcl{L}$.
By Corollary~\ref{coro:SamePowerAndCoeff} there are $m\in\Z_{\geq1}$ and $\ell\in L$ such that $x^m\leq\ell(1+x+\cdots+x^{m-1})$ and $y^m\leq\ell(1+y+\cdots+y^{m-1})$. Let $n=2m-1$.

By Lemma~\ref{lemma:ABinomialTheorem}, we have 
\begin{align*}
(x+y)^n&=x^m(x+y)^{m-1} + y^m(x+y)^{m-1}=(x^m+y^m)(x+y)^{m-1}\\
&\leq\left( \ell(1+x+x^2+\cdots+x^{m-1}) + \ell(1+y+y^2+\cdots+y^{m-1}) \right) (x+y)^{m-1}\\
&\leq\ell(1+(x+y)+(x+y)^2+\cdots+(x+y)^{m-1})(x+y)^{m-1}\\
&=\ell((x+y)^{m-1}+(x+y)^{m}+(x+y)^{m+1}+\cdots+(x+y)^{2m-2})
\end{align*}
and so, because $2m-2=n-1$, $x+y$ is \J-int over $L$.

Now say $x\in\jcl{L}$ and $a\in A$; we will show that $ax\in\jcl{L}$. By assumption there are $n\in\N$ and  $\ell_0,\ell_1,\ldots,\ell_{n-1}\in L$ such that $x^n\leq \dsum_{i=0}^{n-1}\ell_i x^i$. Multiplying through by $a^n$ we get that $(ax)^n=a^nx^n\leq \dsum_{i=0}^{n-1}a^n\ell_i x^i=\dsum_{i=0}^{n-1}a^{n-i}\ell_i (ax)^i$. Since $a^{n-i}\in A$, $\ell_i\in L$, and $L$ is an $A$-module, we have $a^{n-i}\ell_i\in L$. Thus $ax\in\jcl{L}$.
\end{proof}

We now move on to integral closures of sub-semirings.

\begin{lemma}\label{lemma:J-int_equivalents}
Let $A\subseteq R$ be an extension of semirings, let $x\in R$, and let $n\in\Z_{\geq1}$. The following are equivalent:
\begin{enumerate}
\item\label{lemmaItem:J-int} There are $a_0,\ldots,a_{n-1}\in A$ such that $x^n\leq a_0+a_1x+\cdots+a_{n-1}x^{n-1}$.
\item\label{lemmaItem:SingleCoeff} There is an $a\in A$ such that $x^n\leq a(1+x+\cdots+x^{n-1} )$.
\item\label{lemmaItem:LargeCoeff} There is a $b\in A$ such that $x^n\leq a(1+x+\cdots+x^{n-1} )$ whenever $a\in A$ and $a\geq b$.
\item\label{lemmaItem:BinomialIneq} There is an $a\in A$ such that $x^n\leq a(x+a)^{n-1}$.
\item\label{lemmaItem:BinomialEquation} There is an $a\in A$ such that $(x+a)^n=a(x+a)^{n-1}$.
\end{enumerate}
\end{lemma}

\begin{proof}

Lemma~\ref{lemma:J-int_equivalents-module} tells us that (\ref{lemmaItem:J-int})$\iff$(\ref{lemmaItem:SingleCoeff})$\iff$(\ref{lemmaItem:LargeCoeff}). Since it is clear that (\ref{lemmaItem:BinomialIneq})$\implies$(\ref{lemmaItem:J-int}), it is enough to show that
%(3)$\implies$(4), 
(\ref{lemmaItem:LargeCoeff})$\implies$(\ref{lemmaItem:BinomialIneq}) 
%and (4)$\iff$(5).
and (\ref{lemmaItem:BinomialIneq})$\iff$(\ref{lemmaItem:BinomialEquation}).

To see that (\ref{lemmaItem:LargeCoeff})$\implies$(\ref{lemmaItem:BinomialIneq}), note that, if $b\in A$ is as in (\ref{lemmaItem:LargeCoeff}) and we let $a=1+b$ in, then we have $x^n\leq a(1+x+\cdots+x^{n-1})$. 
Since $1\leq a$ we have $1\leq a^i$ for all $i\geq0$, so 
$$x^n\leq a\dsum_{i=0}^{n-1} x^i \leq a\dsum_{i=0}^{n-1} a^{n-1-i}x^i = a(x+a)^{n-1}.$$

Towards showing that (\ref{lemmaItem:BinomialIneq})$\iff$(\ref{lemmaItem:BinomialEquation}), note that, by definition, $x^n\leq a(x+a)^{n-1}$ means that $x^n + a(x+a)^{n-1} = a(x+a)^{n-1}$. On the other hand, applying Lemma~\ref{lemma:ABinomialTheorem} with $m=1$ and $y=a$ gives us that $(x+a)^n = x^n(x+a)^{0}+a^1(x+a)^{n-1} = x^n + a(x+a)^{n-1}$. So $x^n\leq a(x+a)^{n-1}$ if and only if $(x+a)^n=a(x+a)^{n-1}$.
\end{proof}

\begin{coro}\label{coro:J-int_closure_subsemiring_is_submodule}
Let $A\subseteq R$ be an extension of semirings. Then the \J-int closure $\jcl{A}$ of $A$ is an $A$-submodule of $R$.
\end{coro}\begin{proof}
This is a particular case of Proposition~\ref{prop:J-int_closure_submodule_is_submodule}.
\end{proof}

On the other hand, the product of \J-int elements is not always \J-int as shown in the following example, so $\jcl{A}$ is not always a semiring. It is for this reason that we also considered \J-int closures of $A$-submodules. %In order to be able to verify details of the example, we need the following lemma.

\begin{example}\label{ex: J-int-not-semiring}
Let $\calc=\angbra{1+X+X^2\sim 1+X}$ be the congruence on $\B[X]$ generated by $1+X+X^2\sim 1+X$, 
consider the semiring 
%$\B[X]/\angbra{1+X+X^2\sim 1+X}$
$R=\B[X]/\calc$, 
and let $x$ be the image of $X$ in $R$. So we have $x^2\leq 1+x$, so $x$ is \J-int over $\B$. However, $y=x^2$ is \emph{not} \J-int over $\B$. If it was \J-int over $\B$, then by Lemma~\ref{lemma:J-int_equivalents} there would be some $n\in\Z_{\geq1}$ such that $y^n\leq 1+y+\cdots+y^{n-1}$, i.e., $\dsum_{i=0}^n y^i=\dsum_{i=0}^{n-1}y^i$, which is to say $\dsum_{i=0}^n x^{2i}=\dsum_{i=0}^{n-1}x^{2i}$. So we would have $\left(\dsum_{i=0}^n X^{2i},\dsum_{i=0}^{n-1}X^{2i}\right)\in 
%\angbra{1+X+X^2\sim 1+X}$.
\calc$.
On the other hand, Lemma~\ref{lemma:MaclaganRincon_weak_congruence_structure_lemma} tells us that every non-diagonal pair in 
%$\angbra{1+X+X^2\sim 1+X}$ 
$\calc$ 
is in the transitive closure of the set of pairs of the form $(X^m(1+X+X^2)+h,\, X^m(1+X)+h)$ or $(X^m(1+X)+h,\, X^m(1+X+X^2)+h)$ for $m\in\Z_{\geq0}$ and $h\in \B[X]$. But this implies that, for every non-diagonal pair $(F,G)$ in $\calc$, each of $F$ and $G$ contains two consecutive powers of $X$. But this contradicts $\left(\dsum_{i=0}^n X^{2i},\dsum_{i=0}^{n-1}X^{2i}\right)\in \calc$.\exEnd
\end{example}

\begin{remark}
Example~\ref{ex: J-int-not-semiring} also shows that \J-int closure is \emph{not} a closure operation. We have seen above that $\jcl{\B}$ contains $\B$ and $x$, but not $y:=x^2$. Then $x^2\leq 1+x$ is a \J-int inequality for $y$ over $\jcl{\B}$, so $\jcl{\B}$ is strictly contained in the \J-int closure of $\jcl{\B}$.
\end{remark}

% \red{K:The following statement is true for all types of quasi integral elements. If we want to keep it and is needed, change the hypotheses, otherwise move it to misc.tex in this folder.} \blue{Nati: One way to frame this prop is that all quasi-int closures are downward-closed. Then we could state a corollary of the following form: Let $R$ be additively generated by a set $G$ and let $A$ be any sub-semiring of $R$. Then $\sqcl{A}=\dcl{A}$ ($\qcl{A}=\dcl{A}$, $\wqcl{A}=\dcl{A}$) iff $G\cap\sqcl{A}\subseteq \dcl{A}$ (resp. $G\cap\qcl{A}\subseteq \dcl{A}$, $G\cap\wqcl{A}\subseteq \dcl{A}$)}

\begin{proposition}\label{prop:QIntIsDC}
    % \green{
    % Consider a semiring $R$ and let $x, y$ be two elements of $R$, such that $x+y$ is \JQ-int over $A$, a sub-semiring of $R$. Then both $x$ and $y$ are \JQ-int over $A$.
    % }
    For any extension $A\subseteq R$ of semirings, $\sqcl{A}$, $\qcl{A}$, and $\wqcl{A}$ are downward-closed in $R$.
\end{proposition}

\begin{proof}
    % \green{
    % Since $x \leq x+y$, $x$ is contained in the downward-closure of $A[x+y]$ in $R$, i.e., $x\in \aangx{A}{x+y}$, therefore $\aangx{A}{x} \subseteq \aangx{A}{x+y}$. Since $x+y$ is \JQ-int over $A$, then there is an $\aangx{A}{x+y}$-module $M \leq R$, which is finite as an $A$-module and annihilator-free faithful as an $\aangx{A}{x+y}$-module. Since $M$ is an $\aangx{A}{x+y}$-module and $\aangx{A}{x} \subseteq \aangx{A}{x+y}$, $M$ is an $\aangx{A}{x}$-module. Since $M$ is a annihilator-free faithful $\aangx{A}{x+y}$-module and $\aangx{A}{x} \subseteq \aangx{A}{x+y}$, $M$ is a annihilator-free faithful $\aangx{A}{x}$-module, so $x$ is \JQ-int over $A$.
    % }

We give the proof for $\sqcl{A}$; the proofs for $\qcl{A}$ and $\wqcl{A}$ are analogous.

Say $x\leq y$ and $y\in\sqcl{A}$. Note that $x$ is in the downward closure of $A[y]$, i.e., $x\in \aangx{A}{y}$, so $\aangx{A}{x}\subseteq\aangx{A}{y}$. Since $y$ is \SQ-int over $A$, there is a faithful $\aangx{A}{y}$-module $M\subseteq R$ that is finite as \dc{} $A$-module. Since $M$ is a (faithful) $\aangx{A}{y}$-module and $\aangx{A}{x}\subseteq\aangx{A}{y}$, $M$ is a (faithful) $A$-module. Thus $x$ is \SQ-int over $A$.
\end{proof}

The proof of the following corollary uses a small piece of Proposition~\ref{prop:DClosureIsSemiring}, whose proof does not depend on this corollary in any way.

\begin{coro}
Let $R$ be a semiring that is additively generated by a set $G$ and let $A$ be any sub-semiring of $R$. Then $\sqcl{A}=\dcl{A}$ ($\qcl{A}=\dcl{A}$, $\wqcl{A}=\dcl{A}$) if and only if $G\cap\sqcl{A}\subseteq \dcl{A}$ (resp. $G\cap\qcl{A}\subseteq \dcl{A}$, $G\cap\wqcl{A}\subseteq \dcl{A}$).
\end{coro}\begin{proof}
We give the proof for $\sqcl{A}$; the proofs for $\qcl{A}$ and $\wqcl{A}$ are analogous.

($\Rightarrow$) This direction is clear.

($\Leftarrow$) Lemma~\ref{lem:DintImpliesQint} tells us that $\dcl{A}\subseteq\sqcl{A}$, so we just need to show $\sqcl{A}\susbeteq \dcl{A}$. Fix $x\in \sqcl{A}$ and write $x=g_1+\cdots+g_n$ with $g_i\in G$. Then each $g_i\leq x$, so Proposition~\ref{prop:QIntIsDC} tells us that $g_i\in \sqcl{A}$. 
%\red{(Technically we haven't shown that $\dcl{A}$ is closed under addition, so it is possible that we should move this to just after Proposition~\ref{prop:DClosureIsSemiring}.)}
So $g_i\in G\cap\sqcl{A}\subseteq\dcl{A}$, and because Proposition~\ref{prop:DClosureIsSemiring} tells us that $\dcl{A}$ is closed under addition, $x=g_1+\cdots+g_n\in\dcl{A}$.
\end{proof}

%\section{More facts about integral closure in general (obviously to be renamed)}
\section{\D-int closure}

%\red{Nati: Consider putting a sentence here about 4.1-4.3.}

% \orange{
% We proceed to show that for an idempotent semiring extension $A \subseteq R$ the operation assigning $A$ to $\dcl{A}$ is a closure operator. An analogous behavior is observed in \cite{JRT20} when $A$ is an ideal.
% }
%

%We proceed to show that, for any semiring $R$, the map $A\mapsto \dcl{A}$ is a closure operation on the set of sub-semirings of $R$. An analogous behavior is observed in \cite{JRT20} for ideals.

%Let $A\subseteq R$ be an extension of semirings. 
In this section we focus our attention on \D-int closure, i.e., the operation $A\mapsto\dcl{A}$. We start by proving alternative characterizations of $\dcl{A}$ that will be useful later on. We then proceed to show that \D-int closure has two desirable properties that \J-int ``closure'' does not. First we show that $\dcl{A}$ is always a semiring. Then we prove that $A\mapsto\dcl{A}$ is a closure operation, which is analogous to the behavior for ideals observed in \cite{JRT20}. This justifies calling $\dcl{A}$ the \D-int closure of $A$.

Perhaps surprisingly, our first result of this section does not involve \Dinty\ at all. It will, however, allow us to prove an alternate characterization of $\dcl{A}$ in Corollary~\ref{coro:dclIsDownClosureOfNcl}

\begin{prop}\label{prop:JintAndNint}
Let $A\subseteq R$ be an extension of semirings and let $x\in R$. Then $x$ is \J-int over $A$ if and only if there is an element of $x+A=\{x+a\,:\, a\in A\}$ that is \n-int over $A$.
\end{prop}
\begin{proof}
Suppose that $x$ is \J-int over $A$. By parts (\ref{lemmaItem:J-int}) and (\ref{lemmaItem:BinomialEquation}) of Lemma~\ref{lemma:J-int_equivalents}, there is an $a\in A$ and $n\in\Z_{\geq1}$ such that $(x+a)=a(x+a)^{n-1}$. Thus $x+a$ is the desired element of $x+A$.

For the other direction, suppose that $y=x+a\in x+A$ is \n-int over $A$. 
So there are $n\in\Z_{\geq1}$ and $b_0,b_1,\ldots,b_{n-1}\in A$ such that $y^n=b_0+b_1y+\cdots+b_{n-1}y^{n-1}$. 
In particular $y^n\leq b_0+b_1y+\cdots+b_{n-1}y^{n-1}$, so parts (\ref{lemmaItem:J-int}) and (\ref{lemmaItem:BinomialEquation}) of Lemma~\ref{lemma:J-int_equivalents} tell us that there is a $b\in A$ such that $(y+b)^n=b(y+b)^{n-1}$, i.e., $(x+a+b)^{n}=b(x+a+b)^{n-1}$.
Multiplying $x+a+b\geq a+b\geq b$ by $(x+a+b)^{n-1}$ we find that
$$(x+a+b)^n\geq(a+b)(x+a+b)^{n-1}\geq b(x+a+b)^{n-1}=(x+a+b)^n.$$
Thus, we have $c=a+b\in A$ satisfying $(x+c)^n=c(x+c)^{n-1}$, so parts (\ref{lemmaItem:J-int}) and (\ref{lemmaItem:BinomialEquation}) of Lemma~\ref{lemma:J-int_equivalents} tell us that $x$ is \J-int over $A$.
\end{proof}

\begin{coro}
Let $A\subseteq R$ be an extension of semirings and let $x\in R$ and $a\in A$. If $x+a$ is \J-int over $A$ then $x$ is \J-int over $A$.
\end{coro}\begin{proof}
By the forward direction of Proposition~\ref{prop:JintAndNint} since $x+a$ is \J-int there is a $b\in A$ with $(x+a)+b$ \n-int over $A$. Since both $a$ and $b$ are in $A$, we can apply the backward direction of Proposition~\ref{prop:JintAndNint} to $x+(a+b)$ to get that $x$ is \J-int over $A$.
%\red{prove this. Idea: By the prop, there is a $b\in A$ with $(x+a)+b$ \n-int over $A$. So, applying the prop to $x+(a+b)$, we get that $x$ is \J-int over $A$.}
\end{proof}

\begin{coro}\label{coro:dclIsDownClosureOfNcl}
Let $A\subseteq R$ be an extension of semirings. Then $\dcl{A}$ is the downward closure of $\ncl{A}$.
\end{coro} \begin{proof}
Because $\ncl{A}\subseteq\jcl{A}$, we have that the downward closure of $\ncl{A}$ is contained in the downward closure of $\jcl{A}$ which, by definition, is $\dcl{A}$.

For the other direction, say $x\in\dcl{A}$. So there is some $y\in\jcl{A}$ with $x\leq y$. By Proposition~\ref{prop:JintAndNint}, there is an $a\in A$ with $y+a\in\ncl{A}$. Since $x\leq y\leq y+a$, $x$ is in the downward closure of $\ncl{A}$.
\end{proof}

Before we show that $\dcl{A}$ is a semiring, we provide one more alternative characterization of $\dcl{A}$.

\begin{prop}\label{prop: AJJisAD}
Let $A\subseteq R$ be an extension of semirings. Then $\dcl{A}$ is the \J-int closure of the $A$-module $\jcl{A}$.
\end{prop}
\begin{proof}
% We first show that $\dcl{A}$ is contained in the \J-int closure of the $A$-module $\jcl{A}$. 
% %Let $d$ be \D-int over $A$, then there exists a $d'$ which is \J-int over $A$ with $d\leq d'$. 
% Let $d$ be \D-int over $A$, i.e., there is a $d'\in\jcl{A}$ such that $d\leq d'$. 
% %
% By (2) of Lemma~\ref{lemma:J-int_equivalents} we need to show that there exists a positive integer $n$ and $a \in \jcl{A}$  such that $d^n \leq a(d^{n-1}+ \ldots + 1)$. This inequality is satisfied with $n = 1$ and $a = d'$, so $d$ is \J-int over $\jcl{A}$.

Applying the fact that downward closure is contained in \J-int closure % I have now included this statement just before Lemma~\ref{lemma:J-int_equivalents-module}.
to the $A$-module $\jcl{A}$ shows that $\dcl{A}$ is contained in the \J-int closure of $\jcl{A}$.

% The other inclusion follows verbatim the proof of \cite[Proposition 5.5]{JRT20}\footnote{in \cite{JRT20} the authors refer to the \J-int closure of $\jcl{I}$ as $({I^{int}})^{int}$ and to $\dcl{I}$ as $({I^{int}})'$}, where the authors show that the \J-int closure of the $A$-module $\jcl{A}$ is contained in $\dcl{A}$. Since in their case $A$ is an ideal, we use Lemma~\ref{lemma:J-int_equivalents}(5) which is the semiring version of their \cite[Proposition 5.1]{JRT20}.
The proof of the other inclusion follows
the part of the proof of \cite[Proposition 5.5]{JRT20}\footnote{In \cite{JRT20} the authors refer to the \J-int closure of $\jcl{I}$ as $({I^{int}})^{int}$ and to $\dcl{I}$ as $({I^{int}})'$.} where the authors show that the \J-int closure of the \J-int closure of an ideal is contained in the \D-int closure of that ideal.\footnote{This part starts with ``Let $x\in (I^{int})^{int}$.'' on the third line of the proof.} 
The only change that needs to be made is that the implicit use of \cite[Proposition 5.1]{JRT20} for ideals is replaced with  Lemma~\ref{lemma:J-int_equivalents}(5) for semirings.
\end{proof}

We now move into showing that $\dcl{A}$ has the desired properties that $\jcl{A}$ does not.

\begin{prop}\label{prop:DClosureIsSemiring}
Let $A\subseteq R$ be an extension of semirings. Then $\dcl{A}$ is a semiring.
\end{prop}
\begin{proof}

%Let $a, b \in \dcl{A}$, that is, there exist $u, v \in \jcl{A}$ with $a \leq u$ and $b \leq v$. Since $a+b \leq u+v$, by Corollary~\ref{coro:J-int_closure_subsemiring_is_submodule} $u+v$ is \J-int over $A$ and so $a+b$ is \D-int over $A$. Remains to show that the product $ab$ is also \D-int. Since $u$ and $v$ are \J-int there exist $c$  and $n\in \Z_{\geq 1}$ large enough such that $u \leq c(1+u+\ldots+u^{n-1})$ and $v \leq c(1+v+\ldots+v^{n-1})$. Since $ab \leq uv$ it is enough to show that there exists a $z \in \jcl{A}$ with $uv \leq z$.
Let $a, b \in \dcl{A}$, that is, there exist $u, v \in \jcl{A}$ with $a \leq u$ and $b \leq v$. By Corollary~\ref{coro:J-int_closure_subsemiring_is_submodule} $u+v$ is \J-int over $A$ and so, because $a+b \leq u+v$, we have $a+b\in\dcl{A}$. 

It remains only to show that the product $ab$ is also \D-int. 
By Corollary~\ref{coro:SamePowerAndCoeff}, there are $c\in A$ and an integer $n\geq2$ 
such that $u^{n} \leq c(1+u+\ldots+u^{n-1})$ and $v^{n} \leq c(1+v+\ldots+v^{n-1})$. Since $ab \leq uv$ it is enough to show that there exists a $z \in \jcl{A}$ with $uv \leq z$.
%

%We claim that $z = (1+u+\ldots+u^{n-1})(1+v+\ldots+v^{n-1}) $ and $z^2 \leq ez$, for $e\in A$. 
We claim that $z = (1+u+\ldots+u^{n-1})(1+v+\ldots+v^{n-1})$ works. Certainly $uv\leq z$; we will show that $z^2 \leq ez$ for some $e\in A$. 
Note that the condition that there exists $e\in A$ with $z^2 \leq ez$ is equivalent to the existence of a polynomial $f(x,y) \in A[x,y]$ with both $x$-degree and $y$-degree 
%at most $n$ such that $z^2 \leq f(x,y)$. 
less than $n$ such that $z^2\leq f(u,v)$.
%
%Let $f(x,y)\in A[x,y]$, such that $z^2\leq f(x,y)$ 
Fix $f(x,y)\in A[x,y]$ such that $z^2\leq f(u,v)$ 
and $m =\max\{x\text{-degree of } f, y\text{-degree of } f\}$ is as small as possible. 
Note that 
%$m \leq 2n$, since 
there is such an $f$ because
$z^2 \leq F(u,v)$, where $F(x,y) = (1+x+\ldots+ x^{2n})(1+y+\ldots+ y^{2n})$. 
If 
%$m \leq n$ 
$m<n$
we are done, so assume for contradiction that 
%$m > n$.
$m\geq n$.
%

% We can assume without loss of generality that there exists a $d \in A$ such that 
% \begin{align*}
%     f(x,y) &= d(1+x+\ldots+ x^{n})(1+y+\ldots+ y^{n})\\
%     f(x,y) &= \wt{f}(x,y) + d(x^my^m + \sum_{0\leq j\leq n} x^m y^j + \sum_{0\leq i\leq n} x^i y^m), \text{where}\\
%     \wt{f}(x,y) &= d(1+x+\ldots+ x^{m-1})(1+y+\ldots+ y^{m-1}).
% \end{align*}

%By increasing 
By increasing the coefficients of
% \red{scaling?}\blue{No, we need more than just scaling. For example, if $f$ started out as $1+t^2x^2+ty^5+t^9xy^2$ then the new $f$ would be $t^9(1+x+\cdots+x^5)(1+y+\cdots+y^5)$.} \red{I see but we should still explain it somehow, since increasing is not necessarily clear.} \blue{Does this new version work better?}
$f$ if necessary, we can assume without loss of generality that $f$ is of the form $f(x,y) = d(1+x+\ldots+ x^{m})(1+y+\ldots+ y^{m})$ for some $d\in A$.
So we can write 
\begin{align*}
    f(x,y) &= \wt{f}(x,y) + d\left( x^my^m + \sum_{0\leq j< m} x^m y^j + \sum_{0\leq i< m} x^i y^m \right),
    %\text{where}
    \\
    \intertext{where}
    \wt{f}(x,y) &= d(1+x+\ldots+ x^{m-1})(1+y+\ldots+ y^{m-1}).
\end{align*}
%
% Again since $m>n$ from $$u^{n+1} \leq c(1+u+\ldots+ u^{n}) \text{ and } v^{n+1} \leq c(1+v+\ldots+ v^{n})$$ we get 
% \begin{align*}
%     u^{m} &\leq c(u^{m-n-1}+\ldots+ u^{m-1}) \leq c(1+u+\ldots+ u^{m-1}), \text{ and }\\ 
%     v^{m} &\leq c(v^{m-n-1}+\ldots+ v^{m-1}) \leq c(1+v+\ldots+ v^{m-1}).
% \end{align*}
Since $m\geq n$, we can multiply each of 
$$u^{n} \leq c(1+u+\ldots+ u^{n-1}) \text{ and } v^{n} \leq c(1+v+\ldots+ v^{n-1})$$
by $u^{m-n}$ and $v^{m-n}$, respectively, to get
\begin{align*}
    u^{m} &\leq c(u^{m-n}+\ldots+ u^{m-1}) \leq %c(1+u+\ldots+ u^{m-1}) \text{ and }\\ 
    c\dsum_{0\leq i<m}u^i \text{ and }\\ 
    v^{m} &\leq c(v^{m-n}+\ldots+ v^{m-1}) \leq %c(1+v+\ldots+ v^{m-1}).
    c\dsum_{0\leq j<m}v^j.
\end{align*}
%
% Now using the presentation of $f(x,y)$ above, we obtain 
% \begin{align*}
% f(u,v) = &\wt{f}(u,v) + d(u^mv^m + \sum_{0\leq j\leq n} u^m v^j + \sum_{0\leq i\leq n} u^i v^m) \leq g(u,v) :=\\
%         & \wt{f}(u,v) + dc^2(1+u+\ldots+u^{m-1})(1+v+\ldots+v^{m-1}) + \\
%         &dc \big(\sum_{0\leq j\leq m} v^j(1+u+\ldots+u^{m-1})\big) + dc \big(\sum_{0\leq i\leq m} u^i(1+v+\ldots+v^{m-1})\big) 
% \end{align*}
Now using the presentation of $f(x,y)$ above, we obtain 
\begin{align*}
f(u,v) =&\ \wt{f}(u,v) + d\left( u^mv^m + \sum_{0\leq j< m} u^m v^j + \sum_{0\leq i<m} u^i v^m \right)\\%\leq g(u,v) :=\\
=&\ \wt{f}(u,v) + d\left( u^mv^m + u^m\sum_{0\leq j< m} v^j + v^m\sum_{0\leq i<m} u^i \right) \\
% &\leq
% \wt{f}(u,v) + dc^2(1+u+\ldots+u^{m-1})(1+v+\ldots+v^{m-1}) + \\
% &dc \big(\sum_{0\leq j\leq m} v^j(1+u+\ldots+u^{m-1})\big) + dc \big(\sum_{0\leq i\leq m} u^i(1+v+\ldots+v^{m-1})\big) 
\leq&\ \wt{f}(u,v) + dc^2\left(\dsum_{0\leq i<m}u^i\right)\left(\dsum_{0\leq i<m}v^j\right)\\
&\ + dc \left(\dsum_{0\leq i<m} u^i\right)\left(\sum_{0\leq j< m} v^j\right) + dc \left(\dsum_{0\leq j<m}v^j\right)\left(\sum_{0\leq i< m} u^i\right)\\
&\!\!=:g(u,v).
\end{align*}
%
%What we have obtained is $z^2 \leq f(u,v) \leq g(u,v)$, however the degree of the polynomial $g(x,y)$ in both $x$ and $y$ is less than $m$ which is a contradiction to the choice of $f(x,y).$ 
%Since both the $x$-degree and $y$-degree of $g(x,y)$ are strictly less than $m$, the fact that $z^2\leq f(u,v)\leq g(u,v)$ contradicts the choice of $f(x,y)$.
So $x^2\leq f(u,v)\leq g(u,v)$, but both the $x$-degree and $y$-degree of $g(x,y)$ 
are strictly less than $m$, 
%%%%% Weird spacing
% are \green{$<m$,} 
% \orange{less than $m$,}
%
contradicting the choice of $f(x,y)$.
% end of pinky
%
\end{proof}

%\red{(Nati asks: Should we consider putting a remark here (or somewhere else) to the effect of: We know that $\dcl{A}$ is a semiring and that $\jcl{A}$ is not necessarily a semiring. It is not clear whether $\ncl{A}$, $\sqcl{A}$, $\qcl{A}$, and $\wqcl{A}$ are always semirings or not. It isn't even clear whether these are always additively closed.)}

\begin{remark}
Let $A\subseteq R$ be an extension of semirings. 
We now know that $\vcl{A}$  and $\dcl{A}$ are semirings (by definition and Proposition~\ref{prop:DClosureIsSemiring}) and that $\jcl{A}$ is not necessarily a semiring (by Example~\ref{ex: J-int-not-semiring}).
However, it is not clear whether $\ncl{A}$, $\sqcl{A}$, $\qcl{A}$, and $\wqcl{A}$ are always semirings or not. It is not even clear whether these are always additively closed.
\end{remark}

The following lemma is the semiring analogue of \cite[Lemma 5.4]{JRT20}. 
Its proof follows from the proof of that lemma, except that the implicit use of \cite[Proposition 5.1]{JRT20}\footnote{Used in the proof where it says ``by definition''.} for ideals is replaced by Lemma~\ref{lemma:J-int_equivalents}(5) for semirings.
%% \blue{if you will use the $[X]_{\leq}$ notation it probably should start here.}\red{I was trying to keep it just within a single proof and not make a global notation out of it.} \orange{yes bu we have a somewhat different notation for the same thing in the next lemma.} \red{Giving something a notation and giving it a name seem different to me. Saying that $B$ is the downward closure of $A$ is something that is seems across to me as local and can change from one result to the next. Also, I don't like the notation $[X]_{\leq}$, I just got desperate because the proof of Proposition~\ref{prop:DclIsAClosure} is hard to read without some notation for the downward closure. As such, I'd rather keep it as contained as possible.}

\begin{lemma}\label{lemm: JclofDclIsAJ}
Let $A\subseteq R$ be an extension of semirings. Let $B$ be the downward closure of $A$ in $R$. Then $\jcl{B}=\jcl{A}$. %$\jcl{[A]_{\leq}}$
\end{lemma}
% \begin{proof}
% The proof follows verbatim the proof of \cite[Lemma 5.4]{JRT20}. Since in their case $A$ is an ideal, we use Lemma~\ref{lemma:J-int_equivalents}(5) which is the semiring version of their \cite[Proposition 5.1]{JRT20}.
% \end{proof}

\begin{prop}\label{prop:DclIsAClosure}
Let $R$ be a semiring. The operation $A\mapsto\dcl{A}$ on sub-semirings of $R$ is a closure operation.
\end{prop}
\begin{proof}
%This statement follows immediately by Proposition~\ref{prop: AJJisAD} and Lemma~\ref{lemm: JclofDclIsAJ}, which give us that the \J-int closure of $\dcl{A}$ equals the \J-int closure of ${\jcl{A}}$ which is just $\dcl{A}$, which in turn implies that the \D-int closure of $\dcl{A}$ is just $\dcl{A}$.
It is routine to check that $A\subseteq\dcl{A}$ and that $A\subseteq B\implies \dcl{A}\subseteq\dcl{B}$.

For the remainder of this proof, we denote the downward closure of a set $X\subseteq R$ by $[X]_{\leq}$.

%For idempotence, 
Towards proving that the operation $A\mapsto\dcl{A}$ is idempotent,
%\red{$\leftarrow$ not sure I understand}\blue{ one of the axioms for a closure operation is idempotence: that when you apply the operation twice it is the same as applying it once.} \red{since at a first glace we are not looking at $A^{D^D}$ it might not be obvious what is happening; I don't know probably anyone with a little brain will understand} \blue{Does this new version work better?} 
note that the \J-int closure of $\dcl{A}=[\jcl{A}]_{\leq}$ is the same as the \J-int closure of $\jcl{A}$ by Lemma~\ref{lemm: JclofDclIsAJ}. By Proposition~\ref{prop: AJJisAD}, this is $\dcl{A}$. Thus, the \D-int closure of $\dcl{A}$, being the downward closure of the \J-int closure of $\dcl{A}$, is $[\dcl{A}]_{\leq}=[[\jcl{A}]_{\leq}]_{\leq}=[\jcl{A}]_{\leq}=\dcl{A}$.
\end{proof}

\section{The additively idempotent Cayley-Hamilton Theorem}

\newcommand{\aMatrix}{\mathcal{A}}
\newcommand{\bMatrix}{\mathcal{B}}

\begin{definition}
    Let $\aMatrix$ be an $n\times n$ matrix over an additively idempotent semiring $S$. The \textit{characteristic polynomial} of $\aMatrix$ is $p_\aMatrix(x) := \det(\aMatrix + xI_n) = \operatorname{perm}(\aMatrix + xI_n).$
\end{definition}

The goal of this section is to prove the following result.

\begin{thm}[Tropical Cayley-Hamilton Theorem]
\label{thm:CH}
%\label{thm:tropCH}
Any square matrix over an additively idempotent semiring satisfies all of the bend relations of its characteristic polynomial.
\end{thm}

Our proof of Theorem~\ref{thm:CH} builds on 
%the 
Straubing's
combinatorial proof of the classical Cayley-Hamilton theorem, presented in \cite{Str82}. We refer the reader to that work for an excellent presentation with helpful figures. We begin by introducing notation and outlining the proof given in \cite{Str82}.

Let $\bMatrix=(b_{i,j})$ be an $n\times n$ matrix over a commutative ring and let $p_\bMatrix(x)=\det(xI-\bMatrix)$ be the characteristic polynomial of $\bMatrix$.

%%%%% \wsgn is the "weird" sign that is used here.

%For any set $I$, let $\frakS_I$ denote the symmetric group on $I$. A \emph{partial permutation} of $\{1,\ldots,n\}$ is a $\sigma\in\frakS_I$ for some $I\subseteq\{1,\ldots,n\}$. \red{(Nati says: I don't think we actually use $\frakS_I$ anywhere else, so it is possible that we should remove this notation. I'll wait until I'm done writing this section to be sure.)} 
A \emph{partial permutation} of $\{1,\ldots,n\}$ is a permutation of some $I\subseteq\{1,\ldots,n\}$.
We let $\dom\sigma=I$ denote the domain of $\sigma$ and let $|\sigma|$ denote the cardinality of $I$, which we will call the \emph{degree} of $\sigma$. Then the coefficient of $x^{k}$ in $p_B(x)$ is 
$$c_k:=\dsum_{|\sigma|=n-k}\wsgn\sigma\dprod_{\ell\in\dom\sigma}b_{\ell,\sigma(\ell)}$$
where $\wsgn\sigma$ is $\pm1$ but is not generally the usual sign of the permutation $\sigma$. 
%Straubing shows that $\wsgn\sigma$ is $(-1)^r$ where $r$ is the number of cycles in the disjoint cycle decomposition of $\sigma$. \red{(This previous sentence can almost surely be removed.)}
Note that there is exactly one partial permutation of degree $0$, being the empty permutation; with the convention that the empty product is $1$ we get $c_n=1$.
%The \emph{weight} of $\sigma$ is $\mu(\sigma)=\prod_{\ell\in\dom\sigma}b_{\ell,\sigma(\ell)}$.
Defining the \emph{weight} of $\sigma$ to be  $\displaystyle \mu(\sigma)=\prod_{\ell\in\dom\sigma}b_{\ell,\sigma(\ell)}$, we have $c_k=\dsum_{|\sigma|=n-k}(\wsgn\sigma)\mu(\sigma)$.

\newcommand{\aPath}{\frakp}
\newcommand{\apath}{\aPath}

%A \emph{path} $\aPath$ in $\{1,\ldots,n\}$ is a 
A sequence $\aPath=((\ell_0,\ell_1), (\ell_1,\ell_2), \ldots, (\ell_{k-1},\ell_k))$ with each $\ell_\bullet\in\{1,\ldots,n\}$ is called a \emph{path} in $\{1,\ldots,n\}$. The \emph{length} of $\apath$ is $|\apath|:=k$. The \emph{vertices} of $\apath$ are the $\ell_\bullet$s and its \emph{edges} are the $(\ell_{\bullet-1},\ell_\bullet)$s. 
The \emph{weight} of $\apath$ is $\mu(\apath)=\dprod_{  \substack{ (i,j)\text{ an} \\ \text{edge of }\apath }  }  b_{i,j}$. 
The \emph{start} of $\apath$ is $\alpha(\apath):=\ell_0$ and its \emph{end} is $\omega(\apath):=\ell_k$. 
By convention, for each $i\in\{1,\ldots,n\}$ there is a unique path $\apath_i$ of length $0$ that starts and ends at $\alpha(\apath_i)=\omega(\apath_i)=\ell_0=i$; we set $\mu(\apath_i)=1$. 
With this setup we have that, for any $k\geq0$, the $(i,j)$-entry of $\bMatrix^k$ is $(\bMatrix^k)_{i,j}= \dsum_{\substack{  \apath \text{ a path of} \\ \text{length }k\\\text{from } i\text{ to }j  }} \mu(\apath)
%=  \dsum_{\substack{ \alpha(\apath)=i,\\ \omega(\apath)=j, \\ |\apath|=q }}  \mu(\apath)
$.

We let $T_{i,j}:= \left\{ (\sigma,\apath)\,:\,\substack{\textstyle\sigma\text{ is a partial permutation,}\\ \textstyle |\sigma|+|\apath|=n,\ \alpha(\apath)=i,\ \omega(\apath)=j}\right\}$ and, for any pair $(\sigma,\apath)\in T_{i,j}$, we say that its \emph{weight} is $\mu(\sigma,\apath):=\mu(\sigma)\mu(\apath)$\footnote{This notation is not used in \cite{Str82}, but will be convenient for us.}. So the $(i,j)$-entry of $p_\bMatrix(\bMatrix)$ is $(p_\bMatrix(\bMatrix))_{i,j}=\dsum_{(\sigma,\apath)\in T_{i,j}}(\wsgn\sigma) \mu(\sigma,\apath)$.

Straubing proves that $\dsum_{(\sigma,\apath)\in T_{i,j}}(\wsgn\sigma) \mu(\sigma,\apath)=0$ by constructing an involution $\eta=\eta_{i,j}:T_{i,j}\to T_{i,j}$ such that, if $\eta(\sigma,\apath)=(\sigma',\apath')$ then $\wsgn\sigma'=-\wsgn\sigma$ and $\mu(\sigma',\apath')=\mu(\sigma,\apath)$.

To facilitate understanding $\eta$, Straubing views each pair $(\sigma,\apath)$ as% (or, more accurately, associates to each such pair) 
\footnote{or, more accurately, associates to each such pair}
a two-colored directed (multi)graph (with loops) with vertex set $\{1,\ldots,n\}$. 
The edges of the first color are $(\ell,\sigma(\ell))$ for $\ell\in\dom\sigma$ and the edges of the second color are the edges $(\ell_{0},\ell_1), (\ell_1,\ell_2),\ldots,(\ell_{k-1},\ell_k)$ of $\apath$. 
Note that $\mu(\sigma,\apath)$ only depends on the underlying graph, not on the coloring of the edges.

For $0\leq r\leq k$, consider the following two properties:
\begin{equation}
    \text{There is a }q<r\text{ such that }\ell_q=\ell_r.
    \tag{P1}\label{property:pathCycle}
\end{equation}
\begin{equation}
    %\ell_r\in\dom\sigma.
    \text{The vertex }\ell_r\text{ is in }\dom\sigma.
    \tag{P2}\label{property:permCycle}
\end{equation}
Straubing proves that there is an $r$ satisfying \eqref{property:pathCycle} or \eqref{property:permCycle}. 
He then considers the smallest such $r$ and shows that it satisfies exactly one of \eqref{property:pathCycle} or \eqref{property:permCycle}. 
In either case, there is an associated directed cycle in the graph, given as follows.
If $r$ satisfies \eqref{property:pathCycle} then the directed cycle is $(\ell_{q},\ell_{q+1})\cdots(\ell_{r-1},\ell_r)$. If $r$ satisfies \eqref{property:permCycle}, then the directed cycle consists of the edges $(\ell,\sigma(\ell))$ for $\ell$ in the cycle (in the permutation sense) of $\sigma$ containing $\ell_r$. In either case, $\eta$ acts by switching the color of all edges in the directed cycle.

Since the underlying graph has not changed, $\mu(\sigma,\apath)$ is also unchanged. Straubing also verifies the sign property, completing his proof. %\red{(Should we take out this last line? It is not relevant to the things we need, but it kind of wraps up this part.)}

We are now ready to proceed to the proof of the Tropical Cayley-Hamilton Theorem.

\begin{proof}[Proof of 
%Theorem~\ref{thm:tropCH}]
Theorem~\ref{thm:CH}]
Let $\aMatrix=(a_{i,j})$ be an $n\times n$ matrix over an additively idempotent semiring $S$, and let $p(x)=p_\aMatrix(x)$ be its characteristic polynomial. Write $p(x)=%d_0+d_1x+\cdots+d_nx^n
\dsum_{m=0}^n d_mx^m
$ and, for each $0\leq k\leq n$, define $p_{\hat{k}}(x):=\dsum_{m\neq k}d_mx^m$. Our goal is to show that $p_{\hat{k}}(\aMatrix)=p(\aMatrix)$ for all $k$; we will prove this entry-wise.

We retain all earlier notations except that (1) $\wsgn\sigma$ (and so also $c_k$) is meaningless because $-1_S$ does not make sense in $S$\footnote{Recall that, when $S=\T$, we write elements of $S$ as $t^a$ with $a\in\R$, so $t^{-1}\in S$ but $-1\notin S$.} and (2) in defining $\mu(\sigma)$ and $\mu(\apath)$ we use $a_{\bullet,\bullet}$ instead of $b_{\bullet,\bullet}$ and so the operations are in a semiring rather than a ring.

Expanding out the permanent $p(x)=\operatorname{perm}(\aMatrix+xI)$ we find that, for all $k$, $d_k=\dsum_{|\sigma|=n-k}\mu(\sigma)$. 
So $(p(\aMatrix))_{i,j}=\dsum_{(\sigma,\apath)\in T_{i,j}}\mu(\sigma,\apath)$. 
Moreover, if we let $T_{i,j}(k):=\{(\sigma,\apath)\in T_{i,j}\,:\,|\apath|=k\}$ 
then $d_k(\aMatrix^k)_{i,j}=\dsum_{(\sigma,\apath)\in T_{i,j}(k)}\mu(\sigma,\apath)$ and 
$\left(p_{\hat{k}}(\aMatrix)\right)_{i,j}=\dsum_{(\sigma,\apath)\in T_{i,j}\sdrop T_{i,j}(k)}\mu(\sigma,\apath)$.

Suppose $(\sigma,\apath)\in T_{i,j}(k)$. If $\eta(\sigma,\apath)=(\sigma',\apath')$ then $|\apath'|\neq|\apath|=k$ because $\apath'$ is obtained from $\apath$ by adding or removing edges, but not both. That is, $\eta(\sigma,\apath)\in T_{i,j}\sdrop T_{i,j}(k)$. Thus, we have shown $\eta\left(T_{i,j}(k)\right)\subseteq T_{i,j}\sdrop T_{i,j}(k)$.

We thus have $$d_k(\aMatrix^k)_{i,j} = \dsum_{(\sigma,\apath)\in T_{i,j}(k)} \mu(\sigma,\apath) 
= \dsum_{(\sigma,\apath)\in \eta\left(T_{i,j}(k)\right)} \mu(\sigma,\apath) 
\leq \dsum_{(\sigma,\apath)\in T_{i,j\sdrop }T_{i,j}(k)} \mu(\sigma,\apath) 
= \left(p_{\hat{k}}(\aMatrix)\right)_{i,j},$$
where the second equality holds because $\eta$ is injective and preserves weights of pairs. 
Expanding out the definition of $\leq$, we find that 
$\left(p_{\hat{k}}(\aMatrix)\right)_{i,j} = \left(p_{\hat{k}}(\aMatrix)\right)_{i,j} +d_k(\aMatrix^k)_{i,j}=(p(\aMatrix))_{i,j}$,
as desired.
\end{proof}

%\red{Should we be making a distinction between finitely generated modules and finite modules? Or maybe write something like finitely generated DC-module?}

\begin{coro}\label{coro: CH-coro1}
%\label{coro: tropCH-coro1}
    Let $S$ be an additively idempotent semiring and let $N$ be a finitely generated $S$-module. If $\varphi: N \to N$ is an $S$-module homomorphism then there exists a monic polynomial $p(x)$ in $S[x]$ such that $\varphi$ satisfies the bend relations of $p(x)$.
\end{coro}

\begin{proof}
Pick a presentation $\pi: S^n \to N$. Then there exists a matrix $\aMatrix$ over $S$ that makes the diagram 
\comd{
S^n\ar[r]^{\aMatrix}\surj[d]_{\pi}&S^n\surj[d]^{\pi}\\
N\ar[r]^{\ph}&N
}
commute.
%\green{We will show that if $f,g \in S[x]$ are such that $f(\aMatrix) = g(\aMatrix)$, then $f(\varphi) = g(\varphi)$. Applying this to the bend relations of $p_\aMatrix(x)$ will then prove the result.} 
%\blue{do we want to explicitly reference the CH theorem?} 
We claim that if $f,g \in S[x]$ are such that $f(\aMatrix) = g(\aMatrix)$, then $f(\varphi) = g(\varphi)$.
By Theorem~\ref{thm:CH}, $\aMatrix$ satisfies the bend relations of $p_{\aMatrix}(x)$, so applying our claim to these relations will then prove the result.

Let $e_1, \dots, e_n$ be the generators of $S^n$ so the set of images $\{\pi(e_i)\}$ generates $N$. So it suffices to show that $f(\varphi)(\pi(e_i)) = g(\varphi)(\pi(e_i))$ which is true since $f(\varphi)(\pi(e_i)) = \pi(f(\aMatrix)(e_i)) = \pi(g(\aMatrix)(e_i)) = g(\varphi)(\pi(e_i)).$
\end{proof}

\begin{coro}\label{coro: CH-coro2}
%\label{coro: tropCH-coro2}
    Let $A \subseteq R$ be an extension of additively idempotent semirings. Let $N$ be a %strongly 
    faithful $R$-module that is finitely generated as an $A$-module. Then every $x\in R$ is \J-int over $A$.
\end{coro}

\begin{proof}
    Fix $x\in R$. Multiplication by $x$ gives rise to an $A$-module homomorphism $\varphi_x: N\to N$. Since $N$ is a finitely generated $A$-module, by Corollary~\ref{coro: CH-coro1} there exists a monic polynomial $p(x)$ such that $\varphi_x$ satisfies all bend relations of $p(x)$. Let $q(x)$ be the polynomial obtained from $p(x)$ by removing the leading term (highest degree term). Then $(p(x), q(x))$ is a pair in the bend relations of $p(x)$ and by Corollary~\ref{coro: CH-coro1} $p(\varphi_x) = q(\varphi_x)$. 
    Since $N$ is a faithful $R$ module, the homomorphism $\Phi: R \to \operatorname{End}(N)$, which sends $a$ to $\varphi_a$ is injective. 
    %In particular, we have $\phi(p(x)) = \varphi_{p(x)} = p(\varphi_x) = q(\varphi_x) = \varphi_{q(x)} = \phi(q(x))$, showing $p(x) = q(x)$ and so that $x$ is \J-int. 
    In particular, we have $\Phi(p(x)) =  p(\Phi(x)) = p(\ph_x) = q(\ph_x) = q(\Phi(x)) = \Phi(q(x))$, showing $p(x) = q(x)$ and so that $x$ is \J-int. 
\end{proof}

\begin{coro}\label{coro:D-intAndFaithfulA[x]Mods}
Let $A\subseteq R$ be an extension of additively idempotent semirings. Let 
$$X:=\{x\in R\,:\,\text{there is a %strongly 
faithful }A[x]\text{-module that is finitely generated as an }A\text{-module}\}.$$ 
Then $\dcl{A}$ is the downward closure of $X$.
\end{coro}
\begin{proof}
Given $y\in \dcl{A}$, Corollary~\ref{coro:dclIsDownClosureOfNcl} tells us that there is an $x\geq y$ that is \n-int over $A$. Then part (\ref{propItem:N-integralAsAlgebraFinite}) of Proposition~\ref{prop:IntegralAsAlgebrasFinite} tells us that $A[x]$ is finitely generated as an $A$-module. Since $A[x]$ is a %strongly 
faithful $A[x]$-module, this shows that $x\in X$, so $y$ is in the downward closure of $X$.

Now suppose that $y$ is in the downward closure of $X$, i.e., there is an $x\in X$ with $y\leq x$. Then, by Corollary~\ref{coro: CH-coro2}, $x$ is \J-int over $A$ so, by definition, $y$ is \D-int over $A$.
\end{proof}
% \red{(It could be that we should emphasize Corollary~\ref{coro:D-intAndFaithfulA[x]Mods} more. This is because $X$ is the natural semiring analogue of the more usual faithful module condition in the ring setting, (that $x\in R$ is integral over $A\susbeteq R$ if and only if there is a faithful $A[x]$-module that is finitely generated as an $A$-module) and so this statement gives a form of equation/inequality integrality is the same as module integrality that is true in an arbitrary semiring.)}

%%%%%

\section{Integral closure in  semirings with no zero divisors}\label{section: A+ for cancel gen}

%\red{(Nati says: We might consider separating out something like ``applications and examples'' into a separate small section or subsection. This would include Lemma~\ref{lemma:strongMR} through Corollary~\ref{coro:ModBendVIntIsQInt} and Examples~\ref{ex: allIntegral-toric} and \ref{ex: allIntegral-terry}.)}\blue{yes, let's do that and add the examples from current sect. 5 }
%%%%% Done!

In this section we extend the results of \cite[Section 4]{Tol16} by significantly relaxing the assumptions on both $A$ and $R$ when taking integral closures of $A$ in $R$.
% In this section we extend the results of \cite[Section 4]{Tol16} from 
% \green{
% ``simple'' semirings to 
% %finitely 
% cancellatively 
% generated $\mathbb{B}$-algebras. 
% }
% \pinky{``simple'' semirings, unit-generated semirings, and semifields to semirings with no zero-divisors, cancellatively-generated semirings, and cancellative semirings.} \red{(Nati says: I'm not 100\% sure I got this right; we should double-check it.)} \blue{I think for someone who doesn't know what is coming or is familiar with Jeff's paper this is very confusing since these adjectives point to A or R but not necessarily both. I suggest we either be more clear or super vague.}
We then prove that the different notions of integral closure of a semiring $A$ in a cancellatively generated semiring $R$ agree under some hypotheses.

The following proposition is \cite[Proposition 4.11]{Tol16} with the assumption on the ``simplicity'' of the semiring dropped.

\begin{proposition}\label{prop:Jeff4.11}
    Let A be an additively idempotent semiring. Let $\pi: A \to A/\sqrt{\Delta}$ be the quotient map. 
    %Then $\pi(x) = \pi(y)$ if and only if every homomorphism $v : A\to S$ 
    % %into a totally ordered additively idempotent semifield
    % with $S$ a totally ordered additively idempotent semifield 
    % %
    % satisfies $v(x) = v(y)$, equivalently, if and only if for every prime congruence $P$ on $A$ we have $|x|_P = |y|_P$.
    Then $\pi(x) = \pi(y)$ if and only if $v(x)=v(y)$ for every homomorphism $v : A\to S$ with $S$ a totally ordered additively idempotent semifield.
    Equivalently, this occurs if and only if $|x|_P=|y|_P$ for every prime congruence $P$ on $A$.
    
\end{proposition}

\begin{proof}
    The proof is the same as the one of \cite[Proposition 4.11]{Tol16}, since the assumption that the semiring is ``simple'' is not used in the original proof. 
\end{proof}

% \begin{lemma}\label{lem:noZD-diagNoIK}
%     Let $A$ be a semiring with no zero divisors. Then the ideal-kernel of $\sqrt{\Delta}$ is trivial.
% \end{lemma}

% \begin{proof}
%     Assume that the ideal-kernel of $\sqrt{\Delta}$ is non-trivial, that is, there exists an element $a\in A \setminus \{0\}$ such that $(a, 0) \in \sqrt{\Delta}$. By Theorem~\ref{thm: radGP}
%     %\cite[Proposition 5.2]{JM17} \red{(Ideally this will become a reference to a single statement in Section 2, rather than relying on both \cite[Proposition 5.2(i)]{JM17} and the characterization of $\sqrt{\calc}$ in terms of generalized powers.)} 
%     for the pair $(a,0)$ there exist $i \in \mathbb{Z}_{\geq 1}$ and $h \in A$ such that $(a^i + h)(a,0)$ is in $\Delta$, or equivalently, $(a^i+h)a = 0$. Since $a \neq 0$ by assumption and $A$ has no zero divisors, we must have $a^i + h = 0$ which in turn implies that 
%     %$a^i = h = 0$ 
%     $a^i=0$
%     by Lemma~\ref{lem:zero-sum-free}, making $a$ a zero divisor. Note that $i$ has to be bigger than 1, because $a$ is not 0 by assumption.
% \end{proof}

%this is Jeff's 4.12 in our context.
\begin{proposition}\label{prop:Jeff4.12New}
    Let $A$ be a semiring with no zero-divisors and $A/\sqrt{\Delta}$ cancellative. Let $\pi: A \to A/\sqrt{\Delta}$ be the quotient map. Then $\pi(x) = \pi (y)$ if and only if there exists $s \in A\setminus \{0\}$ such that $sx = sy$.
\end{proposition}

\begin{proof}
    Assume that $\pi(x) = \pi(y)$. Since $A$ has no zero divisors the zero ideal is prime. Let $K = A_{(0)}$ be the localization of $A$ by all non-zero elements of $A$ and let $g:A \to K$ be the localization map. $K$ is a semifield and thus cancellative, so by Corollary~\ref{coro:factor-quotient-map} %\red{ add more; (Right now this is a self-reference. I assume this is supposed to be to a statement that says that, if $\sqrt{\Delta}$ is QC, then $A\to A/\sqrt{\Delta}$ is universal for maps from $A$ to a cancellative semiring.)} 
    the map $g$ factors uniquely as $g = f \circ \pi$, where $f: A/\sqrt{\Delta} \to K$. Then 
    $$\frac{x}{1} = g(x) = f(\pi(x)) = f(\pi(y)) = g(y) = \frac{y}{1},$$
    so there is some $s \in A\setminus\{0\}$ with $sx=sy$.
    
    For the other direction, assume there is some $s \in A\setminus\{0\}$ with $sx=sy$. Since $\pi$ is a homomorphism we have $\pi(s)\pi(x) = \pi(s)\pi(y)$. Since $A/\sqrt{\Delta}$ is cancellative and the ideal-kernel of $\pi$ is $\{0\}$ by Lemma~\ref{lem:noZD-diagNoIK}, we have that $\pi(x) = \pi(y)$.  
\end{proof}

\begin{theorem}\label{thm:Jeff4.13New}
    Let $A$ be a semiring with no zero divisors and with $A/\sqrt{\Delta}$ cancellative. Fix elements $x, y \in A$. Then $sx=sy$ for some $s\in A\sdrop\{0\}$ if and only if for every prime congruence $P$ on $A$ we have $|x|_P = |y|_P$.
\end{theorem}

\begin{proof}
    This immediately follows from Proposition~\ref{prop:Jeff4.11} and  Proposition~\ref{prop:Jeff4.12New}
\end{proof}

\begin{coro}\label{coro:Jeff4.14New}
    Let $A$ be a semiring with no zero divisors and with $A/\sqrt{\Delta}$ cancellative. Fix $x$ in $A$. Then $sx \leq s$ for some $s\in A\sdrop\{0\}$ if and only if for every prime congruence $P$ on $A$ we have $|x|_P \leq 1_{\kappa(P)}$.
\end{coro}

\begin{proof}
For any $s$, $sx\leq s$ is equivalent to $sx+s=s$, i.e., $s(x+1)=s\cdot1$. But we know that there is an $s\in A\sdrop\{0\}$ with $s(x+1)=s\cdot 1$ if and only if for all prime congruences $P$ on $A$ we have $|x|_P+|1|_P=|x+1|_P=|1|_P$, i.e., $|x|_P\leq|1|_P=1_{\kappa(P)}$.
\end{proof}

We need one more piece of machinery before we can make a statement about the different versions of integral closure of semirings in this setting. This is the definition given in \cite[Definition 6.2]{Tol16}.

% \green{
% \begin{definition}
%     Let A be an additively idempotent semiring and let $M$ be an $A$-module. The \emph{$A$-contraction of $M$}, denoted $M_{A\leq1}$, is defined by the following universal property: there is a homomorphism $\varphi : M \to M_{A\leq1}$ such that $\varphi(ax) \leq \varphi(x)$ for all $a \in A$ and $x\in M$, and any other homomorphism $\psi : R \to R'$ which satisfies $\psi(ax) \leq \psi(x)$ for all $a\in A$ and $x\in M$ factors uniquely through $\varphi$.
% \end{definition}
% }

\begin{defi}
Let $A$ be an additively idempotent semiring and let $M$ be an $A$-module. The \emph{$A$-contraction of $M$}, denoted $M_{A\leq1}$, is the quotient of $M$ by the relations imposing that $ax\leq x$ for all $a\in A$ and $x\in M$. That is, $M_{A\leq 1}:=M/\calc$ where $\calc=\angbra{ax+x\sim x\,:\, a\in A, x\in M}$. Thus $M_{A\leq1}$ satisfies the following universal property. There is an $A$-module homomorphism $\ph:M\to M_{A\leq 1}$ (the quotient map) such that $\ph(ax)\leq\ph(x)$ for all $a\in A$ and $x\in M$ and, for any $A$-module homomorphism $\psi:M\to M'$ satisfying $\psi(ax)\leq \psi(x)$ for all $x\in M$, $\psi$ factors uniquely through $\ph$.
\end{defi}

% \green{
% \begin{remark}\label{rmk: contraction}
%     A more congruence theoretic way to think of the contraction $R_{A\leq 1}$ is the following: Let $\mathcal{C}$ be the congruence on $R$ generated by pairs of the form $(1, a+1)$, for all $a\in A$, then $R_{A\leq 1}$ is just the quotient $R/\mathcal{C}$. Equivalently, the congruence $\mathcal{C}$, can be written as
%     $$\mathcal{C} = \{(x, y) : \exists a, b\in A \text{ such that } x \leq ay \text{ and } y \leq bx \}.$$    
%     Moreover, the universal property states that for any $\mathcal{C}' \supset \mathcal{C}$ a homomorphism $\psi: R \to R/\mathcal{C}'$ factors through $R/\mathcal{C}$.
% \end{remark}
% }

\begin{lemma}\label{lemma:contractionCongruence}
Let $A$ be an additively idempotent semiring and let $M$ be an $A$-module. Let $\calC=\angbra{ax+x\sim x\,:\, a\in A, x\in M}$ be the congruence on $M$ generated by the relations imposing that $ax\leq x$ for all $a\in A$ and $x\in M$. Then $$\calc=\{(x, y) : \text{there are } a, b\in A \text{ such that } x \leq ay \text{ and } y \leq bx \}.$$
\end{lemma}\begin{proof}
This is Lemma~6.3 in \cite{Tol16} (modulo typos).
\end{proof}

\begin{coro}\label{coro:contractionOfAlgebraIsAlgebra}
Let $A\subseteq R$ be an extension of additively idempotent semirings. Let 
$$\calc=\angbra{ax+x\sim x\,:\, a\in A, x\in R}_{A-\mathrm{mod-cong}}$$
be the $A$-module congruence on $R$ generated by the relations imposing that $ax\leq x$ for all $a\in A$ and $x\in R$. Then $\calc$ is a semiring congruence on $R$, so $R_{A\leq 1}$ is a semiring.
In particular, $\calc$ is equal to the semiring congruence on $R$ generated by the relations $a+1\sim 1$ for $a\in A$.
\end{coro}\begin{proof}
Using the description of $\calc$ from Lemma~\ref{lemma:contractionCongruence}, it is routine to check that, if $(x,y)\in\calc$ and $r\in R$ then $(rx,ry)\in\calc$.
\end{proof}

\begin{coro}\label{coro:ContractionsPreserveNoZeroDivs}
%Let $A\subseteq R$ be an extension of additively idempotent semirings and suppose that $R$ has no zero divisors. Then $R_{A\leq 1}$ has no zero divisors.
Let $A$, $R$, and $\calc$ be as in Corollary~\ref{coro:contractionOfAlgebraIsAlgebra}. If $R$ has no zero divisors then $\calc$ has trivial ideal-kernel and $R_{A\leq1}$ has no zero divisors.
\end{coro}\begin{proof}
Suppose $(0,y)\in\calc$. By Lemma~\ref{lemma:contractionCongruence}, there are $a,b\in A$ such that $0\leq ay$ and $y\leq b\cdot0=0$, so $y=0$. So $\calc$ has trivial ideal-kernel.
%
%Let $\ph:R\to R_{A\leq1}$ be the quotient map. Suppose that $\ph(x)\ph(y)=0$ in $R_{A\leq 1}$. Then, because $\calc$ has trivial ideal-kernel, $xy=0$ in $R$. Since $R$ has no zero divisors, $x=0$ or $y=0$. 
The fact that $R_{A\leq1}$ has no zero divisors now follows from Lemma~\ref{lemma:ZeroDivsModCongWithNoIdealKernel}.
%
%In fact, we'll be able to get this whole corollary as a lemma just after the definition of contraction by using what will be the next lemma after Lemma~\ref{lemma:ZeroDivsModCongWithNoIdealKernel} - Lemma~\ref{lemma:ZeroDivsModCongGendByTotallyNonzeroPairs}. Except that it actually has to be here because this is the first place we know that $R_{A\leq1}$ is a semiring.
%
%The previous proof and the following proof seem to be equally short. I don't have a strong opinion as to which one we use.
%
% Alternative proof: Note that, for any $a\in A$ and $x\in R$, $ax+x$ and $x$ are either both zero or both nonzero; this follows from Lemma~\ref{lem:zero-sum-free} which tells us $ax+x=0\implies ax=0$ and $x=0$. The result now follows from Lemma~\ref{lemma:ZeroDivsModCongGendByTotallyNonzeroPairs}.
\end{proof}

% \begin{remark}
%     \green{
%     The contraction $R_{A\leq 1}$ is the same as the set of finitely generated \dc\ $A$-modules of $R$.
%     }
%     % \pinky{
%     % The contraction $M_{A\leq 1}$ is the same as the set of finitely generated \dc\ $A$-modules of $M$.
%     % }
% \end{remark}

%\newcommand{\dcmod}[2]{\operatorname{DCMod}^{\mathrm{fg}}_{#1}#2}
\newcommand{\dcmod}[2]{\operatorname{DCMod}_{#1}#2}

\begin{coro}\label{coro:contractionAndDCModules}
Let $A$ be an additively idempotent semiring and let $M$ be an $A$-module. Let $\dcmod{A}{M}$ denote the set of finitely generated \dc\ $A$-submodules of $M$ with the usual operations. Then the map $\Phi:M\to\dcmod{A}{M}$ sending $x$ to the \dc\ module it generates induces an $A$-module isomorphism $\iota:M_{A\leq 1}\cong\dcmod{A}{M}$.

Similarly, if $A\subseteq R$ is an extension of additively idempotent semirings then the map $\Phi:R\to\dcmod{A}{R}$ induces an $A$-algebra isomorphism $\iota:R_{A\leq 1}\cong\dcmod{A}{R}$.

\end{coro}
\begin{proof}
It is routine to check that $\Phi$ is a homomorphism. In light of Lemma~\ref{lemma:contractionCongruence}, the congruence-kernel of $\Phi$ is equal to the congruence $\calc$ whose quotient is $M_{A\leq1}$. Thus there is an induced injective homomorphism $\iota: M_{A\leq 1}\into\dcmod{A}{M}$. The map $\iota$ is surjective because $\Phi$ is, i.e., every finitely generated \dc\ $A$-module is generated as a \dc\ $A$-module by a single element. This is because, if $f_1,\ldots,f_n$ generate $L$ as a \dc\ $A$-module, then so does $f_1+\cdots+f_n$.
\end{proof}

\begin{remark}
In particular, Corollary~\ref{coro:contractionAndDCModules} tells us that $\dcmod{A}{M}$ is an $A$-module and $\dcmod{A}{R}$ is an $A$-algebra. On the other hand, the set of (finitely generated) $A$-submodules of $M$ with the usual operations is generally \emph{not} an $A$-module - it need not satisfy the distributive property.
\end{remark}

\begin{coro}\label{coro:thisDCModuleAndContraction}
Let $A$, $M$, $\Phi$, and $\iota$ be as in Corollary~\ref{coro:contractionAndDCModules}. Let $\ph:M\to M_{A\leq 1}$ denote the quotient map. Then $\iota$ sends $y\in M_{A\leq1}$ to $\{x\in M\,:\, \ph(x)\leq y\}$.
\end{coro}\begin{proof}
Note that the order on $\dcmod{A}{M}$ is given by $L_1\leq L_2$ if and only if $L_1\subseteq L_2$. 
Writing $y=\ph(z)$, we have
%$x\in\iota(y)=\iota(\pi(z))=\Phi(z)$ iff $\Phi(x)\subseteq \Phi(z)$ iff $\iota(\pi(x))=\Phi(x)\leq\Phi(z)=\iota(\pi(z))=\iota(y)$ iff $\pi(x)\leq y$ because $\iota$ is an isomorphism.
$\iota(y)=\iota(\ph(z))=\Phi(z)$. So, for any $x\in M$, we have  
$x\in\iota(y)=\Phi(z)$ if and only if $\Phi(x)\subseteq \Phi(z)$, i.e., $\Phi(x)\leq\Phi(z)$. 
Since $\Phi(x)=\iota(\ph(x))$ and $\Phi(z)=\iota(y)$ and $\iota$ is an isomorphism, this occurs if and only if $\ph(x)\leq y$.
\end{proof}

\begin{proposition}\label{prop: QI-contraction}
    % \green{
    % Let $A \subseteq R$ be and extension of additively idempotent semirings. Let $x\in R$ and $\bar{x}$ be the its image in $R_{A\leq 1}$. Then there exists a non-zero element $r \in R_{A\leq 1}$ such that $\bar{x}r \leq r$ if and only if there exists a non-zero \dc\ $A\left<x\right>$-module $M\subseteq R$ that is finite as a \dc\ $A$-module.
    % }

    Let $A \subseteq R$ be an extension of additively idempotent semirings, let $\ph:R\to R_{A\leq 1}$ be the quotient map, and let $x\in R$. Then there exists a non-zero element $r \in R_{A\leq 1}$ such that $\ph(x)r \leq r$ if and only if 
    %there exists a non-zero \dc\ $A\left<x\right>$-module $M\subseteq R$ that is finite as a \dc\ $A$-module.
    $x$ is \WQ-int over $A$.
    
\end{proposition}

\begin{proof}
    % This statement follows from the proof of \cite[Proposition 6.5]{Tol16}, but we include it here for the reader's convenience as it is not too long.
    
    % Suppose $r\bar{x} \leq r$. Let $M = \{m \in R \ : \ \bar{m} \leq r\}$. By \cite[Lemma 6.4]{Tol16}, $M$ is a
    % \dc\ $A$-submodule of $R$ and is finite as an $A$-module. If $m \in M$, then $xm = \bar{x}\bar{m} \leq \bar{x}r \leq r$ so $xm \in M$ implying that $M$ is an $A\left<x\right>$-submodule of $R$. 
    
    % For the other direction: choose a \dc\ $A\left<x\right>$-submodule $M \subseteq R$ which is finite as an $A$-module. There is some $s \in R_{A\leq1}$ such that $M = \{m \in R \ : \ \bar{m} \leq s\}$. Write $s = \bar{r}$. Then $r \in M$ and so $xr \in M$. Then $\bar{x}s = \overline{xr} \leq s$.

    This statement follows from the proof of \cite[Proposition 6.5]{Tol16}, but can also be quickly deduced from Corollary~\ref{coro:thisDCModuleAndContraction} as follows.

    By Corollary~\ref{coro:contractionAndDCModules} we have that $\iota$ gives a bijection between nonzero elements $r\in R_{A\leq1}$ and nonzero finitely generated \dc\ $A$-modules $M\subseteq R$. By Corollary~\ref{coro:thisDCModuleAndContraction} this bijection sends $r\in R_{A\leq 1}$ to $M=\{y\in R\,:\, \ph(y)\leq r\}$. Thus we see that $M$ is an $\aangx{A}{x}$-module if and only if $r\ph(x)\leq r$.
\end{proof}

We would like to use Corollary~\ref{coro:Jeff4.14New} and Proposition~\ref{prop: QI-contraction} to obtain a different characterization of the valuative integral closure of $A$ in $R$, i.e. $\vcl{A}$, and relate it to $\Jcl{A}$ (resp. $\Dcl{A}$, $\Qcl{A}$, $\sQcl{A}$, $\wQcl{A}$). 

In order to do so we need to show that the contraction satisfies the hypotheses of Corollary~\ref{coro:Jeff4.14New}. In Corollary~\ref{coro:ContractionsPreserveNoZeroDivs} we showed that the property of having no zero divisors is preserved upon taking a contraction. 
So we just need to show that the radical of the diagonal of $R_{A\leq1}$ is cancellative whenever the radical of the diagonal of $R$ is; 

this is the statement of our next proposition.

\begin{proposition}\label{prop: DcancDContrCanc}
    Let $A \subseteq R$ be an extension of additively idempotent semirings. %\pinky{and suppose that $R$ has no zero divisors}. 
    If radical of the diagonal of $R$ %$\sqrt{\Delta_R}$, 
    is cancellative then so is the radical of the diagonal of the contraction $R_{A\leq1}$. %$\sqrt{\Delta_{R_{A\leq1}}}$. 
\end{proposition}

\begin{proof}
    By Corollary~\ref{coro:contractionOfAlgebraIsAlgebra} we can write $R_{A\leq1}=R/\calc$, where $\mathcal{C}=\angbra{(a+1,1)\,:\,a\in A}$ is the semiring congruence on $R$ generated by the pairs $(a+1,1)$ for $a\in A$.
    We need to prove that $\sqrt{\mathcal{C}}$ is cancellative. 

    Let $\pi:R\to R/\sqrt{\Delta_R}$ denote the quotient map and let $\pi_*(\calc)$ denote the pushforward of $\calc$ along $\pi$, i.e., the semiring congruence on $R/\sqrt{\Delta_R}$ generated by the pairs $(\pi(f),\pi(g))$ where $(f,g)\in\calc$. We claim that 
    $$\pi_*(\calc)=\left\{(x,y)\in\left(R/\sqrt{\Delta_R}\right)^2\,:\,\text{there are }a,b\in A\text{ such that }x\leq \pi(a)y\text{ and }y\leq \pi(b)x\right\}.$$
    Note that $\pi_*(\calc)$ is generated as a semiring congruence on $R/\sqrt{\Delta_R}$ by the pairs $(\pi(a)+1,1)$ for $a\in A$. 
    So, if we let $B=\pi(A)$, then $\pi_*(\calc)=\angbra{(b+1,1)\,:\,b\in B}_{\mathrm{semiring}}$ is the semiring congruence on $R/\sqrt{\Delta_R}$ generated by $(b+1,1)$ for $b\in B$, 
    so Corollary~\ref{coro:contractionOfAlgebraIsAlgebra} tells us that $(R/\sqrt{\Delta_R})/\pi_*(\calc)=(R/\sqrt{\Delta_R})_{B\leq 1}$. 
    Thus Lemma~\ref{lemma:contractionCongruence} tells us that 
    \begin{align*}
    \pi_*(\calc)&=\left\{(x,y)\in\left(R/\sqrt{\Delta_R}\right)^2\,:\,\text{there are }c,d\in B\text{ such that }x\leq cy\text{ and }y\leq dx\right\}\\
    &=\left\{(x,y)\in\left(R/\sqrt{\Delta_R}\right)^2\,:\,\text{there are }a,b\in A\text{ such that }x\leq \pi(a)y\text{ and }y\leq \pi(b)x\right\},
    \end{align*}
    proving the claim.
    
    Now assume for contradiction that $\sqrt{\calc}$ is not cancellative, i.e., there is a pair $(hf,hg) \in \sqrt{\mathcal{C}}$ such that neither $(f,g)$ nor $(h,0)$ are in $\sqrt{\mathcal{C}}$. 
    Since $(hf,hg) \in \sqrt{\mathcal{C}}$, Theorem~\ref{thm: radGP} tells us that there exist $k \in \Z_{\geq0}$ and $\delta \in R$ such that $r = (fh+gh)^k+\delta$ and $r(fh, gh) \in \mathcal{C}$.
    By Lemma~\ref{lemma:contractionCongruence} we have $a,b\in A$ 
    such that 
    \begin{equation}\label{eq: ineqs}
        fhr \leq aghr \text{ and } ghr \leq bfhr.
    \end{equation}
    Because $(f,g)$ and $(h,0)$ are not in $\sqrt{\calC}\supseteq\sqrt{\Delta_R}$, we have that $\pi(h)\neq0$ and $\pi(fh+gh)\neq0$ in $R/\sqrt{\Delta_R}$. Since $\sqrt{\Delta_R}$ is cancellative we also get $\pi(r)=\pi(fh+gh)^k+\pi(\delta)\neq0$ and so applying $\pi$ to the inequalities in (\ref{eq: ineqs}) gives us 
    \begin{equation*}
        \pi(f) \leq \pi(a)\pi(g) \text{ and } \pi(g) \leq \pi(b)\pi(f).
    \end{equation*}
    Thus the claim proven above tells us that $(\pi(f),\pi(g))\in\pi_*(\calc)$, so $(f,g)$ is in $\pi^*(\pi_*(\calc))$, i.e., the congruence $\mathcal{E}$ generated by $\calc$ and $\sqrt{\Delta_R}$.
    Since $\sqrt{\mathcal{C}}$ contains $\mathcal{E}$, this shows that $(f,g) \in \sqrt{\mathcal{C}}$, which is a contradiction. 
\end{proof}

We now put all of the parts together to get a characterization of $\vcl{A}$ with minimal assumptions.

\begin{prop}\label{prop:VIntIsWQInt}
Let $A\subseteq R$ be an extension of additively idempotent semirings such that $R$ has no zero divisors and $\sqrt{\Delta_R}$ is cancellative. Then $\vcl{A}=\wqcl{A}$.
\end{prop}\begin{proof}
Let $\ph:R\to R_{A\leq1}$ be the quotient map. 
By Corollary~\ref{coro:ContractionsPreserveNoZeroDivs}, $R_{A\leq 1}$ has no zero divisors. By Proposition~\ref{prop: DcancDContrCanc}, the radical of the diagonal of $R_{A\leq 1}$ is cancellative.
Thus we can apply Corollary~\ref{coro:Jeff4.14New} To $R_{A\leq 1}$.

Fix $x\in R$.
By definition, $x\in\vcl{A}$ if and only if, for every prime congruence $P$ on $R$ such that $|A|_P\leq1$, $|x|_P\leq1$. By the universal property of $R_{A\leq 1}$, this occurs if and only if, for every prime congruence $Q$ on $R_{A\leq 1}$, $|\ph(x)|_{Q}\leq 1$. By Corollary~\ref{coro:Jeff4.14New} this occurs if and only if there is a nonzero $r\in R_{A\leq 1}$ such that $\ph(x)r\leq r$ and this, in turn, 
%occurs if and only if there is a non-zero \dc\ $A\left<x\right>$-module $M\subseteq R$ that is finite as a \dc\ $A$-module, by Proposition~\ref{prop: QI-contraction}. But this is exactly what it means to have $x\in\wqcl{A}$.
occurs if and only if $x\in\wqcl{A}$, by Proposition~\ref{prop: QI-contraction}. But this is exactly what it means to have $x\in\wqcl{A}$.
\end{proof}

% \green{
% \begin{coro}
% \red{(There is now a stronger result than this: Corollary~\ref{coro:ModBendVIntIsQInt}.)}Let $K$ be a field with valuation $K\to S$, let $\Mon$ a toric monoid, and let $I$ be an ideal of $K[\Mon]$ whose variety $Y := V(I)$ has no component contained in the toric boundary. 
% %Then every minimal prime congruence of $\base[\Mon]/\Bend(\trop I)$ has trivial ideal-kernel. 
% %In particular, $\sqrt{\Bend(\trop I)}$ is cancellative.
% Let $R$ be any localization of $\base[\Mon]/\Bend(I)$ by a set of cancellative elements, and let $A$ be any sub-semiring of $R$. Then $\vcl{A}=\wqcl{A}$.
% \end{coro}\begin{proof}
% By Corollary~\ref{coro:trop-coordinate-semiring-is-cancellative - localized}, $\sqrt{\Delta_R}$ is cancellative. 
% Note that $I$ cannot contain a monomial, for if it did, then $Y$ would be contained in the toric boundary. Thus $\trop(I)$ does not contain a monomial. So Corollary~\ref{coro:BendAndNoZeroDivs} tells us that $R_0=\base[\Mon]/\Bend(I)$ has no zero divisors. Since $R$ is a localization of $R_0$ by cancellative elements, $R$ has no zero divisors, either. 
% Thus, Proposition~\ref{prop:VIntIsWQInt} tells us that $\vcl{A}=\wqcl{A}$.
% \end{proof}
% }

We now work to extend Proposition~\ref{prop:VIntIsWQInt} to cases where we can say that the various notions of integral closure all agree. Our first extension will be to cancellatively generated semirings.
Note that, in specifying that $R$ is cancellatively generated, it does not matter whether we talk about $R$ being additively generated by cancellative elements or $R$ being generated as a semiring by cancellative elements because the set of cancellative elements is a multiplicative submonoid of $R$.%\red{Nati: Move this to where we actually start talking about cancellatively generated things.} %%% Done!

\begin{lemma}\label{lem: dc-strFaith}
    Let $A \subseteq R$ be an extension of additively idempotent semirings. Suppose that $R$ is cancellatively generated. Then every non-zero \dc\ $A$-submodule of $R$ is %strongly 
    faithful as an $A$-module. 
\end{lemma}

\begin{proof}
    Let $L\subseteq R$ be non-zero \dc\ $A$-submodule of $R$, so we can fix a nonzero $f\in L$. Writing $f=\dsum_{i=1}^n c_i$ with each $c_i$ cancellative, we get $c_1\leq f$; since $L$ is \dc, $c_1\in L$. Now, if $a,b\in A$ are distinct then $ac_1\neq bc_1$ because $c_1$ is cancellative, so $a$ and $b$ act differently on $L$. Thus, $L$ is %strongly 
    faithful as an $A$-module.
\end{proof}

\begin{remark}
    Note that the proof of the previous lemma did not use the full strength of the assumption that $R$ is cancellatively generated - it only used that, for every nonzero $f\in R$ there is a cancellative $c\in R$ with $c\leq f$. The same applies to the proof of the following lemma and therefore also to the statements of Corollary~\ref{coro:VIntIsSQInt} and Theorem~\ref{thm: A+=AJ-can gen}.
\end{remark}

%\blue{(Trying for the form of the lemma that would be needed for weakening the hypothesis in the next lemma).
\begin{lemma}\label{lem: cancNoZeroDiv}
%Let $R$ be an additively idempotent semiring such that, for every nonzero $x\in R$,  there is a cancellative $c\in R$ satisfying $c\leq x$. 
Let $R$ be a cancellatively generated, additively idempotent semiring. 
Then $R$ has no zero divisors.
\end{lemma}
\begin{proof}
Suppose that there are nonzero $x,y\in R$ with $xy=0$. As in the previous proof, there are cancellative elements $c,d\in R$ with $c\leq x$ and $d\leq y$. Then $cd\leq xy=0$ so $cd=0=c\cdot 0$. Since $c$ is cancellative we get $d=0$, contradicting the cancellativity of $d$.
\end{proof}

\begin{coro}\label{coro:VIntIsSQInt}
Let $A\subseteq R$ be an extension of additively idempotent semirings such that $R$ is cancellatively generated and $\sqrt{\Delta_R}$ is cancellative. 
%\green{Then $\vcl{A}=\sqcl{A}$.}
Then $\vcl{A}=\sqcl{A}=\qcl{A}=\wqcl{A}$.
\end{coro}\begin{proof}
%Since $R$ is cancellatively generated, it has no zero divisors Lemma~\ref{lem: cancNoZeroDiv}. 
By Lemma~\ref{lem: cancNoZeroDiv}, $R$ has no zero divisors. 
Thus Proposition~\ref{prop:VIntIsWQInt} tells us that $\vcl{A}=\wqcl{A}$.  
%\green{Lemma~\ref{lem: dc-strFaith} tells us that $\wqcl{A}=\sqcl{A}$, completing the proof.}
Since $\sqcl{A}\subseteq \qcl{A}\subseteq\wqcl{A}$, Lemma~\ref{lem: dc-strFaith}, which tells us that $\wqcl{A}=\sqcl{A}$, completes the proof.
\end{proof}

\begin{lemma}\label{lemma:CancellativeContraction} 
Let $A\subseteq R$ be an extension of semirings. 
If $R$ is cancellative, so is $R_{A\leq 1}$.
\end{lemma}\begin{proof}
% Let $\ph:R\to R_{A\leq1}$ be the quotient map. We want to show that, for all $s,x,y\in R$ with $\ph(s)$, if $\ph(sx)=\ph(sy)$ then $\ph(x)=\ph(y)$. 
%
% So suppose that $\ph(sx)=\ph(sy)$. By Lemma~\ref{lemma:contractionCongruence}, there are $a,b\in A$ such that $sx\leq asy$ and $sy\leq bsx$. Since $s\in R\sdrop\{0\}$ is cancellative, Lemma~\ref{lemma:cancelIneqs} gives us that $x\leq ay$ and $y\leq bx$. Applying Lemma~\ref{lemma:contractionCongruence} again gives us that $\ph(x)=\ph(y)$.
This is an elementary application of Lemma~\ref{lemma:contractionCongruence} using Lemma~\ref{lemma:cancelIneqs}.
\end{proof}

\begin{prop}\label{prop:RCancellativeGivesAllIntClosuresEqual}
Let $A\subseteq R$ be an extension of semirings with $R$ cancellative. 
%Then the downward closure of $A$ in $R$, $\jcl{A}$, $\dcl{A}$, $\sqcl{A}$, $\qcl{A}$, $\wqcl{A}$, and $\vcl{A}$ are all equal.
Then the following sets are all equal: the downward closure of $A$ in $R$, $\jcl{A}$, $\dcl{A}$, $\sqcl{A}$, $\qcl{A}$, $\wqcl{A}$, and $\vcl{A}$.
\end{prop}\begin{proof}
Corollary~\ref{coro:VIntIsSQInt} tells us that $\vcl{A}=\sqcl{A}=\qcl{A}=\wqcl{A}$. 
Since we always have $\jcl{A}\subseteq\dcl{A}\subseteq\vcl{A}$ and the downward closure of $A$ is always contained in $\jcl{A}$, it suffices to show that $\wqcl{A}$ is contained in the downward closure of $A$. 
%This is more or less shown in the proof of \cite[Proposition~6.9]{Tol16} modulo some changed references \blue{K: this sounds a little weird, esp. since 6.9 doesn't have a proof. Perhaps we can say that "the same idea that establishes 6.9 also shows this more general statement"?} \violet{Nati: It seems to me that 6.9 does have a proof. The statement is at the bottom of page 11 and the proof is at the top of page 12. Let me know if this does not line up with what you are seeing; it is possible that I'm looking at some past or faulty version.} \orange{hahaha, I need to update my acrobat reader; still I think the sentence needs to be rewritten as suggested in the blue}; we include a short proof for the reader's convenience.
%%%%% Done!
The same idea that establishes \cite[Proposition~6.9]{Tol16} also shows this more stronger and more general statement; we include a short proof for the reader's convenience.

%By Lemma~\ref{lemma:CancellativeContraction}
Say $x\in\wqcl{A}$ and let $\ph:R\to R_{A\leq 1}$ be the quotient map. By Proposition~\ref{prop: QI-contraction}, there is a non-zero $r\in R_{A\leq 1}$ such that $\ph(x)r\leq r$. Lemma~\ref{lemma:CancellativeContraction} tells us that $R_{A\leq 1}$ is cancellative, so $\ph(x)\leq 1$, i.e., $\ph(x+1)=1=\ph(1)$. Lemma~\ref{lemma:contractionCongruence} now gives us that there is an $a\in A$ such that $x+1\leq a\cdot 1=a$, so $x\leq a$. Thus, $x$ is in the downward closure of $A$.
\end{proof}

%\red{(Insert verbiage to say why we want the next theorem even though we have the previous one.)}
Neither Proposition~\ref{prop:RCancellativeGivesAllIntClosuresEqual} nor the following Theorem~\ref{thm: A+=AJ-can gen} imply the other. In Section~\ref{sec:AppsAndExs} we use each of these results to compute different interesting examples.

\begin{theorem}\label{thm: A+=AJ-can gen}
     Let $A \subseteq R$ be and extension of additively idempotent semirings such that $R$ is cancellatively generated, $\sqrt{\Delta_R}$ cancellative, and, for every $x \in R$, the \dc\ $A$-module generated by $x$ is finite as an $A$-module. Then $\Jcl{A} =\dcl{A}=\sqcl{A}=\qcl{A}=\wqcl{A}= \vcl{A}$.
\end{theorem}

\begin{proof}
%    \green{In view of Corollary~\ref{coro:VIntIsSQInt} we need to show that $\Jcl{A}  = \sQcl{A}$ and we always have $\Jcl{A}  \subseteq \sQcl{A}$.     For the other direction, say}
    In view of Corollary~\ref{coro:VIntIsSQInt} and the fact that $\jcl{A}\subseteq\dcl{A}\subseteq\sqcl{A}$, we just need to show $\sqcl{A}\subseteq\jcl{A}$. Say $x\in\sqcl{A}$, i.e., there is an $\aangx{A}{x}$-module $M \subseteq R$, which is finite as a downward-closed $A$-module and %strongly 
    faithful as an $\aangx{A}{x}$-module. By Proposition~\ref{prop: QI-contraction} (or its proof), $M$ is generated by a single $y$ as a \dc\ $A$-module and so, by assumption, $M$ is finitely generated as an $A$-module. %So Corollary~\ref{coro: CH-coro2} tells us that $x\in\aangx{A}{x}$ is \J-int over $A$.
    So by Corollary~\ref{coro: CH-coro2} we get that $x\in\aangx{A}{x}$ is \J-int over $A$.
\end{proof}

\begin{remark}
Without any assumptions on $A$, it is not possible to guarantee $\jcl{A}=\vcl{A}$ because $\vcl{A}$ is always a semiring and $\jcl{A}$ need not be; see Example~\ref{ex: J-int-not-semiring}. Thus the best we can hope for is $\dcl{A}=\vcl{A}$.
\end{remark}

\begin{question}
Is it always the case that $\dcl{A}=\vcl{A}$? If not, under what conditions does this occur?
\end{question}

%------------------------------------------------------------------------------

\section{Applications and examples}\label{sec:AppsAndExs}

In this section we aim to compute the integral closure of semirings that arise from geometry. We compute the integral closure of a few cancellative semirings in their semifields of fractions and then give examples of semirings that are not cancelllative, but for which the notions of integrality agree.  %We start with a few auxiliary statements
We start by working to show that, under mild hypotheses, we can apply Corollary~\ref{coro:VIntIsSQInt} to the quotient of $\base[\Mon]$ by the bend congruence of a tropicalized ideal. First, we require a lemma.

\begin{lemma}\label{lemma:strongMR}
%\red{(strong Maclagan Rinc\'{o}n $+\eps$)}
%
Let $\base$ be a subsemifield of $\T$, let $v:K\to\base$ be a valuation on a field $K$, let $\Mon$ be a toric monoid, and let $I$ be an ideal of $K[\Mon]$.
If $(f,g)\in\Bend(\trop I)$ (is not in $\Delta_{\base[\Mon]}$) then there is a chain $f=F_0\sim F_1\sim\cdots\sim F_n=g$ such that 
\begin{enumerate}
\item\label{item:bend} each $F_i\sim F_{i+1}$ is of the form $a\trop(h_i)+H\sim a(\trop(h_i))_{\what{u}}+H$ or the reverse for some $h_i\in I$, $a\in\base^\times$, $u\in\supp(h_i)$, and $H\in\base[\Mon]$,
\item\label{item:coeff} for every monomial $m$, the coefficient of $m$ in $F_i$ is either the coefficient of $m$ in $f$ or the coefficient of $m$ in $g$, and
%
%\item\label{item:unimodal} there is an $\ell$ between $0$ and $n-1$ such that $(F_i,F_{i+1})$ is increasing if $i<\ell$ and decreasing if $i\geq\ell$, and
%
\item\label{item:uniqueChange} for each monomial $m$, there is at most one $i$ for which the coefficients of $m$ in $F_i$ and $F_{i+1}$ are different.
\end{enumerate}
\end{lemma}\begin{proof}
This is proven in the proof of \cite[Proposition 2.6]{MR14} which is stated for $\T[x_1^{\pm1},\ldots,x_n^{\pm1}]$ but works in this greater generality. Parts \eqref{item:bend} and \eqref{item:coeff} are in the statement of \cite[Proposition 2.6]{MR14}. For part \eqref{item:uniqueChange}, see the third paragraph of the proof of \cite[Proposition 2.6]{MR14}.
\end{proof}

\begin{remark}\label{rmk:lengthOfTransitiveChain}
Note that, in light of \eqref{item:uniqueChange}, %$n=\delta(f,g)$ 
$n$
is the number of monomials whose coefficients in $f$ and $g$ are different. %\red{(change or move this wording.)}
\end{remark}

\begin{prop}\label{prop:QuotientByBendIsCancGend}
Let $\base$ be a subsemifield of $\T$, let $v:K\to\base$ be a valuation on a field $K$, and let $\Mon$ be a toric monoid. If $I\subseteq K[\Mon]$ 
%is a prime ideal such that $V(I)$ is not contained in the toric boundary
is an ideal whose variety has no component contained in the toric boundary,
then the image of every monomial of $\base[\Mon]$ in $\base[\Mon]/\Bend(\trop I)$ is cancellative. In particular, $\base[\Mon]/\Bend(\trop I)$ is cancellatively generated.
\end{prop}\begin{proof}
Note that the hypothesis on $I$ guarantees that $I$ is \emph{monomially saturated}: 
%if $m\in K[\Mon]$ is a monomial and $mf\in I$ then $f\in I$.
for any $u\in\Mon$, if $\chi^u f\in I$ then $f\in I$.

Fix a $u\in\Mon$ and suppose that $(\chi^u f, \chi^u g)\in\Bend(\trop I)$; we want to show $(f,g)\in\Bend(\trop I)$. If $(\chi^u f, \chi^u g)\in\Delta_{\base[\Mon]}$ then so is $(f,g)$ so we are done. So assume that $(\chi^u f, \chi^u g) \notin \Delta_{\base[\Mon]}$.

Let $\chi^u f=F_0\sim F_1\sim\cdots\sim F_n=\chi^u g$ be a chain of the form we are guaranteed exists by Lemma~\ref{lemma:strongMR}. By part~\eqref{item:coeff} of Lemma~\ref{lemma:strongMR}, 
%\red{of Lemma 7.20?} %% Indeed!
the support of each $F_i$ is contained in $\supp(\chi^u f)\cup \supp( \chi^u g)$, so $F_i$ is divisible by $\chi^u$. 

Lemma~\ref{lemma:strongMR}\eqref{item:bend}
%Part~\eqref{item:bend} \red{of Lemma 7.20?} %% Indeed!
and symmetry now give us that it is enough to show that, if $F_i\sim F_{i+1}$ is of the form $a\trop(h_i)+H\sim a(\trop(h_i))_{\what{\mu}}+H$ for some $h_i\in I$, $a\in\base^\times$, $\mu\in\supp(h_i)$, and $H\in\base[\Mon]$, then $\dfrac{F_i}{\chi^u}\sim\dfrac{F_{i+1}}{\chi^u}$ is in $\Bend(\trop I)$. 
Towards proving this, note that $\supp(F_i)=\supp(h_i)\cup\supp(H)$.
In particular, since $F_i$ is divisible by $\chi^u$, so are $h_i$ and $H$, so we have $\nbar{h}_i=\dfrac{h_i}{\chi^u}\in K[\Mon]$ and $\nbar{H}\in\base[\Mon]$. Moreover, because $I$ is monomially saturated, $\nbar{h}_i\in I$. 
Thus, we have $\nbar{h}_i\in I$, $a\in\base^\times$, $\mu-u\in\supp(\nbar{h}_i)$, and $\nbar{H}\in\base[\Mon]$ such that $\dfrac{F_i}{\chi^u}=a\trop(\nbar{h}_i)+\nbar{H}$ and $\dfrac{F_{i+1}}{\chi^u}=a\trop(\nbar{h}_i)_{\what{\mu-u}}+\nbar{H}$. So the relation $\dfrac{F_i}{\chi^u}\sim \dfrac{F_{i+1}}{\chi^u}$ is in $\Bend(\trop I)$.
\end{proof}

\begin{coro}\label{coro:ModBendVIntIsQInt}
Let $K$ be a field with valuation $K\to S$, let $\Mon$ a toric monoid, and let $I$ be an ideal of $K[\Mon]$ whose variety $Y := V(I)$ has no component contained in the toric boundary. 
Let $R$ be any localization of $\base[\Mon]/\Bend(\trop I)$ by a set of cancellative elements, and let $A$ be any sub-semiring of $R$. 
%Then $\vcl{A}=\wqcl{A}$.
Then $\vcl{A}=\sqcl{A}=\qcl{A}=\wqcl{A}$.
\end{coro}
\begin{proof}
Proposition~\ref{prop:QuotientByBendIsCancGend} gives us that $R$ is cancellatively generated, and Corollary~\ref{coro:trop-coordinate-semiring-is-cancellative - localized} tells us that $\sqrt{\Delta_R}$ is cancellative. Thus Corollary~\ref{coro:VIntIsSQInt} applies and gives the result.
\end{proof}

%\red{(transition verbiage needed.)}

We now give a few examples of computing the integral closures of cancellative semirings in their semifields of fractions. %We use the machinery and results developed in \cite{FM25a}.

\begin{example}
    Let $A = \T[x]/\sqrt{\Delta}$ be the semiring of tropical polynomial functions on $\T$ and let $R := \Frac(A)$. The semiring $R$ consists of all continuous piecewise-linear functions $\R\to\R$; see, for example, \cite{TW24}. Since $A$ is cancellative $R$ is a semifield and so, by 
    Proposition~\ref{prop:RCancellativeGivesAllIntClosuresEqual},
    all notions of integral closure coincide and agree with the downward closure of $A$. By the definition of the radical of a congruence, the primes on $A$ are in canonical bijection with the primes on $\T[x]$.
    Using \cite[Lemma 2.14]{FM23}, we have that each prime is given by one of the following four types of matrices:
    \begin{align*}
        %%%%%\mathcal{C} is reserved for congruences
        % \mathcal{C}_1 = \begin{pmatrix} 0 & a \end{pmatrix}, \quad \mathcal{C}_2 = \begin{pmatrix} 1 & a \end{pmatrix}, \quad 
        % \mathcal{C}_3 = \begin{pmatrix} 0 & a \\ 1 & b \end{pmatrix}, \quad \mathcal{C}_4 = \begin{pmatrix} 1 & a \\ b & c \end{pmatrix}
        C_1 = \begin{pmatrix} 0 & a \end{pmatrix}, \quad C_2 = \begin{pmatrix} 1 & b \end{pmatrix}, \quad 
        C_3 = \begin{pmatrix} 0 & a \\ 1 & b \end{pmatrix}, \quad C_4 = \begin{pmatrix} 1 & b \\ 0 & a \end{pmatrix},
    \end{align*}
    where $a=\pm1$ (or, in the case of $C_1$, $0$) and $b\in\R$. 
    We will denote by $P_i$ the prime congruence on $A$ given by any matrix of the form $C_i$.
    Note that $R_{P_2}$, $R_{P_3}$, and $R_{P_4}$ do not contain $A$, since they do not contain $\T$. More precisely, for any $c\geq 1$ we have that $|t^d|_{P_j} > 1_{\kappa(P_j)}$ for $j = 2, 3, 4$. We now focus on $P_1$. When $C_1 = \begin{pmatrix} 0 & a \end{pmatrix}$% with $a> 0$, 
    with $a=1$,
    then $R_{P_1}$ does not contain $A$ because $|x|_{P_1} > 1_{\kappa(P_1)}$. When $C_1 = \begin{pmatrix} 0 & 0 \end{pmatrix}$, then $R_{P_1} = R$. So we just need to 
    consider the prime $P$ corresponding to $\begin{pmatrix} 0 & -1 \end{pmatrix}$. The valuative integral closure of $A$ in $R$ is $R_{P} \cap R = R_{P}$.

    Note that, on $A=\T[x]/\sqrt{\Delta}$, $P$ is the congruence kernel of the semiring morphism $\T[x]/\sqrt{\Delta}\to\T$ given by $a\chi^u\mapsto t^{-u}$. %, which sends $f$ to $t^{-\mathrm{min.deg} f}$, where by $\text{min.deg}(f)$ we mean the minimal power to which $x$ appears in $f$.
    This morphism sends an arbitrary polynomial $f$ to $t^{-\mathrm{min.deg}(f)}$, where by $\mathrm{min.deg}(f)$ we mean the minimal power to which $x$ appears in $f$.
    %$t^{-\mathrm{min.deg} f+\mathrm{min.deg} g}\leq 1$\\
    %$\mathrm{min.deg} g\leq \mathrm{min.deg} f$
    We therefore have $R_P = \left\{ \frac{f}{g} \in R : \text{min.deg}(g) \leq \text{min.deg}(f) \right\}$.
    Thus a tropical rational function is in $R_P$ if and only if it has a representation as $\dfrac{f}{g}$ where $f,g\in\T[x]$ and $g$ has a nontrivial constant term. These are exactly the continuous piecewise-linear functions $\R\to\R$ that extend continuously to $\R\cup\{-\infty\}\to\R\cup\{-\infty\}$. These can be reasonably understood as the continuous piecewise-linear functions $\T\to\T$, i.e., tropical rational functions $\T\to\T$.
    %
    % First we observe that $R_P \neq A$ since $\frac{x}{x+1}$ is a piecewise linear function that is not a polynomial. More generally, $R_P = \{ \frac{f}{g} \in R : \text{min.deg}(f) < \text{min.deg}(g) \}$, where by $\text{min.deg}(f)$ we mean the minimal power to which $x$ appears in $f$. In particular, $R_P$ and so the valuative integral closure of $A$ in $R$ consists of piece-wise linear functions on $\T[x]$. 
%
    % We also note that in this case the valuative integral closure of $A$ in $R$ is the downwards closure of $A$ in $R$. %don't if I should write this, but on each interval a PL linear function from the integral closure is bounded by a monomial (look at the slope) with some coefficent, adjusted accodringly. The sum of all monomials with gives a polynomial (i.e. something in A) that is bigger than f. This also shows \D-int = \J-int closure.
    \exEnd
\end{example}

\begin{example}\label{ex:nodal-cubic-CPL}
     Let $f = x^2+x^3+y^2$ and $A = \T[x, y]/\sqrt{\Bend(f)}$ be and let $R := \Frac(A)$. Since $A$ is cancellative by Theorem~\ref{thm: trop-coordinate-semiring-is-cancellative} then $R$ is a semifield and so, by %Theorem~\ref{thm:intCloSemifields} 
     Proposition~\ref{prop:RCancellativeGivesAllIntClosuresEqual},
     all notions of integral closure coincide. 
    %  \green{
    %  A prime congruence $P$ of $\T[x,y]$ which contains $\Bend(f)$ and has the property that $R_P \supseteq A$ will have a defining matrix one of the following (up to rescaling):
    % \begin{align*}
    %     %%%%%\mathcal{C} is reserved for congruences
    %     % \mathcal{C}_0 = \begin{pmatrix} 0 & 0 & 0 \end{pmatrix}, \quad \mathcal{C}_1 = \begin{pmatrix}  0 & 0 & -1 \end{pmatrix}, \quad \mathcal{C}_2 = \begin{pmatrix}  0 & -1 & -1 \end{pmatrix}.
    %     C_0 = \begin{pmatrix} 0 & 0 & 0 \end{pmatrix}, \quad C_1 = \begin{pmatrix}  0 & 0 & -1 \end{pmatrix}, \quad C_2 = \begin{pmatrix}  0 & -1 & -1 \end{pmatrix}.
    % \end{align*}
    % %\red{ we do not include C = (0, 2, 3) because $R_P$ does not contain $A$}
    % In view of \cite[Theorem 3.22]{JM25} the only primes containing $\Bend(f)$ are defined by matrices of rank at most 1. 
    % We will denote by $P_i$ the prime congruence on $A$ that matrix $C_i$ gives rise to.
    % So the valuative integral closure of $A$ in $R$ is just $R_{P_1} \cap R_{P_2}$, since $R_{P_0} = R$.
    % }
    In view of \cite[Theorem 3.22]{JM25} the only primes on $\T[x,y]$ containing $\Bend(f)$ are defined by matrices with at most one row. Since the relevant primes must contain the relations $x^2+x^3+y^2 \sim x^2+x^3 \sim x^2+y^2 \sim x^3+y^2 $, this leaves only the primes defined by
    $$
    C_0 = \begin{pmatrix} 0 & 0 & 0 \end{pmatrix}, \quad C_1 = \begin{pmatrix}  0 & 0 & -1 \end{pmatrix}, \quad C_2 = \begin{pmatrix}  0 & -1 & -1 \end{pmatrix}, \text{ and } C_3=\begin{pmatrix}  0 & 2 & 3 \end{pmatrix}.
    $$
    Let $P_i$ denote the prime congruence on $A$ defined by the matrix $C_i$. Since $R_{P_3}\not\supseteq A$ and $R_{P_0}=R$, we have that $\vcl{A}=R_{P_1}\cap R_{P_2}$.
    By computing as in the previous example, we find that $R_{P_1}=\left\{\frac{f}{g}\,:\,\operatorname{min.deg}_{y}(g)\leq \operatorname{min.deg}_{y}(f)\right\}$ and $R_{P_2}=\left\{\frac{f}{g}\,:\,\operatorname{min.deg}(g)\leq \operatorname{min.deg}(f)\right\}$.
    Here $\operatorname{min.deg}_{y}(f)$ means the smallest degree to which $y$ is raised in a term of $f$. 
    Thus $\vcl{A}=\left\{\frac{f}{g}\,:\,\operatorname{min.deg}_{y}(g)\leq \operatorname{min.deg}_{y}(f)\ \&\ \operatorname{min.deg}(g)\leq \operatorname{min.deg}(f)\right\}$, which can be viewed as the set of all continuous piecewise-linear functions on $V(f)\subseteq\T^2$ with values in $\T$.
    %\red{(Unfortunately, I don't easily see that this is isomorphic to the version for the tropical bat signal.)}\blue{this should be Frac A (tropical rational functions on V(f)), with y/x adjoint}\red{The full Frac(A) doesn't see the points at infinity. I'm not sure what you mean by "with y/x adjoint".}
    So in $\vcl{A}\sdrop A$ we have elements such as $\frac{y}{x}$, $\frac{y}{x+y^2}$, and $\frac{x}{x^2+y}$.
    %
    %\red{Should we note that we get exactly the continuous PL functions on the tropical variety including points at infinity? \blue{yes} We could also relate this to the downward closure. \blue{no}}
    %Explicit computation shows $\vcl{A} = \{ \frac{f}{g} \in R : \text{min.deg of $y$ in } f \leq \text{min.deg of $y$ in } g\}$, where by $ \text{min.deg of $y$ in } f$ we mean the minimal power to which the variable $y$ appears in $f$. We can see that $\frac{x+y^2}{y}$ in $\vcl{A}\setminus A$. %\red{So the int closure is just PL linear functions on $A'[y]$, where $A' = \T[x]$ and has norm 0. }
    \exEnd
\end{example}

We now give examples of {non-cancellative pairs of} semirings that satisfy the hypotheses of 
%\green{
%Theorem~\ref{thm: A+=AJ-can gen}, that is, semirings for which the different notions of integral closure align. 
%}
% \orangutan{
% Theorem~\ref{thm: A+=AJ-can gen}. Thus, for these pairs of semirings, the different notions of integral closure all agree. 
% }
{
Theorem~\ref{thm: A+=AJ-can gen}. Thus, for these pairs of semirings, $\jcl{A}=\dcl{A}=\sqcl{A}=\qcl{A}=\wqcl{A}=\vcl{A}$. 
}
%\red{(Nati says: Technically this isn't true, because $\ncl{A}$ does not equal the others. One way to deal with this would be to say, just after defining the different types of integralities, that \n-int is not one that we are really interested, it just ends up being a useful concept in proving relationships between the other ones.)}\blue{sure, but it didn't exist when this was written -- just reword to make explicit what notions agree}

\begin{example}\label{ex: allIntegral-toric}
    Let $\base$ be a sub-semifield of $\T$ and let $\Mon$ be a toric monoid. 
    Then the hypotheses of Theorem~\ref{thm: A+=AJ-can gen} hold for $R = \base[\Mon]$ and any \dc\ $A\subseteq R$. To see this, first note that $R$ is cancellatively generated. The fact that $\sqrt{\Delta_R}$ is cancellative is the generalization of \cite[Theorem 4.9(vi), Theorem 4.14 (v)]{JM17} from polynomial semirings to toric semirings and can be deduced by replacing the matrices in the proof thereof with invoking Hahn's embedding theorem as in \cite[Theorem A.5]{FM23} and also follows immediately from \cite[Theorem 1.4]{FM25a}. Towards verifying the last condition of Theorem~\ref{thm: A+=AJ-can gen}, note that the \dc\ $A$-module generated by $f=\dsum_{i=1}^n s_ix^{u_i}$ with $s_i\in \base$ and $u_i\in\Mon$ is the module sum of the \dc\ modules generated by each of the $s_ix^{u_i}$'s. It is routine to check that the $A$-module generated by $s_ix^{u_i}$ is \dc, and so the \dc\ $A$-module generated by $f$ is generated as an $A$-module by $s_1x^{u_1},\ldots,s_nx^{u_n}$. \exEnd
\end{example}

\begin{example}\label{ex: allIntegral-terry}
Let $B$ be a principal ideal domain or, more generally, an integral domain in which every finitely generated ideal is principal\footnote{Such domains are known as \emph{B\'{e}zout domains}; see \cite{Kap71}, page 32.}. Let $K$ be the field of fractions of $B$ and let $C$ be a ring with $B\subseteq C\subseteq K$. 
%%%%% This imlies that every finitely generated $B$-submodule of $C$ is principal.
Let $A=\TerryPair{B}{B}$ denote the semiring of finitely generated 
%\green{$B$-submodules of $B$ }
ideals of $B$ 
and let $R=\TerryPair{C}{B}$ be the semiring of finitely generated $B$-submodules of $C$. We refer the reader to \cite[Definition 3.6]{Mac20} and the discussion following it for the fact that $A$ and $R$ are semirings.

We first show that $R$ is cancellative. 
%\red{Or we could say "it is routine to show that $R$ is cancellative."}
Indeed, say $r,s,t\in R$ with $r\neq 0$ are such that $rs=rt$. So there are $\rho,\sigma,\tau\in C$ with $\rho\neq0$ and $r=(\rho)_B$, $s=(\sigma)_B$ and $t=(\tau)_B$ where $(\bullet)_B$ denotes the $B$-module generated by $\bullet$. Since $(\rho\sigma)_B=rs=rt=(\rho\tau)_B$ and $C$ is an integral domain, there is a unit $u$ of $B$ such that $\rho\sigma=u\rho\tau$. Again using that $C$ is a domain we get $\sigma=u\tau$, so $s=(\sigma)_B=(\tau)_B=t$.

In particular, $R$ is cancellatively generated and $\sqrt{\Delta_R}=\Delta_R$ is cancellative.

Note that, for any $(\rho)_B,(\sigma)_B\in R$, $(\rho)_B\leq(\sigma)_B$ if and only if $(\rho)_B\subseteq(\sigma)_B$ and this, in turn, this happens if and only if $(\rho)_B$ is a multiple of $(\sigma)_B$ by something in $A$. Thus the downward $A$-module generated by $x\in R$ is also generated by $x$ as an $A$-module.
\exEnd
\end{example}

\section{Cancellative elements and the total semiring of fractions}\label{sec:cancellativeElements} %\red{we might want to combine this chapter with witness pairs}

In this section we specialize to computing the valuative integral closure $\vcl{A}$ of a semiring $A$ in its total semiring of fractions. Our main interest will be the case when $A$ is a quotient of the polynomial semiring by the bend congruence of a tropicalized (principal) ideal. Note that such $A$ is almost never cancellative and so $\Frac{A}$ is not a semifield. In order to form the total semiring of fractions of $A$, we need to first describe the set of cancellative elements in $A$; see Definition~\ref{def: semiring-of-fractions}.

Our first result is to show that when $\base $ is a sub-semifield of $\T$ and $\Mon$ is a toric monoid the only cancellative polynomials in $\base[\Mon]$ are monomials. While this may be intuitively obvious to experts (and possibly known to some experts), no proof is recorded in the literature, so we prove it here. We first need to introduce some notation.

\begin{definition}
Let $\base$ be a semifield. Let $\Mon$ be a toric monoid and let $M$ be the finitely generated free abelian group that $\Mon$ lives in; $M$ is naturally identified with the groupification of $\Mon$. Let $M_{\R}:=M\otimes_{\Z}\R$. The Newton polytope of $f \in \base[\Mon]$, denoted $\newt(f)$, is the convex hull in $M_{\R}$ of the exponent vectors of the monomials of $f$ in $\Mon$.

\end{definition}

\begin{lemma}\label{lem: f2notcancelThenfnotcan}
    % \green{
    % Let $\base$ be a sub-semifield of $\T$ and let $f \in \base[\Mon]$. If $f^2$ is not cancellative in $\base[\Mon]$ then neither is $f$. 
    % }

    Let $A$ be a semiring and let $f\in A$. If $f^2$ is not cancellative in $A$ then neither is $f$.
    
\end{lemma}

\begin{proof}
    Let $g, h \in A$ such that $g\neq h$ and $f^2g = f^2h$. If $fg = fh$ we are done. Now assume that is not the case, i.e., $fg \neq fh$. 
    %\green{By multiplying both sides by $f$ we get $f^2g \neq f^2h$ which is a contradiction, so we must have $fg = fh$.} 
    Then multiplying both sides by $f$ gives us $f(fg)=f^2g=f^2h=f(fh)$, so $f$ is not cancellative in this case either.
\end{proof}

\begin{lemma}\label{lem: specialfNotCancelS}
    Let $\base$ be a semifield. Let $f \in \base[\Mon]$ be such that there are three distinct monomials in the support of $f$ whose exponent vectors lay on a line. Then $f$ is not cancellative.
\end{lemma}
%Note: S does not need to be totally ordered, because we can always take in the inequalities a_m = b_m..
%And, if $n$ is an instance of $\mu$, $b_n=a_n+\frac{a_mb_m}{a_{n'}}$.

\begin{proof}
    We write monomials as $n,m,$ and $\mu$ and write their corresponding exponent vectors as $n_p,m_p,$ and $\mu_p$. By assumption there are three monomials $m,n',n''\in\supp(f)$ such that $m_p,n_p',$ and $n_p''$ lay on a line and $m_p$ is between $n_p'$ and $n_p''$. Without loss of generality, we may assume that the distance between $n_p'$ and $m_p$ is less than or equal to the distance between $m_p$ and $n_p''$. Let $n_p=2m_p-n_p'$ be the other point on this line such that the lattice distance between $m_p$ and $n_p$ is the same as the one between $m_p$ and $n_p'$. While $n$ may not be in $\supp(f)$, we do have $n_p\in\newt(f)$.
    
    Write $f=\dsum_{\mu\in\supp(f)}a_\mu\mu$ and let $g=\dsum_{\mu_p\in\newt(f)}b_\mu\mu$ be any polynomial (with support contained in $\newt(f)$) such that
    \begin{itemize}
    \item $b_\mu\geq a_\mu$ for all $\mu\in\supp(f)\sdrop\{m\}$,
    \item $b_m\leq a_m$, and
    \item $b_n\geq\dfrac{a_mb_m}{a_{n'}}$.
    \end{itemize}
    Note that this is possible by taking $b_m=a_m$ and $b_\mu=a_\mu+\dfrac{a_mb_m}{a_{n'}}$ for $\mu\in\{n\}\cup\supp(f)\sdrop\{m\}$.
    Define $h=\dsum_{\mu\in\supp(g)\sdrop\{m\}}b_\mu\mu$, so $h\neq g$ and $g=h+b_mm$. We will show that $fg=fh$ and so $f$ is not cancellative. Since $fg=f(h+b_mm)=fh+b_mmf$, it suffices to show that $b_mmf\leq fh$.

    Let $f'=\dsum_{\mu\in\supp(f)\sdrop\{m\}}a_\mu\mu$. Since $b_\mu\geq a_\mu$ for all $\mu\in\supp(f)\sdrop\{m\}$, we can write $h=f'+h'$ for some $h'$. Now $b_mmf=b_mm(f'+a_mm)=b_mmf'+a_mb_mm^2$ and 
    $$b_mmf'\leq a_mmf'\leq(f'+a_mm)(f'+h')=fh,$$ 
    so it suffices to show that $a_mb_mm^2\leq fh$. But $$a_mb_mm^2=a_mb_mn'n\leq a_{n'}b_nn'n=(a_n'n')(b_nn)\leq fh,$$ completing the proof.
\end{proof}

Recall that we use \emph{term} to refer to expressions of the form $a\chi^u$ while \emph{monomial} is reserved for the $\chi^u$s.

\begin{proposition}\label{prop: generalNotCancelS}
    Let $\base$ be a semifield and let $\Mon$ be a toric monoid. Then the cancellative elements of $\base[\Mon]$ are exactly the individual terms. 
\end{proposition}

\begin{proof}
    Let $m$ be a term and consider polynomials $g \neq h$. %If $g, h$ are monomials \red{terms?} and $mg = mh$ then $g = h$. %\green{By induction on the number of monomials of $g, h$ we can observe} 
    Straightforward expansion of $g$ and $h$ as sums of terms shows that if $mg = mh$ then $g = h$. 
    
    Now suppose $f$ is not a term. If $f \in \base[\Mon]$ meets the hypotheses of Lemma~\ref{lem: specialfNotCancelS} then we are done. %\green{Otherwise, we will replace $f$ with $f^2$. Note that $f^2$ meets all hypotheses of Lemma~\ref{lem: specialfNotCancelS}. In particular, there exists a monomial in $f^2$ whose exponent vector lies between two vertices of the newton polytope of $f^2$, namely the one obtained from the product of these two vertices. Now Lemma~\ref{lem: specialfNotCancelS} applies to $f^2$ showing that $f^2$ is not cancellative and by Lemma~\ref{lem: f2notcancelThenfnotcan} we see that $f$ is also not cancellative. }
    Otherwise, let $E$ be an edge of $\newt(f)$ and let $v_p$ and $v_p'$ be the two vertices of $E$. By assumption, there is no $u\in\supp(f)$ with $u_p$ on $E$ other than $v$ and $v'$. From this, it follows that $v^2$, $vv'$, and $(v')^2$ all occur in $\supp(f^2)$.  Since $2v_p$, $v_p+v_p'$, and $2v_p'$ are collinear, Lemma~\ref{lem: specialfNotCancelS} tells us that $f^2$ is not cancellative, so $f$ is not cancellative by Lemma~\ref{lem: f2notcancelThenfnotcan}.
\end{proof}

Proposition~\ref{prop: generalNotCancelS} immediately gives us the following corollary.
\begin{coro}\label{coro:fracOfToricSemiring}
Let $\base$ be a semifield and let $\Mon$ be a toric monoid whose torus has character lattice $M$. Then $\Frac(\base[\Mon])=\base[M]$.
\end{coro}

%\red{Question: can we compute A+ by only looking at 1-line primes? }
The results above already allow us to compute some examples of $\vcl{A}$.
Recall that, for $R$ a semiring and $P$ a prime congruence on $R$, we let $R_P:=\{x\in R\,:\,|x|_P\leq 1_{\kappa(P)}\}=\{x\in R\,:\,(x+1,1)\in P\}$.
%
%\red{Showing that $\base[\Mon]$ is normal, i.e., integrally closed in its total semiring of fractions.} %% Done!!!

The following example shows that when $\base$ is a sub-semifield of $\T$ and $\Mon$ is a toric monoid $\base[\Mon]$ then $\base[\Mon]$ is normal, i.e., integrally closed in its total semiring of fractions.

\begin{example}
Let $\base$ be a sub-semifield of $\T$ and let $\Mon$ be a toric monoid corresponding to a cone $\sigma$ in $N_{\R}$.\footnote{See \cite[Definition~2.13]{FM25a} for the specific definition and \cite{CLS} and \cite{Ful93} for more information about toric monoids, $M$, $N$, and $\sigma$.}
%\red{(We haven't formally introduced the correspondence between cones and toric monoids. But the details of that will be essential here.)}\blue{could we add a footnote or something to go read the details in our previous paper?} %% Done!
Let $A=\base[\Mon]$ and let $R=\Frac(\base[\Mon])$, which, by Corollary~\ref{coro:fracOfToricSemiring}, is $\base[M]$ where $M$ is the groupification of $\Mon$. Since $M$ itself is also a toric monoid, Example~\ref{ex: allIntegral-toric} tells us that the assumptions of Theorem~\ref{thm: A+=AJ-can gen} hold in this case, so $\Jcl{A} =\dcl{A}=\sqcl{A}=\qcl{A}=\wqcl{A}= \vcl{A}$.

%We now compute these integral closures. \red{(I don't love the wording of this past sentence, but I just wanted to put something on the page.)}\orange{how about "We compute the valuative integral closure, since it is the most straightforward."}
We compute the valuative integral closure, since it is the most straightforward. 
Given a point $v\in\sigma$, let $P^v$ be the congruence given by the matrix $\begin{pmatrix}0&v\end{pmatrix}$. That is, $P^v$ is the congruence kernel of the semiring morphism $\base[\Mon]\to\T$ that sends $a\chi^u$ to $t^{\angbra{v,u}}$. Since $\Mon=\{u\in M\,:\,\angbra{w,u}\leq 0_{\R}\text{ for all }w\in\sigma\}$, we have $\base[\Mon]\subseteq R_{P^v}=\base[\{u\in M\,:\,\angbra{v,u}\leq 0_{\R}\}]$. This also shows that $\dcap_{v\in\sigma}R_{P^v}=\base[\Mon]$. Thus, we have $\Jcl{A} =\dcl{A}=\sqcl{A}=\qcl{A}=\wqcl{A}= \vcl{A}=\base[\Mon]$.
\exEnd
\end{example}

\begin{example}\label{ex: cuspidalCubic} % A_1, A are now A and B
    %Let $\CR$ be the coordinate semiring of the cuspidal cubic, $\CR=S[x^2, x^3]$ and let $\FCR = \Frac(\CR)= S[x^{\pm 1}]$. 
    Let $S$ be a sub-semifield of $\T$, let $\CR=S[x^2,x^3]$ be the coordinate semiring of the cuspidal cubic, and let $\FCR := \Frac(\CR)= S[x^{\pm 1}]$.
    Every prime congruence on $\FCR$ is defined by a $1\times2$ or $2\times 2$ matrix. 
    %Up to re-scaling \red{perhaps also in view of downwards row reduction? \cite[Lemma 2.12]{FM23}} we have the following options for defining matrix of $P$:
    By using the downward row reduction of \cite[Lemma 2.12]{FM23}, we see that every $P$ is defined by one of the matrices
    \begin{center}
        $\begin{pmatrix}0 & 0 \end{pmatrix}$, $\begin{pmatrix}0 & 1 \end{pmatrix}$, $\begin{pmatrix}0 & -1 \end{pmatrix}$, $\begin{pmatrix}1 & a \end{pmatrix}$, $\begin{pmatrix}0 & \pm{1} \\ 1 & 0 \end{pmatrix}$, $\begin{pmatrix}1 & a \\ 0 & a' \end{pmatrix}$
    \end{center}
    for $a,a'\in\R$. For each of the cases we will determine $\FCR_P$ if $\CR \subseteq \FCR_P$.
%\end{example}

% \noindent (1) If the defining matrix for $P$ is $\begin{pmatrix}0 & 0 \end{pmatrix}$, then we have $A_P = \FCR$.

% \noindent (2) If the defining matrix for $P$ is $\begin{pmatrix}0 & 1 \end{pmatrix}$, then $|x^2|_P = 2 > 1_{\kappa(P)}$ so we have that $A_P \not\supseteq \CR$.

% \noindent (3) If the defining matrix for $P$ is $\begin{pmatrix}0 & -1 \end{pmatrix}$, then $|x^n|_P = -n \leq 1_{\kappa(P)}$ and $|t^c|_P = 1_{\kappa(P)}$, for all $t^c\in S$, so $\CR \subseteq A_P$. To compute $A_P$, 
% \begin{align*}
%     A_P &= \{ f = \sum_{n \in \mathbb{Z}} f_n x^n : |f|_P \leq 1_{\kappa(P)}\} \\
%         &= \{ f = \sum_{n \in \mathbb{Z}} f_n x^n : |f_n x^n|_P \leq 1_{\kappa(P)}, \forall n\} \\
%         &= \{ f = \sum_{n \in \mathbb{Z}} f_n x^n : |f_n|_P = 1_{\kappa(P)} \text{ or } n\geq 0\} \\
%         &= S[x].
% \end{align*}

% \noindent (4) If the defining matrix for $P$ is $\begin{pmatrix}1 & a \end{pmatrix}$, then $|t^c|_P = c$ for each $t^c\in S$ we have that $A_P \not\supseteq \CR$.

% \noindent (5) If the defining matrix for $P$ is $\begin{pmatrix}0 & \pm{1} \\ 1 & 0 \end{pmatrix}$, then $|t^c|_P = \begin{pmatrix} 0 \\ c\end{pmatrix}$ we have that $t^c\in S$ and so $A_P \not\supseteq \CR$.

% \noindent (6) If the defining matrix for $P$ is $\begin{pmatrix}1 & a \\ 0 & a' \end{pmatrix}$ then similarly to the previous case $A_P \not\supseteq \CR$.
% Thus $\vcl{\CR} = B \cap S[x] = S[x]$.\exEnd
%
\begin{enumerate}
%\noindent(1) 
\item
If $P$ is defined by $\begin{pmatrix}0 & 0 \end{pmatrix}$, then we have $\FCR_P = \FCR$.

%\noindent (2) 
\item
If $P$ is defined by $\begin{pmatrix}0 & 1 \end{pmatrix}$, then 
%$|x^2|_P = 2 > 1_{\kappa(P)}$ 
$|x^2|_P$ corresponds to $\begin{pmatrix}0 & 1 \end{pmatrix}\begin{pmatrix}
0\\2
\end{pmatrix}
=\begin{pmatrix}
2
\end{pmatrix}>_{\text{lex}}\begin{pmatrix} 0 \end{pmatrix}$, so $|x^2|_P>1_{\kappa(P)}$. Thus, $\FCR_P \not\supseteq \CR$.

%\noindent (3) 
\item
If $P$ is defined by $\begin{pmatrix}0 & -1 \end{pmatrix}$, then %$|x^n|_P = -n \leq 1_{\kappa(P)}$ and $|t^c|_P = 1_{\kappa(P)}$, for all $t^c\in S$, so $\CR \subseteq A_P$. To compute $A_P$, 
\begin{align*}
    \FCR_P &= \{ f = \sum_{n \in \mathbb{Z}} f_n x^n \,:\, |f|_P \leq 1_{\kappa(P)}\} \\
        &= \{ f = \sum_{n \in \mathbb{Z}} f_n x^n \,:\, |f_n x^n|_P \leq 1_{\kappa(P)}, \forall n\} \\
        &= \{ f = \sum_{n \in \mathbb{Z}} f_n x^n \,:\, f_n = 0_S \text{ or } n\geq 0_{\Z}\} \\
        &= S[x].
\end{align*}

%\noindent (4) 
\item
If $P$ is defined by $\begin{pmatrix}1 & a \end{pmatrix}$, then $|t^c|_P$ is given by $\begin{pmatrix}c\end{pmatrix}$ for each $t^c\in S$. So we have that $\FCR_P \not\supseteq \CR$.

%\noindent (5) 
\item
If $P$ is defined by $\begin{pmatrix}0 & \pm{1} \\ 1 & 0 \end{pmatrix}$, then $|t^c|_P$ is given by $\begin{pmatrix} 0 \\ c\end{pmatrix}$ for all $t^c\in S$ and so $R_P \not\supseteq \CR$.

%\noindent (6) 
\item
If $P$ is defined by $\begin{pmatrix}1 & a \\ 0 & a' \end{pmatrix}$ then similarly to the previous case $R_P \not\supseteq \CR$.
\end{enumerate}
Thus $\vcl{\CR} = \FCR \cap S[x] = S[x]$. In fact, using Example~\ref{ex: allIntegral-toric} and Theorem~\ref{thm: A+=AJ-can gen}, we get that $\jcl{A}=\dcl{A}=\sqcl{A}=\qcl{A}=\wqcl{A}=\vcl{A}=S[x]$.
\exEnd

\end{example}
%
%\smallskip 

% \green{
% Our next goal is to investigate the cancellative elements in $A = \base[x_1, \dots, x_n]/\Bend (J)$. More specifically, we will focus on the "principal" case, that is, let $K$ be a field and $\Mon$ a toric monoid. Given any principal ideal $I = (f')$ of $K[\Mon]$, let $\trop f' = f$ and let $J = \trop I$. 
% }

Our next goal is to investigate the cancellative elements in $A = \base[\Mon]/\Bend (J)$. We restrict our attention to the case of the tropicalization of a principal ideal. That is, we let $K$ be a field with a valuation $K\to\base$ and let $\Mon$ a toric monoid. Given any principal ideal $I = (f')$ of $K[\Mon]$, let $\trop f' = f$ and let $J = \trop I$. 

% \green{
% By Proposition~\ref{prop: generalNotCancelS} we know that there exists a non-diagonal pair of polynomials $(g,h)$, such that $f(g,h) \in \Delta$, however this pair can easily be contained in $\Bend(J)$ as shown in the next example.}

By Proposition~\ref{prop: generalNotCancelS} we know that there exists a pair $(g,h)\notin\Delta_{\base[\Mon]}$, such that $f(g,h) \in \Delta_{\base[\Mon]}$. However $(g,h)$ can easily be contained in $\Bend(J)$ as shown in the next example.

\begin{example}
    % \green{
    % Let $f=x^2+xy+y^2$, then $g = t^2x^2+t^2y^2 + xy$ and $h = t^2x^2+t^2y^2$. We can see that $(g,h) \in \Bend(f)$ since
    % %$$g = t^2x^2+t^2y^2+t^{-2}(f) \sim t^2x^2+t^2y^2+t^{-2}(x^2+y^2) = h+t^{-2}(x^2+y^2)=h.$$ 
    % $$g = t^2x^2+t^2y^2+f \sim t^2x^2+t^2y^2+(x^2+y^2) = t^2x^2+t^2y^2 = h.$$
    % \red{Nati is up to here} \blue{the previous paragraph already introduces this example}
    % This is true for all coefficients of $xy$ and $x^2, y^2$ that meet the conditions of Lemma~\ref{lem: specialfNotCancelS}. %This phenomenon happens when $f \geq F$, as then $g = h + g_mF \sim hF_{\hat{v}} + g_mF\subseteq Bend(F)$. \exEnd
    % }
    Let $f=x^2+xy+y^2$. As in the proof of Lemma~\ref{lem: specialfNotCancelS}, if we let $g=ax^2+bxy+cy^2$ and $h=ax^2+cy^2$ with $a\geq1$, $c\geq 1$, and $b\leq 1$ then $fg=fh$. That said, all such pairs $(g,h)$ are in $\Bend(f)$ because 
    $$g = ax^2+cy^2+bf \sim ax^2+cy^2+b(x^2+y^2) = ax^2+cy^2 = h.$$
    \par\nopagebreak\vspace{-1.5\baselineskip}\mbox{}
    \exEnd 
\end{example}

The pairs $(g,h)$ constructed in Lemma~\ref{lem: specialfNotCancelS} and Proposition~\ref{prop: generalNotCancelS} are not the only non-diagonal pairs with the property $fg=fh$. We illustrate this in the following two examples.

\begin{example}\label{ex: allWitnessPairs}%\blue{this example has been checked with M2 as well}
\ 
    \begin{enumerate}
        \item Let $f = 1+x+y$, $h = 1 + x^2 + xy$, and $g = 1 + x + x^2 + xy$. Direct computation shows that $fh = fg$.
        \item Let $f = 1+x+x^2+y^2$, $h_1 = 1 + x^3 + y^2$, and $g_1 = 1 + x + x^3 + y^2$. Then $fh_1 = fg_1$. Alternatively, we can consider $h_2 = 1 + x^3 + xy^2$, and $g_2 = 1 + x + x^3 + xy^2$ for which we have $fh_2=fg_2$.
%        \item Let $f = x^2+xy+y^2$, then $f = (x+y)^2$. The taking polynomials $h = x$ and $g=y$  -- \red{this is not what we want, what should it be? Probably we should not worry about these, because the quotient by the bend of binomial is some S[M] and that's not cancellative.}
    \item Let $f = x^2+x^3+y^2$. We can consider $f^2$ and obtain a pair $(g,h)$ which multiplies $f^2$ into the diagonal.  Lemma~\ref{lem: f2notcancelThenfnotcan} ensures that $f$ is not cancellative either. Take $h = x^4 + x^6 + y^4$ and $g = h + x^5$. Direct computation shows that $hf^2 = gf^2$. 
    % Notably, none of the pairs $(g, h)$ above are in the congruence $\Bend(f)$. \red{is this true? do we mean $\bend(f)?$} 
    Now let $K$ be a trivially valued field, $f' = x^2 + x^3 - y^2 \in K[x,y]$ and let $I = (f')\subseteq K[x,y]$. Note that $f = \trop f'$ and that $(g, h) \in \Bend (\trop I)$ because $g = \trop(f'(x^2 + x^3 + y^2))$. However, if we take $h = x^4 + x^6 + x^3y^2$ and $g = h + x^5$ we get a pair that is not in $\Bend (\trop I)$. %- see Nati's notes showing this is not in the bend congruence.
    %However, note that $f = \trop (x^2 + x^3 - y^2)$, then $(g, h) \in \Bend (\trop \left< x^2 + x^3 - y^2\right>)$ because $g = \trop((x^2 + x^3 - y^2)(x^2 + x^3 + y^2))$. Still, if we take $h = x^4 + x^6 + x^3y^2$ and $g = h + x^5$ we get a pair that is not in $\Bend (\trop \left< x^2 + x^3 - y^2\right>)$.  %- see Nati's notes showing this is not in the bend congruence.
    \item We can also find a pair that works directly for $f = x^2+x^3+y^2$, without having to consider $f^2$. Take $h = x^5+xy^2+y^4$ and $g = h+x^2y^2$ we see that $gf=hf$. \exEnd
    \end{enumerate} 
\end{example}

For a fixed polynomial $f \in A$ we present a general algorithm for constructing non-diagonal pairs $(g,h)$, that is, pairs $(g,h)\notin\Delta_{\base[\Mon]}$, such that $fg=fh$ in Section~\ref{sec: witness pairs}.

%Based on the examples we have computed above we make the following conjectures.
Based on computing examples, including the ones above, we make the following conjecture.

% \green{
% \begin{conjecture}%\label{conj: cancellatives}
% % \green{
% % Let $K$ be a field and $\Mon$ a toric monoid. Given any principal ideal $I \subseteq K[\Mon]$ let $J = \trop I \subseteq \base[\Mon]$, where $\base $ is a sub-semifield of $\T$. We conjecture that the only cancellative polynomials in $\base[x_1, \dots, x_n]/Bend (J)$ are monomials.
% % }

% Let $K$ be a field with valuation $K\to\base$ where $\base$ is a sub-semifield of $\T$ and let $\Mon$ a toric monoid. Given any principal (prime) ideal $I \subseteq K[\Mon]$ that does not contain a monomial or binomial%
% %
% \footnote{
% If $I$ is prime and contains a monomial $m$ then $V(I)$ is contained in the closure of a toric boundary stratum. If $I$ contains a binomial then $V(I)$ is contained in the closure of (a translate of) a proper subtorus. Either way, $V(I)$ is best viewed as being embedded in a different toric variety and so $\base[\Mon]/\Bend(\trop I)$ is best given by a different presentation.
% }%
% \footnote{
% From a congruence perspective, these assumptions mean that terms are not being set equal to $0$ or polynomials with more than one term. Thus one can still sensibly talk about terms in the quotient $\base[\Mon]/\Bend(J)$.
% },
% %
% let $J = \trop I \subseteq \base[\Mon]$. Then any element of $\base[\Mon]/\Bend(J)$ that is cancellative is the image of a term of $\base[\Mon]$.

% \end{conjecture}
% }% end of green

\begin{conjecture}\label{conj: cancellatives}

Let $K$ be a field with valuation $K\to\base$, where $\base$ is a sub-semifield of $\T$ and let $\Mon$ a toric monoid. Given any principal prime ideal $I \subseteq K[\Mon]$ that does not contain a monomial or binomial%
\footnote{
If $I$ is prime and contains a monomial $m$ then $V(I)$ is contained in the closure of a toric boundary stratum. If $I$ contains a binomial then $V(I)$ is contained in the closure of (a translate of) a proper subtorus. Either way, $V(I)$ is best viewed as being embedded in a different toric variety and so $\base[\Mon]/\Bend(\trop I)$ is best given by a different presentation.
}%
\footnote{
From a congruence perspective, these assumptions mean that terms are not being set equal to $0$ or polynomials with more than one term. Thus one can still sensibly talk about terms in the quotient $\base[\Mon]/\Bend(J)$.
},
let $J = \trop I \subseteq \base[\Mon]$. Then an element of $\base[\Mon]/\Bend(J)$ is cancellative if and only if it is the image of a term of $\base[\Mon]$.
\end{conjecture}

The ($\Leftarrow$) direction follows from Corollary~\ref{coro:ModBendVIntIsQInt}. So the open content of the conjecture is the ($\Rightarrow$) direction.
% end of pinky

\begin{remark}\label{rmk:varietyNormalization}
Let $K$, $\base$, $\Mon$, and $I$ be as in Conjecture~\ref{conj: cancellatives}. Since $I$ is prime $V(I)$ has only one component, and since $I$ does not contain a monomial, that component is not contained in the toric boundary. Thus Corollary~\ref{coro:ModBendVIntIsQInt} tells us that, if we let $A=\base[\Mon]/\Bend(\trop I)$ and $R=\Frac(A)$, then $\vcl{A}=\sqcl{A}=\qcl{A}=\wqcl{A}$.
\end{remark}

If we assume Conjecture~\ref{conj: cancellatives} then we can compute $\vcl{A}$ in concrete cases.

%\red{(Corollary~\ref{coro:ModBendVIntIsQInt} tells us that the normalization can be viewed as any of these kinds of integral closures. We should probably note this later, maybe just before Example~\ref{ex: nodalCubic-2}? Nati says: I've attempted this above in Remark~\ref{rmk:varietyNormalization}. Let me know what you think.)}\blue{looks great!}

\begin{example}\label{ex: nodalCubic-2}
    % \green{If we assume Conjecture~\ref{conj: cancellatives} then we can compute $\vcl{A}$.}
    % \green{Let $K$ be a field, $f' = x^2 + x^3 - y^2 \in K[x,y]$ which generates an ideal $I = (f')\subseteq K[x,y]$, and let $J = \trop I$. }
    Let $K$ be a field, let $f' = x^2 + x^3 - y^2 \in K[x,y]$, let $I:= (f')\subseteq K[x,y]$, and let $J = \trop I$.
    %\green{Then $f = x^2+x^3+y^2$, $A = \T[x, y]/\Bend(J)$, and $R = \Frac(A) = \T[x^{\pm}, y^{\pm}]/\Bend(J)$. }
    If we let $A = \T[x, y]/\Bend(J)$, then Conjecture~\ref{conj: cancellatives} would give us that $R = \Frac(A)$ is $ = \T[x^{\pm}, y^{\pm}]/\Bend(J)$.
%    \green{Using the computations in the Example~\ref{ex:nodal-cubic-CPL} we conclude that $$\vcl{A} = R_{P_1} \cap R_{P_2} = \T[x^n y^m : m\geq 0, m+n \geq 0]/\Bend(J) = A\big[\frac{y}{x}\big].$$}
    The same computations as in Example~\ref{ex:nodal-cubic-CPL} would then show that $$\vcl{A} = R_{P_1} \cap R_{P_2} = \T[x^n y^m : m\geq 0, m+n \geq 0]/\Bend(J) = A\big[\frac{y}{x}\big].$$
%    \green{Recall that the nodal cubic with equation $x^2-x^3+y^2$ and coordinate ring $R$ has normalization with coordinate ring $R[\frac{y}{x}]$.} 
    Recall that the nodal cubic with equation $x^2-x^3+y^2=0$ and coordinate ring $B$ has normalization with coordinate ring $B[\frac{y}{x}]$.  
    \exEnd
\end{example}

% \green{
% Based on the examples above we arrive to our second conjecture, which can be verified only after Conjecture~\ref{conj: cancellatives} is addressed. 

% \begin{conjecture}\label{conj: normalization}
%     Let $C$ be an (embedded) algebraic curve with coordinate ring $B = k[x_1, \dots, x_n]/I$ and let $A = S[x_1, \dots, x_n]/\Bend(trop(I))$ be the tropicalization of $R$. If $R=\Frac(A)$, then $\vcl{A}$, the valuative closure of $A$ in $R$, is the tropicalization of the normalization of $B$ (after choosing an appropriate embedding).
% \end{conjecture}
% }

Based on Examples~\ref{ex: cuspidalCubic} and \ref{ex: nodalCubic-2} we have the following conjecture.

\begin{conj}
Let $C$ be an irreducible affine algebraic curve with one singular point. Then there is an embedding of $C$ with ideal $I\subseteq K[x_1,\ldots,x_n]$ satisfying the following property. Let $A=\base[x_1,\ldots,x_n]/\Bend(\trop I)$, let $R=\Frac(A)$, and let $\vcl{A}$ be the valuative integral closure of $A$ in $R$. Then the integral closure of $B=K[x_1,\ldots,x_n]/I$ in its field of fractions is obtained by adjoining to $B$ lifts of the elements in $\vcl{A}\sdrop A$.
\end{conj}

If this conjecture is true, it would naturally bring up the following question.

\begin{question}
Let $C$ be an arbitrary algebraic curve. Is there an embedding of $C$ in a toric variety such that the normalization of $C$ can be seen at a tropical level by looking in each torus-invariant open affine?
\end{question}

\section{Witness Pairs}\label{sec: witness pairs}

%\blue{comments sent by email, sorry!!!}

Let $\base$ be a totally ordered semifield and let $\Mon$ be a cancellative monoid. For this section we fix an $f\in\base[\Mon]$ with more than one term. 

\begin{defi}
A pair $(g,h)$ with $g,h\in\base[\Mon]$ is a \emph{witness pair} for $f$ if $g\neq h$ but $fg=fh$.
\end{defi}

In this section, we give an algorithm that can be used to construct all witness pairs for $f$. We start by giving several reductions.

A witness pair $(g,h)$ for $f$ is \emph{increasing} if $g\leq h$ and is \emph{decreasing} if $g\geq h$.

\begin{lemma}\label{lemma:ReduceToMonotone}
Let $(g,h)$ be a witness pair for $f$. Then either $(g,h)$ is increasing, $(g,h)$ is decreasing, or there is a $\ph\in\base[\Mon]$ such that $(g,\ph)$ and $(\ph,h)$ are witness pairs for $f$ with $(g,\ph)$ increasing and $(\ph,h)$ decreasing.
\end{lemma}
\begin{proof}
Suppose that $(g,h)$ is neither increasing nor decreasing. 
%Then $\ph:=g+h$, $g$, and $h$ are all distinct. 
Then $g$, $h$, and $\ph:=g+h$, are all distinct. 
Since $fg=f\ph=fh$ and $g\leq \ph\geq h$, the pairs $(g,\ph)$ and $(\ph,h)$ are as desired.
\end{proof}

%In light of Lemma~\ref{lemma:ReduceToMonotone}, it suffices to construct all decreasing witness pairs for $f$. 
Lemma~\ref{lemma:ReduceToMonotone} tells us that an arbitrary witness pair has a transitive decomposition into an increasing witness pair, a decreasing witness pair, or both. Thus, due to symmetry, it suffices to construct all decreasing witness pairs for $f$. 
For our next reduction, we introduce the following notation. Given $g=\dsum_{u\in\Mon}g_u\chi^u$ and $h=\dsum_{u\in\Mon}h_u\chi^u$ in $\base[\Mon]$, we let $\delta(g,h)$ be the number of $u\in\Mon$ for which $g_u\neq h_u$.

\begin{lemma}\label{lemma:reduceToOneCoeff}
Let $(G,H)$ be a decreasing witness pair for $f$. Then there are decreasing witness pairs $(g_1,h_1), (g_2,h_2), \ldots, (g_n,h_n)$ for $f$ with $G=g_1$, $h_{i}=g_{i+1}$ for $1\leq i\leq n-1$, and $h_n=H$, where $\delta(g_i,h_i)=1$ for $1\leq i\leq n$ and $n=\delta(G,H)$
\end{lemma}\begin{proof}
Since $G\geq H$ and $\delta(G,H)=n$, we can write $G=H+\dsum_{k=1}^n a_k\chi^{u_k}$. For $1\leq i\leq n$, let $g_i=H+\dsum_{k=1}^{n+1-i}a_k\chi^{u_k}$ and $h_i=H+\dsum_{k=1}^{n-i}a_k\chi^{u_k}$. It is then routine to check that the pairs $(g_i,h_i)$ are as claimed.
%For the fact that they are witness pairs, we have $G\geq g_i\geq h_i\geq H$; mutiplying by $f$ gives $fG\geq fg_i\geq fh_i\geq fH=fG$.
\end{proof}

Note that, if $(g,h)$ is a decreasing witness pair with $\delta(g,h)=1$ then there is a $\ph\in\base[\Mon]$, a monomial $m$ not in the support of $\ph$, and $a>b$ in $\base$ such that $g=\ph+am$ and $h=\ph+bm$. As such, the following lemma gives us a further reduction.

% \blue{We do define the support of $f\in\base[\Mon]$ and this notation in the preliminaries and use it a few times throughout the paper. I have now made it a little bit more clear; see the watermelon in the preliminaries.} \red{K: looks good!}

\begin{lemma}\label{lemma:ReduceToWitnessPairs}
Let $h\in\base[\Mon]$, let $m$ be a monomial not in $\supp h$, and let $a>b$ in $\base$. Then $(h+am,h+bm)$ is a witness pair for $f$ if and only if $(h+am,h)$ is.
\end{lemma}\begin{proof}
If $(h+am,h)$ is a witness pair for $f$ then $f(h+bm)=fh+fbm=f(h+am)+fbm=f(h+am+bm)=f(h+am)$, so $(h+am,h+bm)$ is also a witness pair.

For the remainder of this proof, given a polynomial $\ph$ and $u\in\Mon$ we let $\ph_u$ denote the coefficient of $\chi^u$ in $\ph$.

Now suppose that $(h+am,h+bm)$ is a witness pair for $f$. 
Then $fam\leq f(h+am)=f(h+bm)$ so, for each $u\in\Mon$, $(fam)_u$ is at most the sum of $(fh)_u$ and $(fbm)_u$. 
Because $\base$ is totally ordered, this means that $(fam)_u$ is less than or equal to one of $(fh)_u$ or $(fbm)_u$
Since $b<a$ and $m$ is a monomial, $(fam)_u$ is either $0_S$ or strictly greater than $(fbm)_u$. 
Hence $(fh)_u$ must be at least $(fam)_u$. 
This being true for all $u\in\Mon$ exactly says that $fh\geq fam$, i.e., $fh+fam=fh$.
Thus, $(h+am,h)$ is a witness pair for $f$.
\end{proof}

Since scaling by a nonzero $a\in\base$ doesn't affect whether a pair is a witness pair, we now only need to consider witness pairs of the form $(h+m,h)$ where $m$ is a monomial not in the support of $h$.

\begin{defi}
Let $m$ be a monomial of $\base[\Mon]$. An \emph{$m$-witness pair for $f$} is a witness pair $(g,h)$ 
%where 
such that 
$g=h+m$ and $m\not\in\supp h$. 
\end{defi}

\begin{example}
   In Example~\ref{ex: allWitnessPairs} (1) the pair $(g,h)$ is a $m$-witness pair for $f$ with $m = x$.% \red{(Should we consider leaving out this example?)} \blue{I am fine either way, I thought it might be nice to relate these definitions to the multiple examples we gave in the previous section}
   \exEnd
\end{example}

We now give a construction of $m$-witness pairs.

\begin{construction}\label{constr: pairsGH}
    %Let 
    Fix a monomial $m$ and let 
    $f_1,\ldots,f_k$ be the terms of $f$, so $f = \dsum_{i =1}^k f_i$. For each $i$, suppose that there is a $j_{i}\neq i$ and a term $h_i$ such that $mf_i\leq f_{j_i}h_i$
    Then let $h=\dsum_{i=1}^k h_i$ and let $g=h+m$. The output of this construction is the pair $(g,h)$.
    
    If this can be done then we say that $m$ is a \emph{\MissMon} for $f$. 
\end{construction}

The following lemma is elementary.

\begin{lemma}\label{lemma:ConstructionWitnesses}
Any pair $(g,h)$ that arises from Construction~\ref{constr: pairsGH} is an $m$-witness for $f$.
\end{lemma}
% \begin{proof}
% We have $fm=\dsum_{i=1}^k mf_i\leq\dsum_{i=1}^k f_{j_i}h_i\leq\dsum_{i=1}^k fh=fh$ so $fh=fh+fm=fg$. Since $j_i\neq i$, $h_i$ cannot be a constant multiple of $m$, so $m\notin\supp h$. \red{(Or we could say "The following lemma is elementary." and not give a proof.)} \blue{sure, we can omit the proof (just comment it out in case we need to bring it back)}
% \end{proof}

Our next lemma provides a partial converse to Lemma~\ref{lemma:ConstructionWitnesses}.

%\newcommand{\diagEl}{\delta}
%(\diagEl,\diagEl) is an element of the diagonal. Macoring this because I started using $\delta$ to mean a different thing.
\newcommand{\diagEl}{\eta}
\newcommand{\diagel}{\diagEl}
\newcommand{\dgel}{\diagEl}
\begin{lemma}\label{lem: allm-witness}
A pair is an $m$-witness for $f$ if and only if it can be written in the form $(g,h)+(\dgel,\dgel)$, where $(g,h)$ arises from Construction~\ref{constr: pairsGH} and $\dgel\in\base[\Mon]$ with $m\notin\supp(\dgel)$.
\end{lemma}\begin{proof}
If $(g,h)$ arises from Construction~\ref{constr: pairsGH}, then Lemma~\ref{lemma:ConstructionWitnesses} tells us that $fg=fh$ so $fg+f\dgel=fh+f\dgel$. Since $m\not\in\supp\dgel$, $g+\dgel\neq h+\dgel$ so $(g,h)+(\dgel,\dgel)$ is an $m$-witness pair.

For the other direction, let $(G, H)$ be an $m$-witness pair for $f$ %, i.e., $Gf=Hf$ and $G = H+m$. 
and let $f_1,\ldots,f_k$ be the terms of $f$ so $f=\dsum_{i=1}^k f_i$. 
Since $Hf=Gf=Hf+mf$, we have $mf \leq Hf$. 
In particular, for each term $f_i$ of $f$, there are terms $f_{j_i}$ of $f$ and $h_{i}$ of $h$ such that $mf_i\leq h_if_{j_i}$.
Note that, because $m\notin\supp H$, we must have $j_i\neq i$, for $mf_i\leq h_if_i$ and the cancellativity of monomials would give $m\leq h_i\leq H$.
Let $\mathfrak{T}$ be the set of those terms of $H$ that do not occur as an $h_i$, let $h=\dsum_{i=1}^k h_i$, and let $\diagel=\dsum_{\tau\in\mathfrak{T}}\tau$.
Letting $g=h+m$ we have that $(G,H)=(g,h)+(\diagel,\diagel)$ and that $(g,h)$ arose from Construction~\ref{constr: pairsGH}.
Since $\diagel\leq H$, $m\notin\supp\diagel$.
\end{proof}

As an immediate consequence of Lemma~\ref{lem: allm-witness}, we have the following.

\begin{coro}
A monomial $m$ is a \missmon\ for $f$ if and only if there is an $m$-witness pair for $f$.
\end{coro}

% \green{
% A natural question is whether there are many monomials that are missing monomials for $f$. This is the subject of the first part of our next result. First, we need a definition.

% \begin{defi}
% We say that a property $\calP$ of monomials holds for $m$ \emph{sufficiently divisible} if there is an monomial $m_0$ such that, whenever $m_0|m$, $\calP$ holds for $m$.
% \end{defi}

% \begin{lemma}\label{lemma:sufficientlyDivisibleMonomials}
% Any sufficiently divisible monomial $m$ is a \MissMon\  for $f$. Moreover, if $m$ and $m'$ are sufficiently divisible monomials then the pairs $(g',h')$ that arise from Construction~\ref{constr: pairsGH} for $m'$ are exactly those pairs of the form $\dfrac{m'}{m}(g,h)$ where $(g,h)$ is a pair that arises from Construction~\ref{constr: pairsGH} for $m$.
% \end{lemma}\begin{proof}
% Let $f_1,\ldots,f_k$ be the terms of $f$ and write $f_i=a_i\chi^{u_i}$. Let $m_0=\dprod_{i=1}^k \chi^{u_i}$ and consider any monomial $m$ that is a multiple of $m_0$. 
% Then for any $j_i\neq i$ we have $f_{j_i}|mf_i$. So we can let $h_i=\dfrac{mf_i}{f_{j_i}}$, showing that $m$ is a \missmon\ for $f$. Moreover, the terms that work as $h_i$ are exactly those of the form $a\dfrac{mf_i}{f_{j_i}}$ where $a\geq 1$ is in $\base$, which implies the second claim.
% %\red{Prove this!} If $m$ is sufficiently divisible, then there exists $f_j $ which divides $m$ and so $f_j | mf_i$ and $ mf_i \leq f_j h_i$ for some $h_i$.
% \end{proof}
% }

%%%%%%

It might seem that (running this computation) would have to be done for every monomial and so would not be computationally effective. Our next proposition tells us that we only need to do a finite computation. Recall that we have fixed $f\in\base[\Mon]$ with more than one term.

\newcommand{\Gp}{M}
\begin{prop}
Let $m_0$ be the product of the monomials that occur in $f$. Let $\Gp$ be the groupification of $\Mon$. A pair $(g,h)\in (\base[\Mon])^2$ arises from Construction~\ref{constr: pairsGH} (for some monomial) if and only if there is a monomial $\mu\in\base[\Gp]$ such that $\mu(g,h)$ arises from Construction~\ref{constr: pairsGH} for $m_0$.
\end{prop}
\begin{proof}
%Let $m$ and $m'$ be monomials of $\base[\Mon]$. 
Let $m$ and $m'$ be monomials of $\base[\Mon]$ and let $f_1,\ldots,f_k$ be the terms of $f$.
We proceed to prove a series of claims.

Claim 1: $m$ is a \missmon\ for $f$ if and only if, for each $i\in\{1,\ldots,k\}$ there is a $j_i\neq i$ such that $f_{j_i}|mf_i$. In this case, the terms that work as $h_i$ are those of the form $a_i\dfrac{mf_i}{f_{j_i}}$ for $a\geq 1$ in $\base$ and $j_i\neq i$ with $f_{j_i}|mf_i$.

Proof of Claim 1: This follows immediately from the details of Construction~\ref{constr: pairsGH}.

Claim 2: If $m|m'$ and $(g,h)$ arises from Construction~\ref{constr: pairsGH} for $m$ then $\dfrac{m'}{m}(g,h)$ occurs from Construction~\ref{constr: pairsGH} for $m'$. In particular, if $m$ is a missing monomial for $f$ and $m|m'$ then $m'$ is a missing monomial for $f$. 

Proof of claim 2:
%To see this, suppose $(g,h)$ arises from Construction~\ref{constr: pairsGH} for $m$, so for each $i=1,\ldots,k$ we have a $j_i\neq i$ and a term $h_i$ such that $mf_i\leq f_{j_i}h_i$, $h=\dsum_{i=1}^k h_i$, and $g=h+m$. Multiplying through by $\dfrac{m'}{m}$, we get $m'f_i=\dfrac{m'}{m}mf_i\leq\dfrac{m'}{m}f_{j_i}h_i=f_{j_i}\left(\dfrac{m'}{m}h_i\right)$, $\dfrac{m'}{m}h=\dsum_{i=1}^k \dfrac{m'}{m}h_i$, and $\dfrac{m'}{m}g=\dfrac{m'}{m} h + \dfrac{m'}{m}m = \left(\dfrac{m'}{m} h\right) + m'$. So the pair $\dfrac{m'}{m}(g,h)$ arises from Construction~\ref{constr: pairsGH} for $m'$. This completes the proof of claim 1.
This follows from Claim 1 and the fact that, in applying Construction~\ref{constr: pairsGH} for $m$, we have $g=h+m$.

Claim 3: 
The monomial $m_0$ is a missing monomial for $f$.
Moreover, if $m$ and $m'$ are both divisible by $m_0$ then the pairs $(g',h')$ that arise from Construction~\ref{constr: pairsGH} for $m'$ are exactly those pairs of the form $\dfrac{m'}{m}(g,h)$ where $(g,h)$ is a pair that arises from Construction~\ref{constr: pairsGH} for $m$.

Proof of claim 3: 
%Say $m_0|m$ and $i\in\{1,\ldots,k\}$. Then for any $j_i\neq i$ we have $f_{j_i}|mf_i$, so we can let $h_i=\dfrac{mf_i}{f_{j_i}}$. Since we can do this for $m=m_0$, we see that $m$ is a \missmon\ for $f$. Moreover, the terms that work as $h_i$ are exactly those of the form $a\dfrac{mf_i}{f_{j_i}}$ where $a\geq 1$ is in $\base$, which implies the part of the claim about $m$ and $m'$.
This follows from Claim 1 because, if $m_0|m$ then, for every $j_i\neq i$, we have $f_{j_i}|m$ and so $f_{j_i}|mf_i$.

Claim 4: If $(g,h)$ arises from Construction~\ref{constr: pairsGH} for $m$ then $\dfrac{m_0}{m}(g,h)$ arises from Construction~\ref{constr: pairsGH} for $m_0$.

Proof of Claim 4: Let 
%$m'=m_0\cdot m$ \orange{do you want the cdot here, it is the only place you use it for multiplication}\blue{(it just seemed to me that $m_0m$ would look weird. But if you prefer it without the cdot, I'm okay with that.)}\red{it is not that weird so let's do that}, 
$m'=m_0m$,
so $m|m'$ and $m_0|m'$.  So Claim 2 tells us that $\dfrac{m'}{m}(g,h)$ arises from Construction~\ref{constr: pairsGH} for $m'$, and thus Claim 3 tells us that $\dfrac{m_0}{m'}\dfrac{m'}{m}(g,h)=\dfrac{m_0}{m}(g,h)$ arises from Construction~\ref{constr: pairsGH} for $m_0$. This proves Claim 4.

Claim 4 gives us the $(\Rightarrow)$ direction of the proposition. For the other direction, suppose that $(g,h)\in(\base[\Mon])^2$ and there is a monomial $\mu\in\base[M]$ such that $(g_0,h_0):=\mu(g,h)$ arises from Construction~\ref{constr: pairsGH} for $m_0$. 
% Write $\mu=\dfrac{m'}{m}$ with $m,m'\in\base[\Mon]$ monomials and $\gcd(m,m')=1$. Then the fact that 
% %$\dfrac{m}{m'}(g_0,h_0)=\mu^{-1}(g_0,h_0)=(g,h)$
% $\dfrac{m}{m'}g_0=\mu^{-1}g_0=g\in\base[\Mon]$
% gives us that $m'$ divides $g_0$ in $\base[\Mon]$. 
% Since $m_0$ is a term of $g$, $m'|m_0$; let $m'':=\dfrac{m_0}{m'}$, which is a term of $\base[\Mon]$.
% We have that $\dfrac{m'' f_i}{f_{j_i}}=$
Since
$\mu^{-1}g_0=g\in\base[\Mon]$
and $m_0$ is a term of $g_0$, we have that $\mu^{-1}m_0$ is a term in $\base[\Mon]$. 
Let $h_1,\ldots,h_k$ be the terms constructed in using Construction~\ref{constr: pairsGH} to produce $(g_0,h_0)$.
Since each $h_i$ ($1\leq i\leq k$) is a term of $g_0$, $\mu^{-1}h_i$ is a term in $\base[\Mon]$. By Claim 1, $h_i=a_i\dfrac{m_0f_{i}}{f_{j_i}}$ for some $j_i\neq i$ with $f_{j_i}|m_0f_{i}$ and $a_i\geq 1$ in $\base$.
We have $\mu^{-1}h_i=a_i\dfrac{(\mu^{-1}m_0)f_i}{f_{j_i}}$, so $f_{j_i}|(\mu^{-1}m_0)f_i$. Thus $h=\mu^{-1}h_0=\dsum_{i=1}^k a_i\dfrac{(\mu^{-1}m_0)f_i}{f_{j_i}}$ and $g=\mu^{-1}g_0=h+\mu^{-1}m_0$ form a pair $(g,h)$ that arises from Construction~\ref{constr: pairsGH} for $\mu^{-1}m_0$.
\end{proof}
% end of pinky
%%%%%%

We are now ready to put together the many reductions of this section with Construction~\ref{constr: pairsGH}.

\begin{thm}\label{thm:constructAllWitnessPairs}
Consider the following process.
\begin{enumerate}
\item
For each monomial $m$ in $\base[\Mon]$, let $X^{(1)}_m$ denote the set of those pairs that arise from Construction~\ref{constr: pairsGH} for $m$.

\item 
Let $X^{(2)}_m$ be the set of sums $(g,h)+(\theta,\theta)$ where $(g,h)\in X^{(1)}_m$ and the coefficient of $m$ in $\theta$ is strictly less than $1_{\base}$. Then let $X^{(2)}$ denote the union of $X^{(2)}_m$ over all $m$.

\item Let $X^{(3)}$ be the set of all nonzero constant multiples of pairs in $X^{(2)}$.

\item Let $X^{(4)}$ be the transitive closure of $X^{(3)}$, i.e., the set of those pairs $(G,H)$ for which there are $n\geq 1_{\Z}$ and $(g_1,h_1),(g_2,h_2),\ldots,(g_n,h_n)\in X^{(3)}$ such that $G=g_1$, $h_i=g_{i+1}$ for $1\leq i\leq n-1$ and $h_n=H$.

\item Let $X^{(5)}$ be the set of those pairs of the form $(g,h)$ or $(h,g)$ for $(g,h)\in X^{(4)}$, or of the form $(h,h')$ for some $(g,h),(g,h')\in X^{(4)}$ with $h\neq h'$.
\end{enumerate}

Then $X^{(5)}$ is the set of witness pairs for $f$.
\end{thm}\begin{proof}
Consider $(g,h)\in X^{(1)}_m$ and a polynomial $\theta$ such that the coefficient of $m$ in $\theta$ is strictly less than $1_{\base}$. 
Then $\theta$ is an arbitrary polynomial of the form $\theta=\diagel+am$ where $m\notin\supp\diagel$ and $a<1_{\base}$. 
So letting $X^{(1.5)}_m$ be the set of those pairs $(g,h)+(\diagel,\diagel)$ with $(g,h)\in X^{(1)}_m$ and $m\notin\supp\diagel$ we have that $X^{(2)}_m$ is the set of those pairs of the form $(g',h')+(am,am)$ with $(g',h')\in X^{(1.5)}_m$ and $a<1_{\base}$.

Lemma~\ref{lem: allm-witness} tells us that $X^{(1.5)}_m$ is the set of all $m$-witness pairs for $f$. 
Note that if $(g,h)\in X^{(1.5)}_m$ and $a<1_{\base}$ then $(g,h)+(am,am)=(h+m+am,h+am)=(h+m,h+am)$. So, by Lemma~\ref{lemma:ReduceToWitnessPairs}, $X^{(2)}_m$ is the set of those decreasing witness pairs $(g,h)$ such that $g$ and $h$ differ only in their coefficient of $m$ and the coefficient of $m$ in $g$ is $1_{\base}$. 
Thus $X^{(3)}$ is the set of all decreasing witness pairs $(g,h)$ where $g$ and $h$ differ only in the coefficient of a single monomial.
Lemma~\ref{lemma:reduceToOneCoeff} now tells us that $X^{(4)}$ is the set of all decreasing witness pairs for $f$, and so Lemma~\ref{lemma:ReduceToMonotone} tells us that $X^{(5)}$ is the set of all witness pairs for $f$. 
\end{proof}

\begin{coro}\label{coro:witnessCongruence}
Let $X$ be the set of all witness pairs for $f$ and let $\calc$ be the congruence on $\base[\Mon]$ generated by all of the pairs that arise from Construction~\ref{constr: pairsGH}. Then $X\cup\Delta_{\base[\Mon]}=\calc$.
\end{coro}\begin{proof}
Let $Y$ be the set of those pairs that arise from Construction~\ref{constr: pairsGH}.

It is routine to check that $X\cup\Delta_{\base[\Mon]}=\{(g,h)\in(\base[\Mon])^2\,:\,fg=fh\}$ is a congruence and contains 
%all pairs that arise from Construction~\ref{constr: pairsGH}, 
$Y$
so $X\cup\Delta_{\base[\Mon]}\supseteq\calc$. %\red{K: do you want to decorate the diagonal with a subscript as in the statement of the proposition} \blue{Nati: I was thinking that within this proof, it is clear what diagonal we are talking about. But we could definitely add those subscripts in.}

Since $\calc$ is a congruence, $\Delta\subseteq \calc$. So we only need to show $X\subseteq \calc$. But $Y\subseteq\calc$ by definition and Theorem~\ref{thm:constructAllWitnessPairs} tells us that if we start with $Y$ and only do operations that preserve being in a congruence, we get all of $X$.
\end{proof}

One way to approach Conjecture~\ref{conj: cancellatives} would be to show that if $f\in\base[\Mon]$ has more than one term then there is a pair $(g,h)\notin\Bend(J)$ such that $fg=fh$ in $\base[\Mon]$, i.e., there is a witness pair for $f$ that is not in $\Bend(J)$. 
Corollary~\ref{coro:witnessCongruence} tells us that this is equivalent to the existence of a pair that arises from Construction~\ref{constr: pairsGH} but is not contained in $\Bend(J)$.
%end of pinky 10

% In particular, the above lemma gives a sufficient condition for checking if $f$ is not cancellative in the quotient by the bend congruence of a tropical ideal $I$, namely, it is enough to check that none of the $m$-witness pairs that it gives rise to (as per Lemma~\ref{lem: wit-to-M-wit}) are in $\Bend(I)$.

% %We now consider when a pair of the form $(g,h)$ with $g=h+m$ and $m\notin\supp(h)$ can be in the bend congruence of a tropical ideal.

% %------------------------------------------------------------------------------

\appendix
%\section{ Counting circuits }
%\section{Degree bounds on circuits }
\section{Bounds on circuits} 
\label{app: countingCircuits}

In this appendix, we give a computational approach to Conjecture~\ref{coro:witnessCongruence}. Using as a starting point the discussion at the end of Section~\ref{sec: witness pairs}, we want to show that there are pairs $(g,h)$ that arise from Construction~\ref{constr: pairsGH} that are not in $\Bend(\trop I)$. We begin by characterizing when such a pair is in $\Bend(\trop I)$.

\begin{prop}\label{prop:thisPairInBend?}
Let $\base$ be a subsemifield of $\T$, let $v:K\to\base$ be a valuation on a field $K$, and let $\Mon$ be a toric monoid. 
Let $I\subseteq K[\Mon]$ be an ideal.
Let $h\in\base[\Mon]$, let $m$ be a monomial not in the support of $h$, and let $g=h+m$.
Then $(g,h)\in\Bend(\trop I)$ if and only if there is a $G\in I$ in which $m$ occurs with coefficient $1_K$ such that $\trop(G)\leq g$.
\end{prop}
\begin{proof}
($\Rightarrow$) 
Since $\delta(g,h)=1$ and $(g,h)$ is decreasing, 
Lemma~\ref{lemma:strongMR} and Remark~\ref{rmk:lengthOfTransitiveChain} tell us that there is an $a\in\base^\times$, $\frakf\in I$ with $m\in\supp(\frakf)$, 
and $H\in\base[\Mon]$ such that $g=a\trop(\frakf)+H$ and $h=a(\trop(\frakf))_{\what{m}}+H$. 
Since $m\notin\supp(h)$ and $H\leq h$, $m\notin\supp(H)$. Thus the coefficient of $m$ in $g=a\trop(\frakf)+H$ equals the coefficient of $m$ in $a\trop(\frakf)$. Letting $\beta$ be the coefficient of $m$ in $\frakf$, this says that $1_{\base}=a\cdot v(\beta)$. So, letting $G:=\beta^{-1}\frakf\in I$, the coefficient of $m$ in $G$ is $\beta^{-1}\beta=1_K$ and we have $\trop(G)=a\trop(\frakf)\leq a\trop(\frakf)+H=g$.

($\Leftarrow$) In this case we have $\trop(G)_{\what{m}}\leq h$, so 
\begin{align*}
(g,h)&=(h+m,h)\\
&=
(h+\trop(G)_{\what{m}}+m,h+\trop(G)_{\what{m}})\\
&=
(h,h)+(\trop(G),\trop(G)_{\what{m}})\in\Bend(\trop I).
\end{align*}
\par\nopagebreak\vspace{-1.4\baselineskip}\mbox{}
\end{proof}
% end of pinky 1

Putting together the discussion after Corollary~\ref{coro:witnessCongruence} with Proposition~\ref{prop:thisPairInBend?}, 
%we \red{aim?} 
to show that the image of a polynomial $f\in\base[\Mon]$ %\red{$\in \base[\Mon]$?} 
is not cancellative in $\base[\Mon]/\Bend(\trop I)$ it suffices to find a pair $(g,h)$ arising from Construction~\ref{constr: pairsGH} for some monomial $m$ such that there is no $G\in I$ in which $m$ occurs with coefficient $1_K$ such that $\trop(G)\leq g$. 
In particular, it suffices to do this restricted to the case where $G\in I$ has minimal support. Given such a $G$, $\trop(G)$ is a \emph{valuated circuit} in the \emph{valuated matroids} given by the \emph{tropical ideal} $\trop I$; see \cite[Section 4, the second paragraph after Definition 4.1]{MR14} and \cite{MR18} for the relevant definitions.
Note that we only need to consider those $\trop(G)$ which have at most as many terms as $g$.

Since the tropical ideal $\trop I$ is not finitely generated, one may wonder whether there is any hope for being able to check this computationally. Indeed, there are examples where there are infinitely many valuated circuits of the same size, all of which include the same monomial $m$.

\begin{example}
Let $K$ be a field of characteristic $p$, and let $f$ be a polynomial in $K[\Mon]$ with nonzero constant term and let $I$ be the ideal generated by $f$. Then, for any $n$, $f^{p^n}\in I$ has the same support size as $f$ and has $1\in\supp(f^{p^n})$. 
%\red{(technically there is another part here: so there is some minimal $g_n$ with $1\in\supp(g_n)\subseteq\supp(f^{p^n})$ and this $g_n$ has support size $\leq|\supp(f)|$. Since there are only finitely many options for the support size and infinitely many $g_n$, infinitely many of them have the same size.) (Should we include this, leave it out, or make some brief mention?)}
\exEnd
\end{example}

\begin{example}
Let $K$ be any field, let $f=(1+x)(1+y)=1+x+y+xy\in K[x,y]$, and let $I$ be the ideal generated by $f$. A consideration of Newton polytopes shows that $f$ does not contain any monomials, binomials, or trinomials. So $f\cdot\left(\dsum_{i=0}^n(-1)^ix^i\right)=1+y+(-1)^{n}x^{n+1}+(-1)^nx^{n+1}y$ in $I$ gives a valuated circuit whose support has size $4$ and contains $1$.
\exEnd
\end{example}
% end of pinky 2

We now simplify the problem further. To check that there is no $G\in I$ in which $m$ occurs with coefficient $1_K$ such that $\trop(G)\leq g$, it is enough to show that there is no $G\in I$ with $m\in\supp(G)\subseteq\supp(g)$. As before, it suffices to consider such $G$ with minimal support. In this case, $\supp(G)$ is a \emph{circuit} of the \emph{underlying matroids} of the valuated matroids of $\trop(I)$; see \cite[Section 4, the second paragraph after Definition 4.1]{MR14}. 
For any tropical ideal $J$, we will say that $C$ is a \emph{circuit of $J$} if it is a circuit of the underlying matroid of a valuated matroid of $J$. Similarly, we will talk about a circuit of an ideal in $K[\Mon]$. Thus we have shown the following corollary.

\begin{coro}\label{coro:SuppConditionForNotCancellative}
Let $\base$ be a subsemifield of $\T$, let $v:K\to\base$ be a valuation on a field $K$, and let $\Mon$ be a toric monoid. 
Let $I\subseteq K[\Mon]$ be an ideal and fix $f\in\base[\Mon]$. If there is a pair $(g,h)$ that arises from Construction~\ref{constr: pairsGH} for $f$ such that $\supp(g)$ does not contain any circuit of $I$, then the image of $f$ in $\base[\Mon]/\Bend(\trop I)$ is not cancellative. 
\end{coro}
%\red{Currently, the details of the proof of this Corollary are scattered in various places. Should we put them together into a proof here?} \blue{that might be an overkill, if you want we can put a few references.}
%%%%%Okay, I'm good leaving it as it is right now.

Note that we only need to check those circuits of size at most the size of the support of $g$. Before we state our main theorem of this appendix, we need one more definition.

\begin{definition}\label{def:PrimitiveCircuit}
Let $J\subset \T[x_1, \dots, x_n]$ be a tropical ideal. Let $C$ be a circuit of $J$. We call $C$ a \emph{primitive circuit} if it is not a nontrivial monomial multiple of another circuit. 
\end{definition}

Here we say that a circuit $C$ is a nontrivial monomial multiple of a circuit $C'$ if there are $f,f'\in I$ where $C=\supp(f)$, $C'=\supp(f')$ and $f$ is the product of $f'$ and a monomial other than $1$.

We will call $f\in I$ a \emph{cycle} of $I$ if it has minimal support among elements of $I$. If $f$ is a cycle with support $C=\supp(f)$, then the cycles with support $C$ are exactly the nonzero constant multiples of $f$.

In analogy to Definition~\ref{def:PrimitiveCircuit}, we will say that a cycle $f$ of $I$ is a \emph{primitive cycle} of $I$ if it is not a nontrivial monomial multiple of any other cycle. 
Note that, if $I$ is prime and does not contain any monomial, then a cycle $f\in I$ is primitive if and only if it is not divisible by any monomial other than $1$.
We will also use the fact that, if $g$ and $g'$ are cycles with the same support, then $g$ and $g'$ are scalar multiples of each other.
% end of pinky 3

\begin{thm}\label{thm:TrinomialFinitePrimitiveCircuits}
%Let $K$ be a field of characteristic 0. Let $f\in K[x,y]$ be a trinomial whose exponent vectors are not collinear, i.e., they do not lie on the same (affine) line, and let $I$ be the ideal generated by $f$. Then $I$ has finitely many primitive circuits of each size.
Let $K$ be a field of characteristic 0. Let $f\in K[x_1,\ldots,x_n]$ be a trinomial whose exponent vectors are not collinear, i.e., they do not lie on the same (affine) line, and let $I$ be the ideal generated by $f$. Then $I$ has finitely many primitive circuits of each size.
\end{thm}

Our proof will, in fact, give slightly more information; see Theorem~\ref{thm:TrinomialPrimitiveCircuitsCount}.
% end of pinky 4

We will prove Theorem~\ref{thm:TrinomialFinitePrimitiveCircuits} by reducing to a specific case.  In order to do this, we recall the "abc theorem for polynomials", also known as the Mason-Stothers theorem, and a generalization thereof.

\begin{theorem}[ABC for polynomials; \cite{Mason}, \cite{Stot}]\label{thm: abc}
    Let $K$ be a field of characteristic $0$. Let $a, b, c \in K[x]$ be polynomials such that $a, b, c$ 
    %do not have a common divisor 
    have greatest common divisor $1$
    and $a + b = c$. Then 
    $$\max\{\deg a, \deg b, \deg c\} \leq 
    %n_0(abc) 
    r(abc)
    -1,$$
    where 
    %$n_0(f)$ 
    $r(f)$
    is the number of distinct roots of $f\in K[x]$ in an algebraic closure of $K$.
\end{theorem}

To prove Theorem~\ref{thm:TrinomialFinitePrimitiveCircuits}, we will need a generalization of this theorem due to Brownawell and Masser \cite{BM86} and independently due to Voloch \cite{Volo85}. After adapting the language of 
%\cite[Theorem B]{BM86} 
\cite[Corollary I]{BM86} 
we have the following theorem:

\begin{theorem}[Generalized ABC for polynomials]\label{thm: abcGeneral}
    Let $K$ be a field of characteristic 0. Let $f_1, \dots, f_n \in K[x]$, $n\geq 3$ be polynomials such that $f_1, \dots, f_n $ 
    %do not have a common divisor. 
    have greatest common divisor 1.
    If $f_1 + \dots + f_n = 0$ and 
    %no smaller sub-sum is zero, 
    every proper subset of $\{f_1,\ldots,f_n\}$ is linearly independent over $K$,
    then 
    $$\deg f_i \leq \frac{(n-1)(n-2)}{2}(r(f_1\dots f_n)-1),$$
    where $r(f_1\dots f_n)$ is the number of distinct roots of $f_1\dots f_n\in K[x]$ in an algebraic closure of $K$.
\footnote{
We thank Anthropic's Claude Sonnet 4.6 for bringing to our attention this theorem and its usefulness in relation to circuits of an ideal.
No generative artificial intelligence was used in the writing of this paper.
} 
\end{theorem}

We are now ready to prove the specific case of Theorem~\ref{thm:TrinomialFinitePrimitiveCircuits} that we will later reduce to.

\begin{lemma}\label{lemma:1XYCircuits}
Let $K$ be a field of characteristic $0$ and let $I\subseteq K[x,y]$ be the ideal generated by $1+x+y$. Then any primitive circuit of $I$ of size $n$ consists of monomials of degree at most $\dfrac{(n-1)(n-2)}{2}$. In particular, $I$ has only finitely many primitive circuits of size $n$.
\end{lemma}
\begin{proof}

Because taking the quotient by $I$ is equivalent to making the substitution $y=-1-x$, for any $g(x,y)\in K[x,y]$ we have $g(x,y)\in I$ if and only if $g(x,-1-x)=0$.

Considering the number of vertices of the Newton polytope of $f\in I$ shows that $I$ cannot contain any monomials or binomials. So $n\geq3$.

Suppose that $g(x,y)=\dsum_{i=1}^n \alpha_i x^{k_i}y^{\ell_i}\in I$ is a primitive cycle whose support has size $n$; we want to show that $\deg g\leq \dfrac{(n-1)(n-2)}{2}$. Since $I$ is prime and contains no monomials,% \red{monomials?}, %%%%%Changed!
%this means that $g(x,y)$ 
the assumption on $g(x,y)$ means that it 
is not divisible by any monomial other than $1$. This, in turn, means that some $k_i$ is zero as is some $\ell_i$ (possibly for different values of $i$).

Let $f_i=\alpha_i x^{k_i}(-1-x)^{\ell_i}\in K[x]$, so $\dsum_{i=1}^n f_i=0$. 
Since $g(x,y)$ is a cycle of $I$, every proper subset of $\{f_1,\ldots, f_n\}$ is linearly independent over $K$. 
Since some $k_i$ is $0$ as is some $\ell_i$, the greatest common divisor of $f_1,\ldots,f_n$ is $1$. 
So, by Theorem~\ref{thm: abcGeneral}, $\deg f_i \leq \frac{(n-1)(n-2)}{2}(r(f_1\cdots f_n)-1)$. 
But the only zeros of $f_1\cdots f_n$ are $x=0$ and $x=-1$, so $\deg f_i\leq\dfrac{(n-1)(n-2)}{2}$. 
Thus, $\deg g=%\max\{k_i+\ell_i\,:\,i=1,\ldots,n\}=
\max\{\deg f_i\,:\, i=1,\ldots,n\}\leq\dfrac{(n-1)(n-2)}{2}$.

The final statement now follows because every primitive circuit of size $n$ is a subset of the finite set of monomials in $K[x,y]$ of degree at most $\dfrac{(n-1)(n-2)}{2}$.
\end{proof}
% end of pinky 6

\begin{example}\label{example:1XYSize3}
%Let $K$ and $I$ be as in Lemma~\ref{lemma:1XYCircuits}. For $n=3$, Lemma~\ref{lemma:1XYCircuits} tells us that every primitive circuit of $I$ of size $3$
For $n=3$, Lemma~\ref{lemma:1XYCircuits} tells us that any primitive cycle of $I=(1+x+y)$ with support size three 
has degree at most $\dfrac{(3-1)(3-2)}{2}=1$ and is therefore a constant multiple of $1+x+y$. Also note that, since $I$ does not contain any monomial or binomial, every trinomial in $I$ is a cycle. Thus, we conclude that the only trinomials that are multiples of $1+x+y$ are the monomial multiples of $1+x+y$. \exEnd
\end{example}

For $n>3$ the computation of the primitive circuits of $I$ of size $n$ requires significantly more work. Already for $n=4$ there are 210 sets of 4 monomials of degree at most $\frac{(4-1)(4-2)}{2}=3$ that need to be considered. Rather than do such computations by hand, we have written a script in SageMath to compute the primitive circuits. The script and output files up through $n=7$ can be found at \url{https://github.com/Nati-F/Circuits-of-trinomial-ideal}. We provide a chart below with a summary of the data.

%\red{(It could be that we should put this chart in a "figure" environment (or a similar one) and label and caption it.)} \blue{If it is a figure, we cannot control where it appears on the page and it is really annoying}

\begin{center}
    \begin{tabular}{|c|c|c|} 
    \hline
    %$n=$ circuit size & number of circuits & appearing in degrees \\
    \multirow{2}{*}{$n=$ circuit size} & number of primitive& primitive cycles \\
    &circuits/cycles&have degrees\\
    \hline
      3  & 1   & 1   \\
      4  & 7   & 2,3   \\
      5  & 69%695   
      & 3,4,5  \\
      6  & 1053   & 4,5,6,7   \\
      7  & 18500   & 5,6,7,8,9   \\
   \hline
   \end{tabular}
\end{center}

\begin{example}\label{example:1XYSize4}
The primitive circuits of $I=(1+x+y)$ of size four are 
\begin{align*}
% \{1, x, x^2, y^2\},&&
% \{1, x, xy, y^2\},&&
% \{1, x^2, xy, y^2\},&&
% \{1, x^2, xy, y\},\\
% \{1, x^2, y, y^2\},&&
% \{x, x^2, y, y^2\},&&\text{and}&&
% \{1, x^3, xy, y^3\},
%
% \{1, x, x^2, y^2\},&&
% \{1, x^2, xy, y^2\},&&
% \{1, x^2, y, y^2\},&&\\
% \{1, x, xy, y^2\},&&
% \{1, x^2, xy, y\},&&
% %\{x, x^2, y, y^2\},&&\text{and}&&
% \{x, x^2, y, y^2\},\\
% \text{and}&&
% \{1, x^3, xy, y^3\},
%
%
\{1, x^2, y^2, x\},&&
\{1, x^2, y^2, xy\},&&
\{1, x^2, y^2, y\},&&\\
\{1, y^2, x, xy\},&&
\{1, x^2, xy, y\},&&
%\{x, x^2, y, y^2\},&&\text{and}&&
\{x^2, y^2, x, y\},\\
\text{and}&&
\{1, x^3, y^3, xy\},
\end{align*}
corresponding to the cycles
\begin{align*}
% (1+x-y)\cdot(1+x+y)&=%(1+x)^2-y^2
% 1+2x+x^2-y^2,
% \\
% (1-y)\cdot(1+x+y)&=%(1-y^2)+x-xy
% 1+x-xy-y^2,
% \\
% (1-x-y)\cdot(1+x+y)&=%1-(x+y)^2
% 1-x^2-2xy-y^2,
% \\
% (1-x)\cdot(1+x+y)&=%(1-x^2)+y-xy
% 1-x^2+y-xy,
% \\
% (1-x+y)\cdot(1+x+y)&=%(1+y)^2-x^2
% 1-x^2+2y+y^2,
% \\
% (x-y)\cdot(1+x+y)&=%x-y+(x^2-y^2)
% x+x^2-y-y^2,
% \\
% \intertext{and}
% (1-x-y+x^2-xy+y^2)\cdot(1+x+y)&=1+x^3-3xy+y^3.%\exEnd
%
%1
(1+x-y)\cdot(1+x+y)&=%(1+x)^2-y^2
1+2x+x^2-y^2,
\\
%3
(1-x-y)\cdot(1+x+y)&=%1-(x+y)^2
1-x^2-2xy-y^2,
\\
%5
(1-x+y)\cdot(1+x+y)&=%(1+y)^2-x^2
1-x^2+2y+y^2,
\\
%2
(1-y)\cdot(1+x+y)&=%(1-y^2)+x-xy
1+x-xy-y^2,
\\
%4
(1-x)\cdot(1+x+y)&=%(1-x^2)+y-xy
1-x^2+y-xy,
\\
%6
(x-y)\cdot(1+x+y)&=%x-y+(x^2-y^2)
x+x^2-y-y^2,
\\
\intertext{and}
%7
(1-x-y+x^2-xy+y^2)\cdot(1+x+y)&=1+x^3-3xy+y^3.%\exEnd
\end{align*}
\par\nopagebreak\vspace{-1.4\baselineskip}\mbox{}
\exEnd
\end{example}
% end of pinky 7

We are now ready to prove the reduction step for Theorem~\ref{thm:TrinomialFinitePrimitiveCircuits}.

\begin{prop}\label{prop:TrinomialReduction}
Let $K$ be a field and let $k\leq r$ be positive integers. 
Let $f\in K[x_1,\ldots,x_r]$ be a polynomial with $k+1$ terms whose exponent vectors are affinely independent. Then there is a bijection between the primitive circuits of $I=(f)$ of size $n$ and the primitive circuits %of $I_0=(1+x_1+\cdots+x_k)$ of size $n$.
of size $n$ of the ideal $I_0$ generated by $1+x_1+\cdots+x_k$ in $K[x_1,\ldots,x_k]$.
\end{prop}

\begin{proof}
Throughout this proof we consider the relation $\sim$ on (Laurent) polynomials given by $g\sim g'$ if $g'$ is a multiple of $g$ by a Laurent monomial. %\red{why not "by a unit"}\blue{Because we want this relation to be multiplication by Laurent monomials even when we are only dealing with polynomials, not Laurent polynomials.}. 
Note that, for any principal ideal $\fraka\subset K[x_1,\ldots,x_r]$, the primitive cycles of $\fraka$ (up to scaling) form a system of representatives for the equivalence classes of cycles of $\fraka$. Moreover, extension gives a bijection between equivalence classes of cycles of $\fraka$ and equivalence classes of cycles of $\fraka K[x_1^{\pm1},\ldots,x_r^{\pm1}]$, and this bijection preserves support size.

Let $\wt{I}:=IK[x_1^{\pm1},\ldots,x_r^{\pm1}]$ and $\wt{I}_0:=I_0 K[x_1^{\pm1},\ldots,x_k^{\pm1}]$. 
Then, by the previous paragraph, it suffices to give a bijection between equivalence classes of cycles of $\wt{I}$ of support size $n$ and equivalence classes of cycles of $\wt{I}_0$ of support size $n$.

%Let $g$ be a cycle of $\wt{I}$ of support size $n$. So $g=fh$ for some $h\in K[x_1^{\pm1},\ldots,x_r^{\pm1}]$.

Let $L$ be the sublattice of $\Z^r$ generated by the edge directions of the Newton polytope of $f$, i.e., the lattice generated by differences between exponent vectors of terms of $f$. Give $R:=K[x_1^{\pm1},\ldots,x_r^{\pm1}]$ the $\Z^r/L$-grading $R=\displaystyle\bigoplus_{\kappa\in\Z^r/L}R_\kappa$ where, for any coset $\kappa\in\Z^r/L$, $R_\kappa:=\displaystyle\bigoplus_{u\in\kappa}K\cdot x^u$. In particular, $f$ is homogeneous in this grading. 
% For the remainder of this proof whenever we refer to "homogeneous" or "grading", it is always with this grading.
For the remainder of this proof, whenever we use the term ``homogeneous'', it is always with respect to this grading.

%We claim that, for any cycle $g$ of $\wt{I}$, all of the exponent vectors of terms of $g$ are in a single coset of $L$ in $\Z^r$. That is, every cycle $g$ of $\wt{I}$ is homogeneous. 
We claim that every cycle $g$ of $\wt{I}$ is homogeneous. %\red{"in this grading"?}\blue{see new pinky in the previous line in response.}
Towards seeing this, let $g=\dsum_{i=1}^{\ell}g_i$ be the decomposition of $g$ into homogeneous components, so the definition of the grading tells us that $\supp(g)$ is the disjoint union of $\supp(g_1), \ldots, \supp(g_\ell)$. Since $\wt{I}=f\cdot K[x_1^{\pm1},\ldots,x_r^{\pm1}]$ is a homogeneous ideal, each $g_i\in \wt{I}$. So, because $\supp(g)$ is minimal among supports of elements of $\wt{I}$, we must have $\ell=1$, so $g=g_1$ is homogeneous.

We now consider $K[L]$, viewed as the subring $R_L$ (where $L$ is being considered as the zero coset of $\Z^r/L$) of $K[x_1^{\pm1},\ldots,x_r^{\pm1}]$ consisting of Laurent polynomials whose exponent vectors are all in $L$.
Since every cycle of $\wt{I}$ is homogeneous, every equivalence class of cycles of $\wt{I}$ has a representative in $K[L]$. Thus, the equivalence classes of cycles of $\wt{I}$ are in bijection with the equivalence classes of cycles of the ideal $J:=\wt{I}\cap K[L]$ of $K[L]$, and this bijection preserves support size. 

Fix a term $\tau$ of $f$ and let $\wt{f}:=\tau^{-1}f$. Then $1$ is a term of $\wt{f}$, $\wt{f}\in K[L]$, and the ideal of $K[x_1^{\pm1},\ldots,x_r^{\pm1}]$ generated by $\wt{f}$ is $\wt{I}$. Moreover, the ideal of $K[L]$ generated by $\wt{f}$ is $J$.

Let $v_1,\ldots,v_k$ be the nonzero exponent vectors of $\wt{f}$, so $\wt{f}=1+\dsum_{i=1}^k a_ix^{v_i}$ with $a_i\in K$ nonzero. Since $v_1,\ldots,v_k$ are linearly independent and generate $L$, there is a group isomorphism $L\toup{\cong}\Z^k$ that sends $v_i$ to $e_i$, the $i^{\text{th}}$ standard basis vector of $\Z^k$. 
This gives a ring isomorphism $\ph:K[L]\to K[\Z^k]=K[x_1^{\pm1},\ldots,x_k^{\pm1}]$ that sends $x^v_i$ to $x^{e_i}=x_i$, and so sends $\wt{f}=1+a_1 x^{v_1}+\cdots+a_k x^{v_k}$ to $\ph\left(\wt{f}\right)=1+a_1x_1+\cdots+a_kx_k$. 
Letting $\psi$ be the automorphism of $K[x_1^{\pm1},\ldots,x_k^{\pm1}]$ sending $x_i$ to $a_i^{-1}x_i$, we have $\psi\left(\ph\left(\wt{f}\right)\right)=1+x_1+\cdots+x_k$. 
Since $\psi\circ\ph: K[L]\to K[x_1^{\pm1},\ldots,x_k^{\pm1}]$ is an isomorphism that preserves monomials, it gives a support-size-preserving bijection between equivalence classes of cycles of $J$ and equivalence classes of cycles of $\psi(\ph(J))$. 
Finally, $\psi(\ph(J))$ is the ideal of $K[x_1^{\pm1},\ldots,x_k^{\pm1}]$ generated by $\psi\left(\ph\left(\wt{f}\right)\right)=1+x_1+\cdots+x_k$.
\end{proof}
% end of pinky 8

\begin{thm}\label{thm:TrinomialPrimitiveCircuitsCount}
Let $K$ be a field of characteristic 0. Let $f\in K[x_1,\ldots,x_n]$ be a trinomial whose exponent vectors are not collinear, i.e., they are affinely independent, and let $I$ be the ideal generated by $f$. Then $I$ has only finitely many primitive circuits of size $n$ and this number is the same as the number of primitive circuits of size $n$ of the ideal of $K[x,y]$ generated by $1+x+y$.
\end{thm}\begin{proof}
This follows immediately from Theorem~\ref{lemma:1XYCircuits} and Proposition~\ref{prop:TrinomialReduction}.
\end{proof}

%\red{(If we make the table of numbers a figure and give it a label, then we should say here that this means that the table also applies to the case of this theorem.)}\blue{the number of circuits is the same but the degree might be different} Good point!

\begin{coro}\label{coro:TrinomialsSize3}
Let $K$ be a field of characteristic 0. Let $f\in K[x_1,\ldots,x_n]$ be a trinomial whose exponent vectors are not collinear and let $I$ be the ideal generated by $f$. Then the only primitive circuit of $I$ of size three is $\supp(f)$.
\end{coro}\begin{proof}
By Theorem~\ref{thm:TrinomialPrimitiveCircuitsCount} and Example~\ref{example:1XYSize3}, $I$ has only one primitive circuit of size $3$. Since $f$ is a primitive circuit of $I$, this completes the proof.
\end{proof}
% end of pinky 9

\begin{coro}
Let $K$ be a field of characteristic 0 and let $f\in K[x_1,\ldots,x_n]$ be a trinomial whose exponent vectors are not collinear. If $g\in K[x_1,\ldots,x_n]$ is such that $fg$ is a trinomial, then $g$ is a single term.
\end{coro}\begin{proof}
If $fg$ is a trinomial then Corollary~\ref{coro:TrinomialsSize3} tells us that it is a monomial multiple of a scalar multiple of $f$, i.e., $fg=fam$ for some monomial $m$ and scalar $a$. Since $K[x_1,\ldots,x_n]$ is an integral domain, $g=am$.
\end{proof}

\begin{coro}
%Let $K$ be a field of characteristic 0 with a valuation $K\to\base$ where $\base$ is a subsemifield of $\T$. Let $f\in K[x,y]$ be a trinomial whose exponent vectors are not collinear and let $I$ be the ideal generated by $f$. Then the image of any binomial of $\base[x,y]$ in $\base[x,y]/\Bend(\trop I)$ is not cancellative.
Let $K$ be a field of characteristic 0 with a valuation $K\to\base$ where $\base$ is a subsemifield of $\T$. Let $f\in K[x_1,\ldots,x_n]$ be a trinomial whose exponent vectors are not collinear and let $I$ be the ideal generated by $f$. Then the image of any binomial of $\base[x_1,\ldots,x_n]$ in $\base[x_1,\ldots,x_n]/\Bend(\trop I)$ is not cancellative.
\end{coro}\begin{proof}
% Let $\theta=a_1x^{u_1}y^{v_1}+a_2x^{u_2}y^{v_2}$ be a binomial in $\base[x,y]$ and let $m=x^{u_1+u_2}y^{v_1+v_2}$. Then in Construction~\ref{constr: pairsGH} we have $h=
% %mf_1/f_2+mf_2/f_1=
% %x^{u_1+u_2}y^{v_1+v_2}\dfrac{a_1x^{u_1}y^{v_1}}{a_2x^{u_2}y^{v_2}} + x^{u_1+u_2}y^{v_1+v_2}\dfrac{a_2x^{u_2}y^{v_2}}{a_1x^{u_1}y^{v_1}}=
% \dfrac{a_1}{a_2}x^{2u_1}y^{2v_1} + \dfrac{a_2}{a_1}x^{2u_2}y^{2v_2}
% $
% and $g=h+m=\dfrac{a_1}{a_2}x^{2u_1}y^{2v_1} + x^{u_1+u_2}y^{v_1+v_2} + \dfrac{a_2}{a_1}x^{2u_2}y^{2v_2}$. Note that the exponent vectors of $g$, being $(2u_1, 2v_1)$, $(u_1+u_2, v_1+v_2)$, and $(2u_2, 2v_2)$, are collinear.

% Since each trinomial in $I$ is the product of a term and $f$, its exponent vectors are not collinear. Thus, $\supp(g)$ does not contain any circuit of $I$ of size $3$. Since $I$ contains no monomials or binomials, this means that $\supp(g)$ does not contain any circuit of $I$. 
% So Corollary~\ref{coro:SuppConditionForNotCancellative} tells us that the image of $\theta$ in $\base[x,y]/\Bend(\trop I)$ is not cancellative.
%
Let $\theta=a\chi^{u}+b\chi^{v}$ be a binomial in $\base[x_1,\ldots,x_n]$ and let $m=\chi^{u+v}$. Then in Construction~\ref{constr: pairsGH} we have $h=
%\dfrac{a_1}{a_2}x^{2u_1}y^{2v_1} + \dfrac{a_2}{a_1}x^{2u_2}y^{2v_2}
\dfrac{a}{b}\chi^{2u} + \dfrac{b}{a} \chi^{2v}
$
and $g=h+m=\dfrac{a}{b}\chi^{2u} + \chi^{u+v} + \dfrac{b}{a} \chi^{2v}$. Note that the exponent vectors of $g$ are collinear.
%\blue{this is the first time we use the $\ubar{x}$ notation.} \red{(N: Would it be better to use $\chi^\bullet$?) K: yes I think} %% Done!

Since each trinomial in $I$ is the product of a term and $f$, its exponent vectors are not collinear. Thus, $\supp(g)$ does not contain any circuit of $I$ of size $3$. Since $I$ contains no monomials or binomials, this means that $\supp(g)$ does not contain any circuit of $I$. 
So Corollary~\ref{coro:SuppConditionForNotCancellative} tells us that the image of $\theta$ in $\base[x,y]/\Bend(\trop I)$ is not cancellative.
\end{proof}
% end of pinky 10

%%%%%
%Because we want to re-state this Corollary later, we create a temporary counter here to use later.
%%%%%
\newcounter{copySection}
\setcounter{copySection}{\value{section}}
\newcounter{singleUseCounter}
\setcounter{singleUseCounter}{\value{thm}}
%%%%%
\begin{coro}\label{coro:ImageOfTrinomialNotCancellative}
Let $K$ be a field of characteristic 0 with a valuation $K\to\base$ where $\base$ is a subsemifield of $\T$. Let $f\in K[x_1,\ldots,x_n]$ be a trinomial whose exponent vectors are not collinear and let $I$ be the ideal generated by $f$. Then the image of any trinomial of $\base[x_1,\ldots,x_n]$ in $\base[x_1,\ldots,x_n]/\Bend(\trop I)$ is not cancellative.
\end{coro}

While Corollary~\ref{coro:ImageOfTrinomialNotCancellative} is 
a fairly straightforward
consequence of results 
%\red{not a fan of "things", perhaps "results"} 
%%%%% Done!
we have already shown, the proof is long. We therefore give the proof its own appendix.
% end of pinky 11

\section{Proof of Corollary~\ref{coro:ImageOfTrinomialNotCancellative}}\label{app:ImageOfTrinomial}

In this section it will often be more convenient to work with exponent vectors rather than monomials. As such, we will call the set of exponent vectors of monomials of a (Laurent) polynomial $f$ the \emph{geometric support} of $f$. Similarly, if $C$ is a circuit of an ideal $I$ then the corresponding \emph{geometric circuit of $I$} is the set of exponent vectors of monomials in $C$.

% end of pinky 0

\newcommand{\circNameZero}{C_0}
\newcommand{\circSetZero}{\{u_1,u_2,u_3\}}
\newcommand{\circNameOne}{C_1}
\newcommand{\circSetOne}{\{2u_1, 2u_2, 2u_3, u_1+u_3\}}
\newcommand{\circNameTwo}{C_4}
\newcommand{\circSetTwo}{\{2u_2, 2u_3, u_1+u_2, u_1+u_3\}}
\newcommand{\circNameThree}{C_2}
\newcommand{\circSetThree}{\{2u_1, 2u_2, 2u_3, u_1+u_2\}}
\newcommand{\circNameFour}{C_5}
\newcommand{\circSetFour}{\{2u_1, 2u_3, u_1+u_2, u_2+u_3\}}
\newcommand{\circNameFive}{C_3}
\newcommand{\circSetFive}{\{2u_1, 2u_2, 2u_3, u_2+u_3\}}
\newcommand{\circNameSix}{C_6}
\newcommand{\circSetSix}{\{2u_1, 2u_2, u_1+u_3, u_2+u_3\}}
\newcommand{\circNameSeven}{C_7}
\newcommand{\circSetSeven}{\{3u_1, 3u_2, 3u_3, u_1+u_2+u_3\}}
%
%%%%% Capital and non-capital variations
%
\newcommand{\circNamezero}{\circNameZero}
\newcommand{\circnameZero}{\circNameZero}
\newcommand{\circnamezero}{\circNameZero}
\newcommand{\circSetzero}{\circSetZero}
\newcommand{\circsetZero}{\circSetZero}
\newcommand{\circsetzero}{\circSetZero}
\newcommand{\circNameone}{\circNameOne}
\newcommand{\circnameOne}{\circNameOne}
\newcommand{\circnameone}{\circNameOne}
\newcommand{\circSetone}{\circSetOne}
\newcommand{\circsetOne}{\circSetOne}
\newcommand{\circsetone}{\circSetOne}
\newcommand{\circNametwo}{\circNameTwo}
\newcommand{\circnameTwo}{\circNameTwo}
\newcommand{\circnametwo}{\circNameTwo}
\newcommand{\circSettwo}{\circSetTwo}
\newcommand{\circsetTwo}{\circSetTwo}
\newcommand{\circsettwo}{\circSetTwo}
\newcommand{\circNamethree}{\circNameThree}
\newcommand{\circnameThree}{\circNameThree}
\newcommand{\circnamethree}{\circNameThree}
\newcommand{\circSetthree}{\circSetThree}
\newcommand{\circsetThree}{\circSetThree}
\newcommand{\circsetthree}{\circSetThree}
\newcommand{\circNamefour}{\circNameFour}
\newcommand{\circnameFour}{\circNameFour}
\newcommand{\circnamefour}{\circNameFour}
\newcommand{\circSetfour}{\circSetFour}
\newcommand{\circsetFour}{\circSetFour}
\newcommand{\circsetfour}{\circSetFour}
\newcommand{\circNamefive}{\circNameFive}
\newcommand{\circnameFive}{\circNameFive}
\newcommand{\circnamefive}{\circNameFive}
\newcommand{\circSetfive}{\circSetFive}
\newcommand{\circsetFive}{\circSetFive}
\newcommand{\circsetfive}{\circSetFive}
\newcommand{\circNamesix}{\circNameSix}
\newcommand{\circnameSix}{\circNameSix}
\newcommand{\circnamesix}{\circNameSix}
\newcommand{\circSetsix}{\circSetSix}
\newcommand{\circsetSix}{\circSetSix}
\newcommand{\circsetsix}{\circSetSix}
\newcommand{\circNameseven}{\circNameSeven}
\newcommand{\circnameSeven}{\circNameSeven}
\newcommand{\circnameseven}{\circNameSeven}
\newcommand{\circSetseven}{\circSetSeven}
\newcommand{\circsetSeven}{\circSetSeven}
\newcommand{\circsetseven}{\circSetSeven}

\begin{lemma}\label{lemma:TrinomialsSize4}
Let $K$ be a field of characteristic 0. Let $f=a_1\chi^{u_1}+a_2\chi^{u_2}+a_3\chi^{u_3}\in K[x_1,\ldots,x_n]$ be a trinomial whose exponent vectors are not collinear and let $I$ be the ideal generated by $f$. 
%Then the sets of exponent vectors of circuits of $I$ of sizes 3 and 4 are shifts of
Then all geometric circuits of $I$ of size at most four are shifts of
\begin{align*}
% C_0&=\{u_1,u_2,u_3\},\\
% C_1&=\{2u_3, u_1+u_3, 2u_1, 2u_2\},\\
% C_2&=\{2u_3, u_1+u_3, u_1+u_2, 2u_2\},\\
% C_3&=\{2u_3, 2u_1, u_1+u_2, 2u_2\},\\
% C_4&=\{2u_3, 2u_1, u_1+u_2, u_2+u_3\},\\
% C_5&=\{2u_3, 2u_1, u_2+u_3, 2u_2\},\\
% C_6&=\{u_1+u_3, 2u_1, u_2+u_3, 2u_2\},\quad\text{and}\\
% C_7&=\{3u_3, 3u_1, u_1+u_2+u_3, 3u_2\}.
%
% \circNameZero&=\circSetZero,\\
% \circNameOne&=\circSetOne,\\
% \circNameTwo&=\circSetTwo,\\
% \circNameThree&=\circSetThree,\\
% \circNameFour&=\circSetFour,\\
% \circNameFive&=\circSetFive,\\
% \circNameSix&=\circSetSix,\quad\text{and}\\
% \circNameSeven&=\circSetSeven.
\circNameZero&=\circSetZero,\\
\circNameOne&=\circSetOne,\\
\circNameThree&=\circSetThree,\\
\circNameFive&=\circSetFive,\\
\circNameTwo&=\circSetTwo,\\
\circNameFour&=\circSetFour\\
\circNameSix&=\circSetSix,\quad\text{and}\\
\circNameSeven&=\circSetSeven.
\end{align*}
%\red{(Nati: Consider macroing the names of these sets as well as the sets themselves. Then reorder the sets to look nicer. And also reorder the elements in the sets to look nicer/be consistent between the sets. Then Also implement this in the instances where these sets are used in the following lemmas.)}
\end{lemma}\begin{proof}
As we have seen before, $I$ contains no monomials or binomials, i.e., $I$ has no circuits of size one or two. Corollary~\ref{coro:TrinomialsSize3} give us that every circuit of $I$ of size three is a monomial multiple of 
%$\circNameZero=\circSetZero$. 
$\{\chi^{u_1}, \chi^{u_2}, \chi^{u_3}\}$.
So we just need to consider circuits of size four. 
To do so, we apply the process given in the proof of Proposition~\ref{prop:TrinomialReduction} to the circuits from Example~\ref{example:1XYSize4}.

We extend $I$ to $\wt{I}:=I K[x_1^{\pm1},\ldots,x_n^{\pm1}]$ and let $I_0$ be the ideal of $K[x,y]$ generated by $1+x+y$.
%, and let $\wt{I}_0:= I K[x^{\pm1},y^{\pm1}]$. 
%
Fix the term $\tau=a_3\chi^{u_3}$ of $f$ and let $\wt{f}:=\tau^{-1}f=1+\frac{a_1}{a_3}\chi^{u_1-u_3}+\frac{a_2}{a_3}\chi^{u_2-u_3}$. We now let $\ph:K[x,y]\to K[x_1^{\pm1},\ldots,x_n^{\pm1}]$ be the $K$-algebra morphism given by $x\mapsto\frac{a_1}{a_3}\chi^{u_1-u_3}$ and $y\mapsto\frac{a_2}{a_3}\chi^{u_2-u_3}$. 
By the proof of Proposition~\ref{prop:TrinomialReduction}, every cycle of $\wt{I}$ is a  Laurent monomial multiple of the image under $\ph$ of a cycle of $I_0$. Applying the corresponding map on monomials to the circuits from Example~\ref{example:1XYSize4} gives us
%\orange{this is a bit overflowing}\blue{(Yes, but this structure mirrors how the sets in Example~\ref{example:1XYSize4} are laid out and also fits with the different appearances of $C_1,C_2,C_3$ versus $C_4,C_5,C_6$ versus $C_7$. In my mind this makes the proof easier to follow.)}
\begin{align*}
% \{1, (\chi^{u_1-u_3})^2, (\chi^{u_2-u_3})^2, \chi^{u_1-u_3}\},&&
% \{1, (\chi^{u_1-u_3})^2, (\chi^{u_2-u_3})^2, (\chi^{u_1-u_3})\chi^{u_2-u_3}\},&&
% \{1, (\chi^{u_1-u_3})^2, (\chi^{u_2-u_3})^2, \chi^{u_2-u_3}\},&&\\
% \{1, (\chi^{u_2-u_3})^2, \chi^{u_1-u_3}, \chi^{u_1-u_3}\chi^{u_2-u_3}\},&&
% \{1, (\chi^{u_1-u_3})^2, \chi^{u_1-u_3}\chi^{u_2-u_3}, \chi^{u_2-u_3}\},&&
% \{(\chi^{u_1-u_3})^2, (\chi^{u_2-u_3})^2, \chi^{u_1-u_3}, \chi^{u_2-u_3}\},\\
% \text{and}&&
% \{1, (\chi^{u_1-u_3})^3, (\chi^{u_2-u_3})^3, \chi^{u_1-u_3}\chi^{u_2-u_3}\}.
%
\hspace{-0.5cm}
\{1, \chi^{2u_1-2u_3}, \chi^{2u_2-2u_3}, \chi^{u_1-u_3}\},&&
\{1, \chi^{2u_1-2u_3}, \chi^{2u_2-2u_3}, \chi^{u_1+u_2-2u_3}\},&&
\{1, \chi^{2u_1-2u_3}, \chi^{2u_2-2u_3}, \chi^{u_2-u_3}\},&&\\
\hspace{-2cm}
\{1, \chi^{2u_2-2u_3}, \chi^{u_1-u_3}, \chi^{u_1+u_2-2u_3}\},&&
\{1, \chi^{2u_1-2u_3}, \chi^{u_1+u_2-2u_3}, \chi^{u_2-u_3}\},&&
\{\chi^{2u_1-2u_3}, \chi^{2u_2-2u_3}, \chi^{u_1-u_3}, \chi^{u_2-u_3}\},\\
\text{and}&&
\{1, \chi^{3u_1-3u_3}, \chi^{3u_2-3u_3}, \chi^{u_1+u_2-2u_3}\}.
\end{align*}
Multiplying by monomials to ensure that we are in $K[x_1,\ldots,x_n]$ instead of $K[x_1^{\pm1},\ldots,x_n^{\pm1}]$, we find that every circuit of $I$ is a Laurent monomial multiple of one of 
\begin{align*}
% \hspace{-0.5cm}
% \{1, \chi^{2u_1-2u_3}, \chi^{2u_2-2u_3}, \chi^{u_1-u_3}\},&&
% \{1, \chi^{2u_1-2u_3}, \chi^{2u_2-2u_3}, \chi^{u_1+u_2-2u_3}\},&&
% \{1, \chi^{2u_1-2u_3}, \chi^{2u_2-2u_3}, \chi^{u_2-u_3}\},&&\\
% \hspace{-2cm}
% \{1, \chi^{2u_2-2u_3}, \chi^{u_1-u_3}, \chi^{u_1+u_2-2u_3}\},&&
% \{1, \chi^{2u_1-2u_3}, \chi^{u_1+u_2-2u_3}, \chi^{u_2-u_3}\},&&
% \{\chi^{2u_1-2u_3}, \chi^{2u_2-2u_3}, \chi^{u_1-u_3}, \chi^{u_2-u_3}\},\\
% \text{and}&&
% \{1, \chi^{3u_1-3u_3}, \chi^{3u_2-3u_3}, \chi^{u_1+u_2-2u_3}\}.
%
\{\chi^{2u_3}, \chi^{2u_1}, \chi^{2u_2}, \chi^{u_1+u_3}\},&&
\{\chi^{2u_3}, \chi^{2u_1}, \chi^{2u_2}, \chi^{u_1+u_2}\},&&
\{\chi^{2u_3}, \chi^{2u_1}, \chi^{2u_2}, \chi^{u_2+u_3}\},&&\\
\{\chi^{2u_3}, \chi^{2u_2}, \chi^{u_1+u_3}, \chi^{u_1+u_2}\},&&
\{\chi^{2u_3}, \chi^{2u_1}, \chi^{u_1+u_2}, \chi^{u_2+u_3}\},&&
\{\chi^{2u_1}, \chi^{2u_2}, \chi^{u_1+u_3}, \chi^{u_2+u_3}\},\\
\text{and}&&
\{\chi^{3u_3}, \chi^{3u_1}, \chi^{3u_2}, \chi^{u_1+u_2+u_3}\}.
\end{align*}
The lemma statement is obtained by taking the exponent vectors of these monomials.
\end{proof}

We will build up to using Lemma~\ref{lemma:TrinomialsSize4} over the course of several lemmas.

\begin{lemma}\label{lemma:TheseExpVectSetsAreDistinct}
Let $u_1,$ $u_2,$ and $u_3$ be distinct elements of 
%an abelian group (written additively) that has no 7-torsion. 
a real vector space. 
Then $\{2u_1-u_2-u_3, 2u_2-u_1-u_3, 2u_3-u_1-u_2\}$ is not equal to either of $\{u_1-u_2, u_2-u_3, u_3-u_1\}$ or $\{u_1-u_3, u_2-u_1, u_3-u_2\}$. 
%\orange{should this be done for real vectors and the 7-torsion observation should be in a remark?}
%%%%% Done!
\end{lemma}\begin{proof}
Note that the permutation interchanging $u_1$ and $u_2$ leaves $\{2u_1-u_2-u_3, 2u_2-u_1-u_3, 2u_3-u_1-u_2\}$ invariant and exchanges $\{u_1-u_2, u_2-u_3, u_3-u_1\}$ and $\{u_1-u_3, u_2-u_1, u_3-u_2\}$. So it suffices to show that $\{2u_1-u_2-u_3, 2u_2-u_1-u_3, 2u_3-u_1-u_2\}\neq\{u_1-u_2, u_2-u_3, u_3-u_1\}$. Suppose, for contradiction, that $\{2u_1-u_2-u_3, 2u_2-u_1-u_3, 2u_3-u_1-u_2\}=\{u_1-u_2, u_2-u_3, u_3-u_1\}$.

If $u_1-u_2=2u_1-u_2-u_3$ then $0=u_1-u_3$, %so $u_1=u_3$, 
contradicting the hypothesis that $u_1$, $u_2$, and $u_3$ are distinct. Thus $u_1-u_2\neq 2u_1-u_2-u_3$. Similar considerations show that $u_2-u_3\neq 2u_2-u_1-u_3$, $u_3-u_1\neq 2u_3-u_1-u_2$, $u_1-u_2\neq 2u_3-u_1-u_2$, $u_2-u_3\neq 2u_1-u_2-u_3$, and $u_3-u_1\neq 2u_2-u_1-u_3$.

So we must have $u_1-u_2=2u_2-u_1-u_3$, $u_2-u_3=2u_3-u_1-u_2$, and $u_3-u_1=2u_1-u_2-u_3$. Rearranging each of these gives $2(u_1-u_2)=u_2-u_3$, $2(u_2-u_3)=u_3-u_1$, and $2(u_3-u_1)=u_1-u_2$. 
%Then $8(u_1-u_2)=u_1-u_2$, i.e., $7(u_1-u_2)=0$. Since there is no 7-torsion, $u_1-u_2=0$, 
Then $8(u_1-u_2)=u_1-u_2$ and so $u_1-u_2=0$, 
contradicting $u_1$, $u_2$, and $u_3$ being distinct.
\end{proof}
% end of pinky 1

\begin{remark}
The proof of Lemma~\ref{lemma:TheseExpVectSetsAreDistinct} shows that the result is also true if $u_1$, $u_2$, and $u_3$ are distinct elements of an abelian group with no 7-torsion.
\end{remark}

\begin{lemma}\label{lemma:quadrilaterals}
Let $u_1$, $u_2$, and $u_3$ be affinely independent vectors in a real vector space. Then the convex hulls of each of the sets 
% $C_2=\{2u_3, u_1+u_3, u_1+u_2, 2u_2\}$, 
% $C_4=\{2u_3, 2u_1, u_1+u_2, u_2+u_3\}$, and 
% $C_6=\{u_1+u_3, 2u_1, u_2+u_3, 2u_2\}$ 
$\circNameTwo=\circSetTwo$, 
$\circNameFour=\circSetFour$, and 
$\circNameSix=\circSetSix$ 
are all quadrilaterals. 
\end{lemma}\begin{proof}
Since the permutations of $\{u_1,u_2,u_3\}$ act transitively on $\{\circNameTwo, \circNameFour, \circNameSix\}$, it suffices to show that the convex hull of $\circNameTwo$ is a quadrilateral.

Note that $2u_2=2(u_1+u_2)-2(u_1+u_3)+1(2u_3)$, 
%with \green{$2-2+1=1$}, 
so $2u_2$ is in the affine hull of $\{2u_3, u_1+u_2, u_1+u_3\}$, so the dimension of the convex hull of $\circNameTwo
%\green{=\circSetTwo}
$ is at most two.

Since $u_1$, $u_2$, and $u_3$ are affinely independent, the line through $\{u_1,u_2\}$ is not parallel to the line through $\{u_1,u_3\}$. So the lines between $\{u_1,u_2\}+u_2=\{u_1+u_2, 2u_2\}$ and $\{u_1,u_3\}+u_3=\{u_1+u_3,2u_3\}$ are not parallel. Thus the dimension of the convex hull of $\circNameTwo
%\green{=\circSetTwo}
$ is at least two. So, in light of the previous paragraph, the dimension of the convex hull of $\circNameTwo$ is exactly two.

Thus, it now suffices to show that each point of $\circNameTwo
%\green{=\circSetTwo}
$ is not contained in the convex hull of the other three. Note that it follows from the affine independence of $u_1$, $u_2$, and $u_3$ that $2u_2$, $2u_3$, $u_1+u_2$, and $u_1+u_3$ are distinct. 
So, to show that $p\in\circNameTwo$ is not in the convex hull of $\circNameTwo\sdrop\{p\}$, it suffices to show that the cone generated by the shift $\circNameTwo-p$ of $\circNameTwo$ by $-p$ is strictly convex. Since this cone is two-dimensional, it suffices to show that this cone can be generated by two vectors.

For $p=2u_3$, we have that $\circNameTwo-p
%=\{2u_2, 2u_3, u_1+u_2, u_1+u_3\}-(2u_3)
=\{2(u_2-u_3), 0, u_1+u_2-2u_3, u_1-u_3\}
$ generates the same cone as $\{u_2-u_3, u_1-u_3\}$. The details for the other values of $p$ are similar.%, and are left to the reader.
%%%%% Here are the details for the other values of $p$.
% $p=2u_2$: $\circNameTwo-p
% =\{2u_2, 2u_3, u_1+u_2, u_1+u_3\}-(2u_2)
% =\{0, 2(u_3-u_2), u_1-u_2, u_1+u_3-2u_2\}$ generates the same cone as $\{u_3-u_2,u_1-u_2\}$.\\
% $p=u_1+u_2$: $\circNameTwo-p
% =\{2u_2, 2u_3, u_1+u_2, u_1+u_3\}-(u_1+u_2)
% =\{u_2-u_1, 2u_3-u_1-u_2, 0, u_3-u_2\}$ generates the same cone as $\{u_2-u_1, u_3-u_2\}$ because $2u_3-u_1-u_2=2(u_3-u_2)+(u_2-u_1)$.\\
% $p=u_1+u_3$: $\circNameTwo-p
% =\{2u_2, 2u_3, u_1+u_2, u_1+u_3\}-(u_1+u_3)
% =\{2u_2-u_1-u_3, u_3-u_1, u_2-u_3, 0\}$ generates the same cone as $\{u_3-u_1, u_2-u_3\}$ because $2u_2-u_1-u_3=2(u_2-u_3)+(u_3-u_1)$.\\
%%%%% These were the details for the other values of $p$.
\end{proof}

\begin{lemma}\label{lemma:G3Works}
Let $u_1$, $u_2$, and $u_3$ be affinely independent vectors in a real vector space. Then no shift of the set $G=\{2u_1+u_3, u_1+2u_2, u_2+2u_3, u_1+u_2+u_3\}$ contains any shift of any of the following sets:
\begin{align*}
% C_0&=\{u_1,u_2,u_3\},\\
% C_1&=\{2u_3, u_1+u_3, 2u_1, 2u_2\},\\
% C_2&=\{2u_3, u_1+u_3, u_1+u_2, 2u_2\},\\
% C_3&=\{2u_3, 2u_1, u_1+u_2, 2u_2\},\\
% C_4&=\{2u_3, 2u_1, u_1+u_2, u_2+u_3\},\\
% C_5&=\{2u_3, 2u_1, u_2+u_3, 2u_2\},\\
% C_6&=\{u_1+u_3, 2u_1, u_2+u_3, 2u_2\},\quad\text{and}\\
% C_7&=\{3u_3, 3u_1, u_1+u_2+u_3, 3u_2\}.
\circNameZero&=\circSetZero,\\
\circNameOne&=\circSetOne,\\
\circNameThree&=\circSetThree,\\
\circNameFive&=\circSetFive,\\
\circNameTwo&=\circSetTwo,\\
\circNameFour&=\circSetFour\\
\circNameSix&=\circSetSix,\quad\text{and}\\
\circNameSeven&=\circSetSeven.
\end{align*}
\end{lemma}\begin{proof}
It suffices to show that $G$ cannot contain any shift of any of $C_0,C_1,C_2,C_3,C_4,C_5,C_6$ and $C_7$. 
Also, since $C_1, C_2, C_3, C_4, C_5, C_6, C_7$, and $G$ all have four elements, 
%a shift of $G$ contains a shift of some $C_i$ for $i=1,\ldots,7$ if and only if those shifts are equal.
$G$ contains a shift of some $C_i$ for $i=1,\ldots,7$ if and only if $G$ equals that shift.

Note that the convex hulls of each of the sets $\circNameOne$, $\circNameThree$, $\circNameFive$ are triangles and that the non-vertex point in each of these sets 
%lays \orange{lies?} along 
is the midpoint of
an edge of the triangle.

Since $u_1+u_2+u_3=\frac{1}{3}(2u_1+u_3)+\frac{1}{3}(u_1+2u_2)+\frac{1}{3}(u_2+2u_3)$, $p=u_1+u_2+u_3$ is in the relative interior of the convex hull of $G\sdrop\{p\}$, the previous paragraph and Lemma~\ref{lemma:quadrilaterals} %\orange{the lemma or its proof} \blue{(The lemma itself. Because the convex hulls of $C_4,C_5$ and $C_6$ are quadrilaterals while the convex hull of $G$ has at most 3 vertices.)} 
give us that 
%no shift of $G$ can be equal to any shift of $C_1,C_2,C_3,C_4,C_5$, or $C_6$.
$G$ cannot be equal to any shift of $C_1,C_2,C_3,C_4,C_5$, or $C_6$.

So we just need to show that $G$ cannot be equal to a shift of $C_7$ or contain a shift of $C_0$.

Suppose that $G$ is equal to a shift of $\circNameSeven=\circSetSeven$. Since $u_1+u_2+u_3\in C_7$ is in the relative interior of the triangle that is the convex hull of $C_7\sdrop\{u_1+u_2+u_3\}$ and $u_1+u_2+u_3\in G$ is in the relative interior of the convex hull of $G\sdrop\{u_1+u_2+u_3\}$, the shift of $u_1+u_2+u_3\in C_7$ must equal $u_1+u_2+u_3\in G$. That is, we must have $C_7=G$. Thus, $(C_7\sdrop\{u_1+u_2+u_3\})-(u_1+u_2+u_3)=\{2u_1-u_2-u_3, 2u_2-u_1-u_3, 2u_3-u_1-u_2\}$ and $(G\sdrop\{u_1+u_2+u_3\})-(u_1+u_2+u_3)=\{u_1-u_2, u_2-u_3, u_3-u_1\}$ must be the same. But, by Lemma~\ref{lemma:TheseExpVectSetsAreDistinct}, this is impossible. 

So now suppose that $G$ contains a shift of $C_0=\{u_1,u_2,u_3\}$. Because $(G\sdrop\{u_1+u_2+u_3\})-(u_1+u_2+u_3)=\{u_1-u_2, u_2-u_3, u_3-u_1\}$ does not contain any two vectors of the forms $u_i-u_k$ and $u_j-u_k$ for $\{i,j,k\}=\{1,2,3\}$, any shift of $C_0$ contained in $G$ cannot contain $u_1+u_2+u_3$. So $G\sdrop\{u_1+u_2+u_3\}=\{2u_1+u_3, u_1+2u_2, u_2+2u_3\}$ must be equal to a shift of $C_0=\{u_1,u_2,u_3\}$.  So 
$\{(2u_1+u_3)-(u_2+2u_3), (u_1+2u_2)-(2u_1+u_3), (u_2+2u_3)-(u_1+2u_2)\}=\{2u_1-u_2-u_3, 2u_2-u_1-u_3, 2u_3-u_1-u_2\}$ must be equal to either $\{u_1-u_2, u_2-u_3, u_3-u_1\}$ or $\{u_1-u_3, u_3-u_2, u_2-u_1\}$. But both of these are impossible by Lemma~\ref{lemma:TheseExpVectSetsAreDistinct}.
\end{proof}

% end of pinky 2

We now recall Corollary~\ref{coro:ImageOfTrinomialNotCancellative} and provide its proof.

\setcounter{section}{\value{copySection}}
\setcounter{thm}{\value{singleUseCounter}}%%%%%This was the state of the "thm" counter just before 
\begin{coro}%\label{coro:ImageOfTrinomialNotCancellative}
Let $K$ be a field of characteristic 0 with a valuation $K\to\base$ where $\base$ is a subsemifield of $\T$. Let $f\in K[x_1,\ldots,x_n]$ be a trinomial whose exponent vectors are not collinear and let $I$ be the ideal generated by $f$. Then the image of any trinomial of $\base[x_1,\ldots,x_n]$ in $\base[x_1,\ldots,x_n]/\Bend(\trop I)$ is not cancellative.
\end{coro}
\begin{proof}%[Proof of Corollary~\ref{coro:ImageOfTrinomialNotCancellative}]
Let 
$\theta=b_1\chi^{v_1}+b_2\chi^{v_2}+b_3\chi^{v_3}\in\base[x_1,\ldots,x_n]$ be a trinomial; we want to show that the image of $\theta$ in $\base[x_1,\ldots,x_n]/\Bend(\trop I)$ is not cancellative.
The possibilities for the geometric support of $g$ where $(g,h)$ arises from Construction~\ref{constr: pairsGH} are
%%%%% To me, the following seems very hard to read - it feels like a solid block of text where the symbols almost start to swim if you're looking in the middle. So I'm doing a version with more space below.
% \begin{align*}
% G_1&=\{2v_1+v_3, 2v_2+v_3, v_2+2v_3, v_1+v_2+v_3\},\\
% G_2&=\{2v_1+v_3, 2v_2+v_3, v_1+2v_3, v_1+v_2+v_3\},\\
% G_3&=\{2v_1+v_3, v_1+2v_2, v_2+2v_3, v_1+v_2+v_3\},\\
% G_4&=\{2v_1+v_3, v_1+2v_2, v_1+2v_3, v_1+v_2+v_3\},\\
% G_5&=\{2v_2+v_2, 2v_2+v_3, v_2+2v_3, v_1+v_2+v_3\},\\
% G_6&=\{2v_2+v_2, 2v_2+v_3, v_1+2v_3, v_1+v_2+v_3\},\\
% G_7&=\{2v_2+v_2, v_1+2v_2, v_2+2v_3, v_1+v_2+v_3\}, \quad\text{and}\\
% G_8&=\{2v_2+v_2, v_1+2v_2, v_1+2v_3, v_1+v_2+v_3\}.
% \end{align*}

\begin{align*}
G_1&=\{2v_1+v_3, &2v_2+v_3,&&v_2+2v_3, &&v_1+v_2+v_3\},\\
G_2&=\{2v_1+v_3, &2v_2+v_3,&&v_1+2v_3, &&v_1+v_2+v_3\},\\
G_3&=\{2v_1+v_3, &v_1+2v_2,&&v_2+2v_3, &&v_1+v_2+v_3\},\\
G_4&=\{2v_1+v_3, &v_1+2v_2,&&v_1+2v_3, &&v_1+v_2+v_3\},\\
G_5&=\{2v_1+v_2, &2v_2+v_3,&&v_2+2v_3, &&v_1+v_2+v_3\},\\
G_6&=\{2v_1+v_2, &2v_2+v_3,&&v_1+2v_3, &&v_1+v_2+v_3\},\\
G_7&=\{2v_1+v_2, &v_1+2v_2,&&v_2+2v_3, &&v_1+v_2+v_3\},\\%& \quad\green{\text{and}}\\
G_8&=\{2v_1+v_2, &v_1+2v_2,&&v_1+2v_3, &&v_1+v_2+v_3\}.
\end{align*}
%\orange{for $G_5-G_8$ in the first column do you mean $2v_2 +v_1$} 
%%%%% Oops! It is now changed to $2v_1+v_2$. The patern is that in column $i=1,2,3$, we have $2v_i$ and then either $+v_j$ or $+v_k$ where $\{i,j,k\}=\{1,2,3\}$. 
%
%\blue{(The "and"s in the above and below displayed/aligned equations seem to me to be grammatically necessary.)}
%
By Corollary~\ref{coro:SuppConditionForNotCancellative} and Lemma~\ref{lemma:TrinomialsSize4}, it suffices to show that one of $G_1,\ldots,G_8$ does not contain any shift of any of 
\begin{align*}
\circNameZero&=\circSetZero,\\
\circNameOne&=\circSetOne,\\
\circNameThree&=\circSetThree,\\
\circNameFive&=\circSetFive,\\
\circNameTwo&=\circSetTwo,\\
\circNameFour&=\circSetFour,\\
\circNameSix&=\circSetSix,\\%\quad\green{\text{and}}\\
\circNameSeven&=\circSetSeven.
\end{align*}
Since all $G_j$ (for $j=1,\ldots,8$) and $C_i$ for $i=1,\ldots,7$ have size four, for any $G_j$ to contain a shift of one of these $C_i$, it must be equal to that shift.

By \ref{lemma:quadrilaterals}, none of $\circNameTwo$, $\circNameFour$, and $\circNameSix$ contain three collinear points.  Since $u_1+u_2+u_3$ is in the relative interior of the triangle whose vertices are $3u_1$, $3u_2$, and $3u_3$, $\circNameSeven$ also does not contain three collinear points. On the other hand, $\circNameOne$, $\circNameThree$, and $\circNameFive$ each have three evenly spaced collinear points (with the fourth point not on that line): 
$\circNameOne$ has $2u_1$, $u_1+u_3$, and $2u_3$, 
$\circNameThree$ has $2u_1$, $u_1+u_2$, and $2u_2$, and
$\circNameFive$ has $2u_2$, $u_2+u_3$, and $2u_3$. 
For each of $\circNameOne$, $\circNameThree$, and $\circNameFive$, we now compute the vectors pointing from the middle point on the line to the outer points. For $\circNameOne$ we get $\pm(u_1-u_3)$, for $\circNameThree$ we get $\pm(u_1-u_2)$, and for $\circNameFive$ we get $\pm(u_2-u_3)$.

Note that both $G_1$ and $G_2$ contain the three evenly spaced collinear points $2v_1+v_3$, $v_1+v_2+v_3$, and $2v_2+v_3$. The vectors pointing from the middle point ($v_1+v_2+v_3$) to the outer points are $v_1-v_2$ and $v_2-v_1$.

We now consider the case where $v_1-v_2\notin\{\pm(u_1-u_2), \pm(u_1-u_3), \pm(u_2-u_3)\}$. By the last two paragraphs, we see that $G_1$ and $G_2$ both cannot be any shift of $C_i$ for $i=1,\ldots,7$.
So, to finish this case, it suffices to show that $G_1$ and $G_2$ don't both contain a shift of $\circnamezero$. 
% Towards this, let $p_1=v_2+2v_3$ and $p_2=v_1+2v_3$ so $G_j=\{2v_1+v_3, 2v_2+v_3, p_j, v_1+v_2+v_3\}$ for $j=1,2$.
% If $G_j$ contains a shift $D_j=\circnamezero+w_j$ of $\circnamezero$ then exactly two vertices of $D_j=\{u_1+w_j, u_2+w_j, u_3+w_j\}$ must be in $\{2v_1+v_3, 2v_2+v_3, v_1+v_2+v_3\}$. 
Towards this, let $p_1=v_2+2v_3$ and $p_2=v_1+2v_3$ so $G_j=\{2v_1+v_3, 2v_2+v_3, p_j, v_1+v_2+v_3\}$ for $j=1,2$ and suppose, for contradiction, that both $G_1$ and $G_2$ contain shifts of $\circnamezero$. 
Let $D_j=\circnamezero+w_j$ be the shift of $\circnamezero$ contained in $G_j$.
So exactly two vertices of $D_j=\{u_1+w_j, u_2+w_j, u_3+w_j\}$ must be in $\{2v_1+v_3, 2v_2+v_3, v_1+v_2+v_3\}\subsetneq G_j$. 
Since $v_1-v_2\notin\{\pm(u_1-u_2), \pm(u_1-u_3), \pm(u_2-u_3)\}$, these two vertices must be $2v_1+v_3$ and $2v_2+v_3$. By re-indexing the $u_i$'s if needed, we may assume without loss of generality that  $2v_1+v_3=u_1+w_j$ and $2v_2+v_3=u_2+w_j$. In particular, $w_1=w_2$. Since the third vertex of $D_j$ is $p_j=u_3+w_j$, we have $p_1=u_3+w_1=u_3+w_2=p_2$, but $p_1=v_2+2v_3$ and $p_2=v_1+2v_3$, so this contradicts the fact that $v_1$ and $v_2$ are distinct. Thus, at least one of $G_1$ and $G_2$ does not contain a shift of $\circnamezero$,
% so the image of $\theta$ is not cancellative in $\base[x_1,\ldots,x_n]/\Bend(\trop I)$.
completing the proof in this case.

By symmetry in the terms of $\theta=b_1\chi^{v_1}+b_2\chi^{v_2}+b_3\chi^{v_3}$, this corollary is also true any time $v_i-v_j\notin\{\pm(u_1-u_2), \pm(u_1-u_3), \pm(u_2-u_3)\}$ for some $i\neq j$.

So now we just need to consider the case, where 
$$\{\pm(v_1-v_2), \pm(v_1-v_3), \pm(v_2-v_3)\}=\{\pm(u_1-u_2), \pm(u_1-u_3), \pm(u_2-u_3)\}.$$ 
In fact, in this case $\{v_2-v_1, v_3-v_2, v_1-v_3\}$ must be one of $\pm\{u_2-u_1, u_3-u_1, u_1-u_3\}$ because $(v_2-v_1)+(v_3-v_2)+(v_1-v_3)=0$. By re-indexing the $u_i$'s, we may assume without loss of generality that $v_2-v_1=u_2-u_1$, and so $\{v_3-v_2,v_1-v_3\}=\{u_3-u_2, u_1-u_3\}$. We consider two sub-cases as to whether $v_3-v_2=u_3-u_2$ and $v_1-v_3=u_1-u_3$ or $v_3-v_2=u_1-u_3$ and $v_1-v_3=u_3-u_2$.

Sub-case 1: $v_3-v_2=u_3-u_2$ and $v_1-v_3=u_1-u_3$. Letting $w=v_1-u_1$ we have $v_1=u_1+w$, $v_2=u_2-u_1+v_1=u_2+w$ and $v_3=u_3-u_2+v_2=u_3+w$. Then $G_3
%=\{2v_1+v_3, v_1+2v_2, v_2+2v_3, v_1+v_2+v_3\}
=\{2u_1+u_3, u_1+2u_2, u_2+2u_3, u_1+u_2+u_3\}+3w
$
which, by Lemma~\ref{lemma:G3Works}, does not contain any shift of any of $C_i$ for $i=0,\ldots,7$.

Sub-case 2: $v_3-v_2=u_1-u_3$ and $v_1-v_3=u_3-u_2$. Letting $w=v_1-u_1$ we have $v_1=u_1+w$, $v_2=u_2-u_1+v_1=u_2+w$ and $v_3=u_1-u_3+v_2=u_1+u_2-u_3+w$.
Then 
\begin{align*}
G_3
%=\{2v_1+v_3, v_1+2v_2, v_2+2v_3, v_1+v_2+v_3\}
&=\{2u_1+(u_1+u_2-u_3), u_1+2u_2, u_2+2(u_1+u_2-u_3), u_1+u_2+(u_1+u_2-u_3)\}+3w\\
&=\{3u_1+u_2-u_3, u_1+2u_2, 2u_1+3u_2-2u_3 , 2u_1+2u_2-u_3\}+3w\\
&=\{2u_1+u_3, u_2+2u_3, u_1+2u_2, u_1+u_2+u_3\}+(u_1+u_2-2u_3)+3w,
\end{align*}
and so Lemma~\ref{lemma:G3Works} tells us that $G_3$ does not contain any shift of any of $C_i$ for $i=0,\ldots,7$.
\end{proof}

\bibliographystyle{alpha}

\end{document}